\documentclass{amsart}

\usepackage{amssymb, amsfonts, amsmath, amsthm}
\usepackage{mathrsfs}
\usepackage{hyperref}
\usepackage{graphicx}
\usepackage{caption}
\usepackage{subcaption}
\usepackage{wrapfig}
\usepackage[margin=1.00 in]{geometry}
\usepackage{enumitem, multicol}
\usepackage{tikz}
\usepackage{nicefrac, bigints}
\usepackage{harpoon}
\usepackage{chngcntr}
\usepackage{apptools}

\theoremstyle{plain}
\newtheorem{thm}{Theorem}[section]
\newtheorem{lemma}[thm]{Lemma}

\newtheorem{coro}[thm]{Corollary}
\newtheorem{prop}[thm]{Proposition}
\newtheorem{fact}[thm]{Fact}
\newtheorem{conj}[thm]{Conjecture}
\newtheorem*{ques}{Question}
\newtheorem*{heur}{Heuristic}

\theoremstyle{definition}
\newtheorem{defn}[thm]{Definition}
\newtheorem{constr}[thm]{Construction}

\newtheorem{rmk}[thm]{Remark}

\newcommand{\ZZ}{\mathbb{Z}}
\newcommand{\RR}{\mathbb{R}}

\newcommand{\CC}{\mathbb{C}}

\newcommand{\ith}{^\text{th}}

\DeclareMathOperator{\Mod}{Mod}

\DeclareMathOperator{\spin}{{spin}}
\DeclareMathOperator{\even}{{even}}
\DeclareMathOperator{\odd}{{odd}}
\DeclareMathOperator{\hyp}{{hyp}}
\DeclareMathOperator{\GA}{\Gamma_\mathsf{A}}
\DeclareMathOperator{\SL}{SL}
\DeclareMathOperator{\Sp}{Sp}
\DeclareMathOperator{\SMod}{SMod}
\DeclareMathOperator{\Arf}{Arf}

\newcommand{\T}{\mathcal{T}}
\newcommand{\M}{\mathcal{M}}
\newcommand{\A}{\mathsf{A}}
\newcommand{\HT}{\mathcal{HT}}
\newcommand{\HM}{\mathcal{HM}}
\newcommand{\QT}{\mathcal{QT}}
\newcommand{\QM}{\mathcal{QM}}
\newcommand{\sing}{\underline{\kappa}}
\newcommand{\Ztwo}{\mathbb{Z}_2}
\newcommand{\Zr}{\mathbb{Z}_r}
\newcommand{\LL}{\mathcal{L}}

\begin{document}


\title{Connected components of strata of Abelian differentials over Teichm{\"u}ller space}

\author{Aaron Calderon}

\begin{abstract}
This paper describes connected components of the strata of
holomorphic abelian differentials on marked Riemann surfaces with prescribed degrees of zeros.
Unlike the case for unmarked Riemann surfaces,
we find there can be many connected components,
distinguished by roots of the cotangent bundle of the surface.
In the course of our investigation we also characterize the images of the fundamental groups of strata inside of the mapping class group.
The main techniques of proof are mod $r$ winding numbers and a mapping class group--theoretic analogue of the Euclidean algorithm.
\end{abstract}


\maketitle

\setcounter{tocdepth}{1}
\tableofcontents

\section{Introduction}

The Hodge bundle $\HM_g$ of holomorphic abelian differentials over the moduli space $\M_g$ of genus $g$ Riemann surfaces is a fundamental object of study in many diverse fields of mathematics. This bundle can be partitioned into a collection of disjoint {\em strata}, suborbifolds (in fact, subvarieties) which are distinguished by the number and degree of the zeros of the differentials in the stratum. For any integer partition $\sing = (k_1, \ldots, k_n)$ of $2g-2$, we write $\HM(\sing)$ to denote the stratum of abelian differentials on genus $g$ Riemann surfaces which have exactly $n$ zeros of degrees $k_1, \ldots, k_n$.

An abelian differential $dz$ defines a flat cone metric $|dz|^2$ on the surface, and so a stratum $\HM(\sing)$ may be identified with the moduli space of finite--area translation surfaces with cone points of angle\\ $2(k_1+1) \pi, \ldots, 2(k_n+1)\pi$.

By pioneering work of Masur \cite{Masur_IETsMF} and Veech \cite{Veech_IETs}, the Teichm{\"u}ller geodesic flow on $\HM_g$ acts ergodically on each connected component of a stratum with respect to a Lebesgue--class measure. More generally, strata are some of the simplest examples of orbit closures for the $\SL_2(\RR)$ action on $\HM_g$. For an overview of these and related topics, see, e.g., \cite{Wright_Survey1} or \cite{Zorich_Survey}.

In \cite{KZ_strata}, Kontsevich and Zorich classified the connected components of these strata. They proved that each stratum $\HM(\sing)$ has at most $3$ components, distinguished by hyperellipticity and the parity of the induced spin structure, an algebro-geometric condition relating to square roots of the canonical bundle (cotangent bundle) over a given Riemann surface (see Theorem \ref{thm:KZ_class}). 

This paper addresses a similar question, posed now over the {\em Teichm{\"u}ller} space.
Recall that the Teichm{\"u}ller space $\T_g$ is the space of {\em marked} genus $g$ Riemann surfaces (up to isotopy).
The change--of--marking action of the mapping class group $\Mod(S)$ on $\T_g$ demonstrates $\T_g$ as the (orbifold) universal cover of $\M_g$,
and there is similarly a Hodge bundle $\HT_g$ over Teichm{\"u}ller space
classifying the holomorphic abelian differentials on marked Riemann surfaces (equivalently, marked translation surfaces of finite area).

Just as over moduli space, the Hodge bundle over Teichm{\"u}ller space is stratified by number and degree of zeros. For any integer partition $\sing = (k_1, \ldots, k_n)$ of $2g-2$, we write $\HT(\sing)$ to denote the stratum of abelian differentials on {\em marked} genus $g$ Riemann surfaces which have exactly $n$ zeros of degrees $k_1, \ldots, k_n$.

Let $r:= \gcd(k_1, \ldots, k_n)$. When $r \in \{2g-2, g-1\}$ (and $g \ge 3$), there are infinitely many hyperelliptic components of $\HT(\sing)$, corresponding to the infinitely many different hyperelliptic involutions of the surface $S$ (see Corollary \ref{coro:comp_hyp}).

Our main theorem deals with the remaining cases by relating the connected components of $\HT(\sing)$ to the set of {\em $r$--spin structures}, $r\ith$ roots of the canonical bundle over a given marked Riemann surface which can equivalently be thought of as mod $r$ winding number functions (see \S\S\ref{sec:rspin}, \ref{sec:diffspin}).

Our results only apply to surfaces of high enough genus; in order to specify exactly which, we must use the following auxiliary function:
\[ g(r) = \left\{
\begin{array}{ll}
13 & r=4 \\
21 & r=8 \\
5 & \text{otherwise}
\end{array} \right. \]

\begin{thm}\label{thm:components}
Suppose that $\sing=(k_1, \ldots, k_n)$ is a partition of $2g-2$ such that $g \ge g(r)$, where
\[r=\gcd(k_1, \ldots, k_n) \notin \{g-1, 2g-2\}.\]
Then the stratum $\HT(\sing)$ has finitely many components. 
\begin{enumerate}
\item If $r$ is odd, then there are exactly $r^{2g}$ components, distinguished by their induced $r$--spin structure.
\item If $r$ is even, then there are at least $r^{2g}$ components, of which at least
\[(r/2)^{2g} \left( 2^{g-1}(2^g +1) \right)\]
have even parity and at least
\[(r/2)^{2g} \left( 2^{g-1}(2^g - 1) \right)\]
have odd.
\end{enumerate}
\end{thm}

The connected components of strata over Teichm{\"u}ller space are intimately connected to the fundamental groups of strata over moduli space. Kontsevich has conjectured \cite{KZ_strings} that every connected component of a stratum is a classifying space for some sort of mapping class group, but little progress has been made either way in this regard.

Our second main theorem deals with certain representations of these fundamental groups inside of the mapping class group.
In particular, suppose that $\Omega$ is a connected component of some stratum $\HM(\sing)$. The forgetful map $p: \HM_g \rightarrow \M_g$ induces a map of orbifold fundamental groups
\[ p_* : \pi_1^{\text{orb}} (\Omega) \rightarrow \pi_1^{\text{orb}} \left(\M_g\right) = \Mod(S)\]
whose image is called the {\em geometric monodromy group} $\mathcal{G}(\Omega)$ of $\Omega$.
\footnote{Technically, this group is only well-defined after choice of basepoint $(X, \omega)$ (where $X$ is a Riemann surface and $\omega$ an abelian differential on $X$) and an identification of $X$ and $S$, i.e., a marking. We discuss this further in \S\ref{sec:wn_monodromy}, but for the purposes of the introduction one may simply think of the geometric monodromy as a subgroup up to conjugation.}

Since the number of connected components of $\HT(\sing)$ which lie over $\Omega$ is the same as the index of $\mathcal{G}(\Omega)$ inside $\Mod(S)$ (see \S\ref{sec:wn_monodromy}), Theorem \ref{thm:components} is essentially equivalent to the following:

\begin{thm}\label{thm:monodromy}
Suppose that $\sing=(k_1, \ldots, k_n)$ is a partition of $2g-2$ such that $g \ge g(r)$, where
\[r=\gcd(k_1, \ldots, k_n) \notin \{g-1, 2g-2\}.\]
If $\Omega$ is a connected component of $\HM(\sing)$, then the geometric monodromy group $\mathcal{G}(\Omega)$ is a finite--index subgroup of $\Mod(S)$.
\begin{enumerate}
\item If $r$ is odd, then $\mathcal{G}(\Omega)$ is the stabilizer inside the mapping class group of an $r$--spin structure.
\item If $r$ is even, then $\mathcal{G}(\Omega)$ is a finite--index subgroup of the stabilizer of an $r$--spin structure.
\end{enumerate}
\end{thm}

The high--genus and finite--index qualifications for even $r$ are not essential, but are rather relics of the mapping class group--theoretic methods which we use to investigate the geometric monodromy groups $\mathcal{G}(\Omega)$. Moreover, the strata $\HM(\sing)$ for $\sing = (2g-2)$ or $(g-1, g-1)$ have non-hyperelliptic components, the geometric monodromy groups of which remain unclassified (see the discussion in \S\ref{sec:onwards}).

\begin{conj}\label{conj:comp}
Let $g \ge 4$ and let $\sing=(k_1, \ldots, k_n)$ be a partition of $2g-2$ with $r = \gcd(k_1, \ldots, k_n)$. The non--hyperelliptic connected components of $\HT(\sing)$ are in one--to--one correspondence with the set of $r$--spin structures on $S$. In particular, there are always exactly $r^{2g}$ non-hyperelliptic components of $\HT(\sing)$.

Equivalently, if $\Omega$ is a non--hyperelliptic connected component of $\HM(\sing)$, then its geometric monodromy group is the stabilizer of an $r$--spin structure.
\end{conj}

{\em Update:} Nick Salter and the author have proven this conjecture for all $r$ and all $g \ge 5$ \cite{CS_strata}.

\subsection{Context: higher spin structures}

While square roots of the canonical bundle $K_X$ over a Riemann surface $X$ (also known as theta characteristics or (classical) spin structures) have been studied since the times of Riemann, its higher roots are a relatively recent addition to the literature. 

The fundamental work of Sipe \cite{Sipe_roots} relates $r\ith$ roots of $K_X$ to the cohomology of the unit tangent bundle $T_0X$ (see \S \ref{sec:def_rspin}), and in that paper and in a sequel \cite{Sipe_quotients} she also describes the action of the mapping class group on the set of $r$--spin structures.
Later, Trapp recovered the same action in his construction of novel representations of the mapping class group acting on the homology of the unit tangent bundle \cite{Trapp_wn}.

Higher spin structures were recently utilized by Salter in the course of his investigations into the geometric monodromy groups of both families of smooth plane curves of fixed degree \cite{Salter_planecurves} and of families of smooth curves in a complete linear system on a smooth toric surface \cite{Salter_monodromy}. In the latter work, he also analyzes the $\Mod(S)$ stabilizer of a fixed $r$--spin structure and gives an explicit criterion for collections of Dehn twists to generate the subgroup (Theorem \ref{thm:Salter_generation}). We make extended use of this result in \S\ref{sec:gen_monodromy}.

Though it seems higher spin structures had been largely forgotten in the Teichm{\"u}ller theory literature until quite recently, they
are routine objects of inquiry
for complex algebraic geometers and topological string theorists. One need only perform a cursory web search to find a wealth of papers relating to moduli of Riemann surfaces with $r$--spin structures and compactifications thereof, e.g. \cite{Jarvis_geometry}, \cite{AbrJar_twistspin}, and \cite{CCC_rspin}. We mention in particular work of Polischuk on moduli of {\em effective} $r$--spin structures, that is, $r$--spin structures which admit holomorphic sections. The $r\ith$ power of one such section is an abelian differential, and in particular the moduli spaces of effective $r$--spin structures are in one--to--one correspondence with strata of abelian differentials \cite[Theorem 1.2]{Poli_rspin}.

Much of the recent activity regarding higher spin curves has focused on their role in a higher spin formulation of Witten's conjecture \cite{Witten_rspin}, which relates intersection theory on the moduli space of stable $r$--spin curves with integrable hierarchies. This conjecture was refined and subsequently proved in certain special cases in \cite{JKV_KdV} and in all generality in \cite{FSZ_rspinWitten}.

Intersection theory over the moduli space of stable curves is known to relate to both the Weil--Petersson volume of moduli space \cite{Mirz_Witten} and the Masur--Veech volume of the principal stratum of quadratic differentials (that is, the stratum with all simple zeros) \cite{DGZZ}. As Masur--Veech volumes are notoriously difficult to compute, it would be interesting to know if intersection theory over the moduli space of stable $r$--spin curves can be related to the volumes of non-principal strata in a similar fashion.

\subsection{Context: connected components}\label{sec:comp_rev}
As stated above, Kontsevich and Zorich classified the connected components of strata over the moduli space of holomorphic abelian differentials \cite{KZ_strata} by hyperellipticity and parity of spin structure. In the infinite--area case, Boissy \cite{Boissy_components} proved that each stratum of meromorphic abelian differentials over moduli space also has at most $3$ components (except when $g=1$), distinguished by the same invariants. Lanneau completed the classification of the connected components of strata of quadratic differentials over moduli space in \cite{Lanneau_strata1} and \cite{Lanneau_strata2}, with a slight correction by Chen and M{\"o}ller when $g=4$ \cite{ChenMoller_QDinlowg}.

Except for the last--named result, all of the above papers rely on a classification of the connected components of the {\em minimal stratum} $\HM(2g-2)$ (or for quadratic differentials, the stratum with a single zero of degree $4g-4$). By ``colliding zeros,'' one may degenerate any stratum to the minimal one, and therefore the number of connected components of a general stratum over moduli space is at most the number of connected components of the minimal stratum. Over Teichm{\"u}ller space, this approach fails miserably, for there are infinitely many components of the minimal stratum $\HT(2g-2)$ (Corollary \ref{coro:comp_hyp}).

In her thesis \cite{Walker_thesis} and in \cite{Walker_components}, Walker used winding numbers and $r\ith$ roots of the square of the cotangent bundle to investigate the connected components of the Teichm{\"u}ller space of quadratic differentials, recovering in some special cases results which are analogous to ours. The characterization of connected components appearing in our main theorem is inspired by her work, and our argument in \S\ref{sec:invariance} is a generalization of her lower bound for the the number of connected components of strata. However, her use of $r\ith$ roots to construct upper bounds uses the connectivity of certain configuration spaces, a technique which requires many zeros of the same multiplicity and thus is insufficient for most of our cases.

\subsection{Context: monodromy of strata}\label{sec:monodromyreview}

While the fundamental groups of strata have remained mysterious outside of the hyperelliptic components and low genera \cite{LM_strata}, their monodromy representations (in both mapping class and symplectic groups) have been studied by multiple authors.

Let $\Omega$ be a connected component of a stratum of abelian or quadratic differentials over moduli space. By marking the zeros of any representative differential in $\Omega$, one may obtain a geometric monodromy representation of $\pi_1^\text{orb}(\Omega)$ not only into the mapping class group, but into the punctured mapping class group. We denote the resulting subgroup of $\Mod(S_{g,n})$ by $\mathcal{G}^\circ(\Omega)$. This representation gives one more information about $\pi_1^\text{orb}(\Omega)$ (since $\mathcal{G}(\Omega)$ is the image of $\mathcal{G}^\circ(\Omega)$ under the forgetful map), but is less related to the components of the stratum of differentials over Teichm{\"u}ller space which cover $\Omega$. 

In addition to her work on connected components of strata over Teichm{\"u}ller space, Walker also considered the groups $\mathcal{G}^\circ(\Omega)$ when $\Omega$ is a stratum of quadratic differentials over moduli space \cite{Walker_thesis}, \cite{Walker_groups}. In some very special cases (in particular, when one has many simple zeros), she proved that this group is the kernel of a certain map and gave an explicit generating set.

During the writing of this paper, Hamenst{\"a}dt released a preprint in which she computes $\mathcal{G}^\circ(\Omega)$ when $\Omega$ is a stratum of abelian differentials \cite{Hamen_strata}. We note that while her main result gives a set of generators for $\mathcal{G}^\circ(\Omega)$ (and hence for $\mathcal{G}(\Omega)$), it does not immediately characterize $\mathcal{G}(\Omega)$ as a subgroup of $\Mod(S)$. In a later draft, by applying the work of Salter, she is able to recover some cases of our main theorems \cite[Theorem 3]{Hamen_strata}.

The geometric monodromy of a component of a stratum can be realized more concretely as a monodromy group by building the corresponding surface bundle. To that end, if $\Omega$ is a component of $\HM(\sing)$ and $\widetilde{\Omega}$ is a component of $\HT(\sing)$ lying over $\Omega$, define $\widetilde{\mathcal{X}}$ to be the bundle over $\widetilde{\Omega}$ whose fiber at a marked abelian differential $(X, f, \omega)$ is simply the Riemann surface $X$. This bundle is trivial over $\widetilde{\Omega}$, but quotienting out by the diagonal action of the mapping class group yields a nontrivial surface bundle $\mathcal{X} \rightarrow \Omega$ whose monodromy group (of a generic fiber) is exactly $\mathcal{G}(\Omega)$.

By replacing each Riemann surface with its homology, one can similarly define an $H_1(X; \RR)$ bundle over $\widetilde{\Omega}$, which descends to a bundle unfortunately also sometimes referred to in the literature as the Hodge bundle over $\Omega$.
\footnote{Observe that with this nomenclature, the Hodge bundle is a bundle over a subvariety of the Hodge bundle! Moreover, it is common in the literature to use the term ``Hodge bundle'' to refer to the $\Mod(S)$ quotients of a number of different real or complex, relative or absolute, homology or cohomology bundles over $\widetilde{\Omega}$ \cite[Remark 4]{Matheus_whatis}.}
We will eschew this terminology, and will instead simply denote this bundle by $H_1\Omega$.

The natural $\SL_2(\RR)$ action on $H_1\Omega$ gives rise to the {\em Kontsevich--Zorich cocycle}, the Lyapunov exponents of which have been studied extensively (see, e.g., \cite{AV_spectra}, \cite{Forni_KZ}). Associated to this cocycle is its algebraic hull, the smallest algebraic group containing (a conjugate of) the cocycle, which has been exploited to great effect by Filip \cite{Filip_monodromy} and Eskin--Filip--Wright \cite{EFW_HullKZcocycle}. Since the Zariski closure of the monodromy of $H_1\Omega$ necessarily contains the algebraic hull, constraints on the monodromy place constraints on the hull.

Filip proved Zariski density in $\Sp(2g, \RR)$ of the monodromy of $H_1\Omega$ \cite[Corollary 1.7]{Filip_monodromy}
\footnote{In fact, he proved an analogous statement for the monodromy group of any affine invariant submanifold.}
and the full computation of the monodromy groups of $H_1\Omega$ was completed by Guti{\'e}rrez-Romo \cite[Corollary 1.2]{GR_RVgroups}.

By the construction of the bundles $\mathcal{X}$ and $H_1\Omega$ above and our discussion of their monodromies, one can see that the monodromy of $H_1\Omega$ is exactly 
\[\psi( \mathcal{G}(\Omega)) \le \Sp(2g, \ZZ),\]
where $\psi$ is the natural symplectic representation of $\Mod(S)$ via its action on homology. Using this fact, we can use Theorem \ref{thm:monodromy} to give a topological proof of the result of Guti{\'e}rrez-Romo.

\begin{coro}[c.f. Corollary 1.2 in \cite{GR_RVgroups}]\label{coro:symp_monodromy}
Suppose that $\sing=(k_1, \ldots, k_n)$ is a partition of $2g-2$ such that $g \ge g(r)$, where
\[r=\gcd(k_1, \ldots, k_n) \notin \{g-1, 2g-2\}.\]
Let $\Omega$ be a connected component of $\HM(\sing)$.
\begin{enumerate}
\item If $r$ is odd, then the monodromy group of $H_1\Omega$ is the entire symplectic group $\Sp(2g, \ZZ)$.
\item If $r$ is even, then the monodromy group of $H_1\Omega$ is the stabilizer in $\Sp(2g, \ZZ)$ of a quadratic form $q$ associated to the spin structure on the chosen basepoint (see \S\ref{sec:Arf}).
\end{enumerate}
\end{coro}

We note that Guti{\'e}rrez-Romo's result (combined with work of Avila, Matheus, and Yoccoz for the hyperelliptic case \cite{AMY_hyperRV}) also recovers the cases when $r \in \{2g-2, g-1\}$ (and for low genera). 
In addition, the original impetus for both \cite{GR_RVgroups} and \cite{AMY_hyperRV} was the computation not of monodromy representations, but rather the {\em Rauzy--Veech groups} of strata, which relate to a discrete version of the Kontsevich--Zorich cocycle and the combinatorial dynamics of the Teichm{\"u}ller geodesic flow.

It would be interesting to know how much of the geometric monodromy group can be recovered from the Teichm{\"u}ller  geodesic flow, perhaps via modular Rauzy--Veech groups (see \cite[Definition 2.3]{GR_RVgroups}).

\subsection{Outline of the paper}

In Section \ref{sec:background}, we recall some necessary background about abelian differentials and strata. We also use this section to collect results about the topology of the hyperelliptic connected components of $\HM(\sing)$ (Theorem \ref{thm:hypstrata_pi1}) and hyperelliptic mapping class groups (Theorem \ref{thm:BH1}). 
The latter theorem plays an important role in the calculations appearing in Appendix \ref{app:curves}.
Combining Theorems \ref{thm:hypstrata_pi1} and \ref{thm:BH1}, we derive the classification of hyperelliptic connected components of $\HT(\sing)$ (Corollary \ref{coro:comp_hyp}). 

In order to parametrize the non--hyperelliptic connected components of $\HT(\sing)$ by their induced $r$--spin structure, in \S\ref{sec:rspin} we recall Sipe's characterization of $r$--spin structures as elements of the cohomology of the unit tangent bundle and the action of the mapping class group on the set of these structures. We also record Salter's criterion (Theorem \ref{thm:Salter_generation}) for generating an {\em $r$--spin mapping class group} $\Mod(S)[\phi]$, the stabilizer of an $r$--spin structure $\phi$ under the mapping class group action. In particular, viewing $r$--spin structures as topological, instead of algebro-geometric, objects will allow us to compare $r$--spin structures on different (marked) Riemann surfaces.

Section \ref{sec:diffspin} contains one final interpretation of $r$--spin structures as mod $r$ winding numbers (Proposition \ref{prop:wn_is_rspin}) and uses this characterization to relate them to the flat geometry of surfaces in $\HT(\sing)$. From this equivalence, it is easy to show that the $r$--spin structures on any two marked differentials in a component of $\HT(\sing)$ must be topologically equivalent (Proposition \ref{prop:wn_const_on_comp}). In particular, this demonstrates that there exist at least as many components of $\HT(\sing)$ as there are (topological equivalence classes of) $r$--spin structures (Theorem \ref{thm:lower_bd}).

The invariance of the $r$--spin structure therefore implies that the geometric monodromy group $\mathcal{G}(\Omega)$ of any connected component $\Omega$ of $\HM(\sing)$ must lie inside some $r$--spin mapping class group $\Mod(S)[\phi]$ (\S\ref{sec:wn_monodromy}). The remainder of the paper consists of using the action of the mapping class group on simple closed curves to show that $\mathcal{G}(\Omega)$ is all of $\Mod(S)[\phi]$ (or when $r$ is even, is of finite index).

In \S\ref{sec:prototypes}, we fix a system of curves $\mathsf{C}$ of combinatorial type compatible with $\Omega$ and use a standard construction to build an explicit (marked) abelian differential in $\Omega$ (Proposition \ref{prop:prototypes}). The core curves of the horizontal and vertical cylinders on this differential are exactly the curves of $\mathsf{C}$, and so by shearing these cylinders (see \S\ref{sec:gen_monodromy}) we are able to realize a subgroup $\Gamma(\mathsf{C}) \le \mathcal{G}(\Omega)$ generated by all of the Dehn twists in the curves of $\mathsf{C}$.

In some special cases, the collection of curves $\mathsf{C}$ is large enough that we are able to immediately apply Salter's theorem. In the case when $r$ is odd, the theorem says that $\Gamma(\mathsf{C}) =\Mod(S)[\phi]$, so we have that
\[\Gamma(\mathsf{C}) =\Mod(S)[\phi] \le \mathcal{G}(\Omega) \le \Mod(S)[\phi] \]
and in particular $\mathcal{G}(\Omega)=\Mod(S)[\phi]$. If $r$ is even, the theorem says that $\Gamma(\mathsf{C})$ is a finite index subgroup of $\Mod(S)[\phi]$, hence $\mathcal{G}(\Omega)$ must be as well, finishing the proof of Theorems \ref{thm:components} and \ref{thm:monodromy}.

However, for many strata the curves of $\mathsf{C}$ do not fulfill the hypotheses of Salter's theorem. To deal with the remaining possibilities, we show in Theorem \ref{thm:Euclidean_algorithm} that by we can ``complete'' the curve system $\mathsf{C}$ to the maximal one allowed by $r$, that is, to the curve system $\mathsf{C}'$ corresponding to the partition $(r, r, \ldots, r)$ of $2g-2$. 
More precisely, we show that $\Gamma(\mathsf{C}) = \Gamma(\mathsf{C}')$.

One of the most novel contributions of this work is the demonstration of the above equality. In order to prove it, we model the operations of standard arithmetic by Dehn twists on certain simple closed curves of $\mathsf{C}$ (see Appendix \ref{app:curves}) and then iteratively apply the Euclidean algorithm to reduce the partition $(k_1, \ldots, k_n)$ to the partition $(r, \ldots, r)$. From this procedure it follows that any Dehn twist in a curve of $\mathsf{C}'$ can be expressed as a product of Dehn twists in the curves of $\mathsf{C}$.

The completed curve system $\mathsf{C}'$ is then large enough to apply Salter's theorem, so we can conclude that
\[\Gamma(\mathsf{C}) = \Gamma(\mathsf{C}')=\Mod(S)[\phi]\]
when $r$ is odd, and is a finite index subgroup of $\Mod(S)[\phi]$ when $r$ is even,
finishing the proof of Theorems \ref{thm:components} and \ref{thm:monodromy}.

We conclude in \S\ref{sec:onwards} by outlining some natural questions that arise in the course of the proof, as well as possible directions for further research.

\subsection{Acknowledgements}

The author is grateful to his advisor, Yair Minsky, for encouraging him to pursue this question and for his continued support and guidance, as well as for comments on earlier drafts of this paper.

Parts of this work were completed at the summer school ``Teichm{\"u}ller dynamics, mapping class groups and applications'' at the Institut Fourier in June 2018 and at the ``Workshop on the dynamics and moduli of translation surfaces'' at the Fields Institute in October 2018. The author is indebted to these venues for their hospitality and the other participants for stimulating discussions. He would particularly like to thank Rodolfo Gutierr{\'e}z-Romo and Sam Grushevsky for helpful conversations about this work.

The author would also like to thank Ursula Hamenst{\"a}dt for helpful comments, as well as the anonymous referee for comments which improved the readability of this paper.

The author was supported in part by NSF grant [DGE-1122492].

\section{Preliminaries}\label{sec:background}

Before proceeding with the proof we will recall some foundational information, which also serves the purpose of allowing us to establish our notation. All of this material can be found in greater detail in the flat surfaces literature, see e.g. \cite{Wright_Survey1}, \cite{Zorich_Survey}. In \S\ref{sec:hyperelliptic}, we record the relationship between hyperelliptic abelian differentials and hyperelliptic mapping class groups, and use this to show that there are infinitely many hyperelliptic components of $\HT(2g-2)$ and $\HT(g-1, g-1)$ (Corollary \ref{coro:comp_hyp}).

Let $S=S_{g,n}$ denote a (smooth, orientable) surface of genus $g$ with $n$ marked points. The moduli space $\M_{g,n}$ of $S$ is the space of complex (equivalently, conformal or hyperbolic) structures on $S$. The moduli space is generally not a manifold but an orbifold, whose orbifold universal cover is the Teichm{\"u}ller space $\T_g$ of (equivalence classes of) marked Riemann surfaces. A point in $\T_g$ is an (equivalence class of) pairs $(X, f)$, where $X$ is a Riemann surface and $f: S \rightarrow X$ is a diffeomorphism (a marking), and where two marked Riemann surfaces $(X, f)$ and $(Y, h)$ are equivalent if the map 
$h \circ f^{-1}: X \rightarrow Y$
is isotopic to a biholomorphism.

The mapping class group $\Mod(S)$ may be defined as
\[\Mod(S)=\pi_0 \left( \text{Diff}^+(S)\right),\]
where $\text{Diff}^+(S)$ is the space of orientation--preserving diffeomorphisms of $S$. If $S$ has punctures and/or boundary components, we allow mapping classes to permute the punctures but insist that they fix the boundary pointwise.

The mapping class group acts on Teichm{\"u}ller space by precomposition (by inverses) with the marking, so that for any $g \in \Mod(S)$, 
\[g \cdot (X, f) = (X, f g^{-1}).\]
A specific family of mapping classes that we will use frequently are {\em Dehn twists}: given any simple closed curve $c$ on $S$, the (left--handed) Dehn twist $T(c)$ in $c$ is realized by cutting the surface along $c$ and regluing the resulting boundary components with a full leftward twist. It is a standard fact that $\Mod(S)$ is generated by a finite collection of Dehn twists.

For the rest of the paper, except when otherwise stated, all surfaces will be closed and without boundary.

A {\em holomorphic abelian differential} $\omega$ on a Riemann surface $X$ is a holomorphic $1$--form, equivalently, a holomorphic section of $K_X$, while a {\em quadratic differential} is a section 
\[q: X \rightarrow K_X^{\otimes 2}.\]
For the rest of the paper, we will assume that all abelian differentials are holomorphic and all quadratic differentials are meromorphic with at worst simple poles.

Around every point of $X$, an abelian (quadratic) differential defines canonical coordinates in which the differential takes the form $z^k dz$ for some $k \ge 0$ (respectively, $z^k dz^2$ for $k \ge -1$). By pulling back the flat metric on $\CC$ along these coordinates, both abelian and quadratic differentials induce flat cone metrics on $X$ with cone angles of $2(k+1)\pi$ at each point (respectively, $(k+2)\pi$). A {\em cylinder} on a flat surface $(X, \omega)$ or $(X, q)$ is an embedded flat cylinder which does not contain any singularities in its interior.

The space of all pairs $(X, \omega)$ where $X$ is a Riemann surface and $\omega$ is a holomorphic abelian differential is naturally a vector bundle over $\M_g$, called the {\em Hodge bundle} $\HM_g$. For a given partition $\sing = (k_1, \ldots, k_n)$ of $2g-2$ by positive integers, we denote the stratum of $\HM_g$ of differentials with exactly $n$ zeros of orders $k_1, \ldots, k_n$ by $\HM(\sing)$. Similarly, there is a Hodge bundle $\HT_g$ over the Teichm{\"u}ller space $\T_g$ and we denote its strata by $\HT(\sing)$. Points in $\HT_g$ correspond to triples $(X, f, \omega)$ where $X$ is a Riemann surface, $f: S \rightarrow X$ is a marking, and $\omega$ is a holomorphic abelian differential on $X$.

Let $(X, f, \omega) \in \HT(\sing)$ and fix a basis $\{\gamma_1, \ldots, \gamma_d\}$ for the homology of $X$ relative to the zeros of $\omega$. One can transport each $\gamma_i$ to nearby $(X',f', \omega')$ in $\HT(\sing)$, yielding {\em period coordinates} on the stratum, local coordinates given by
\[ \begin{array}{ccc}
\HT(\sing) & \rightarrow & \CC^{d} \\
(X, \omega) & \mapsto & \displaystyle \left( \int_{\gamma_1} \omega , \ldots, \int_{\gamma_d} \omega \right).
\end{array}\]
which demonstrate $\HT(\sing)$ as a complex manifold of dimension $d = 2g+n-1$. Quotienting out by the $\Mod(S)$ action, these coordinates descend to coordinates on $\HM(\sing)$, which is a (possibly disconnected) complex orbifold of the same dimension.

The orbifold nature of $\HM(\sing)$ can be observed at differentials $(X, \omega)$ which have extra symmetries (since they project to orbifold points in $\M_g$).
A prominent example occurs when a differential is {\em hyperelliptic}, that is, preserved under some involution of $X$ which acts by $-1$ on homology.
In this case, $\omega$ is obtained by pulling back a (necessarily meromorphic) quadratic differential $q$ on the Riemann sphere along a (branched) covering map whose branch locus is contained in the singularities of $q$.

A stratum $\HM(\sing)$ is not necessarily connected, but the work of Kontsevich and Zorich classifies its connected components. Before we state their theorem, we must record one more definition.

 Suppose that $(X, \omega) \in \HM(\sing)$; then $\omega$ defines a divisor 
\[(\omega)=\sum_{i=1}^n k_i p_i\]
on $X$, where $p_i \in X$ is the point at which $\omega$ has a zero of order $k_i$. When all $k_i$ are even, the divisor $(\omega)/2$ is equivalent to a section of some line bundle $\LL$ such that 
$\LL^{\otimes 2} = K_X$.

\begin{defn}\label{def:spinstr}
Suppose that $(X, \omega) \in \HM(\sing)$, where $\gcd(\sing)$ is even. The line bundle $\LL$ defined above is called the {\em spin structure} associated to $(X, \omega)$.

The {\em parity} of $\LL$ is $h^0(X, \LL) \pmod 2$, the dimension mod 2 of the space of holomorphic sections of $\LL \rightarrow X$.
\end{defn}

\begin{thm}[Theorem 1 of \cite{KZ_strata}]\label{thm:KZ_class}
If $g \ge 4$, then any stratum of abelian differentials over moduli space has at most three connected components:
\begin{itemize}
\item If $\sing = (2g-2)$ or $(g-1, g-1)$ then there is one component $\HM(\sing)^{\textnormal{hyp}}$ consisting entirely of hyperelliptic differentials.
\item If $\gcd(\sing)$ is even then there are two non-hyperelliptic components of $\HM(\sing)$, distinguished by the parity of their induced spin structure.
\item If $\gcd(\sing)$ is odd, there is one non-hyperelliptic component of $\HM(\sing)$.
\end{itemize}
\end{thm}

For uniformity of notation, we will always write $\HM(\sing)^{\spin}$ to denote a component of $\HM(\sing)$ with specified parity of spin structure, even when $\gcd(\sing)$ is odd. In that case, the $\spin$ term will be understood to be vestigial, as such abelian differentials do not determine ($2$--)spin structures. Similar naming conventions will be adopted throughout the paper.

\subsection{Hyperelliptic components and Birman--Hilden theory}\label{sec:hyperelliptic}

In the case when $\Omega$ is the hyperelliptic component of either $\HM(2g-2)$ or $\HM(g-1, g-1)$, its topology is much more tractable. In Theorem \ref{thm:hypstrata_pi1}, we record the topological types of these strata as quotients of configuration spaces.

We then discuss the theory of Birman and Hilden relating hyperelliptic mapping class groups to braid groups (Theorem \ref{thm:BH1}) and explain how to use this theory to classify the hyperelliptic connected components of $\HT(\sing)$ (Corollary \ref{coro:comp_hyp}). While Corollary \ref{coro:comp_hyp} is a consequence of existing statements in the literature and is certainly known to experts, we include a proof of it for completeness and to put our results into context.

Recall that a hyperelliptic differential $(X, \omega) \in \Omega$ is obtained by pulling back an integrable quadratic differential $q$ on $\widehat{\CC}$ via a branched cover $X \rightarrow \widehat{\CC}$. 

\begin{thm}[Folklore, see \cite{LM_strata}]\label{thm:hypstrata_pi1}
The strata $\HM(2g-2)^{\hyp}$ and $\HM(g-1, g-1)^{\hyp}$ are isomorphic to quotients of configuration spaces of points on the Riemann sphere by the action of the group of $(2g+1)^\text{st}$, respectively $(2g+2)^\text{nd}$, roots of unity.

In particular, this implies that $\HM(2g-2)^{\hyp}$ and $\HM(g-1, g-1)^{\hyp}$ are orbifold classifying spaces for finite extensions of the corresponding braid groups.
\end{thm}

We outline the geometric intuition of this theorem below, and direct the curious reader to \cite[\S1.4]{LM_strata} as well as \cite[\S4.2]{AMY_hyperRV} for a dynamical perspective.

\begin{proof}[Sketch of Proof]
Let $(X, \omega)$ be a hyperelliptic abelian differential, coming from a quadratic differential $q$ on $\widehat{\CC}$. Because $(X, \omega)$ is completely determined by the zeros of $q$, one may take the configuration of the singularities of $q$ as moduli for the space of hyperelliptic differentials in $\Omega$. We note that this can only be done locally: there is an action of the multiplicative group $\CC^\times$ on this configuration space, and its respective action on the universal hyperelliptic curve over the configuration space has nontrivial kernel. In particular, we note that the action of $-1 \in \CC^\times$ is the hyperelliptic involution.

To see how many singularities $q$ has, we consider the action of the hyperelliptic involution $\iota$. Suppose first that $\Omega = \HM(2g-2)^{\hyp}$; then since the zero of $\omega$ must necessarily be fixed under $\iota$, $q$ must have a zero of order $2g-3$. Therefore, by the Poincar{\'e}--Hopf theorem, it also has $2g+1$ simple poles.

Similarly, if $\Omega = \HM(g-1, g-1)^{\hyp}$ then since $\iota$ necessarily interchanges the two zeros,
\footnote{This follows because the underlying surface is isomorphic to a plane curve of the form $w^2 = \prod_{i=1}^{2g-2} (z - z_i)$ and the hyperelliptic involution interchanges the two points at infinity.} 
the differentials in $\Omega$ are obtained from a quadratic differential $q$ on $\widehat{\CC}$ with a single zero of order $2g-2$ and $2g+2$ simple poles.
\end{proof}

In order to relate this result to the geometric monodromy groups of the strata, we will appeal to the work of Birman and Hilden on symmetric mapping class groups.

\begin{defn}
If $\iota$ is some hyperelliptic involution of a surface $S$ (for the moment, closed), then the {\em symmetric mapping class group} $\SMod(S)$ (with respect to $\iota$) is the centralizer of $\iota$ in $\Mod(S)$.
\end{defn}

The theory of Birman and Hilden (developed over a series of papers in the 1970s, see the survey \cite{MargWin_BHSurvey} or \cite[\S9.4]{FarbMarg}) relates $\SMod(S)$ to the mapping class group of the quotient $S / \iota$.

By the Riemann--Hurwitz formula, the quotient $\Sigma = S / \iota$ is a sphere with $2g+2$ branch points, and so its mapping class group $\Mod(\Sigma)$ is just the mapping class group of a $(2g+2)$--times punctured sphere, which is a $\Ztwo$ quotient of the spherical braid group on $2g+2$ strands \cite[pg. 245]{FarbMarg}. Suppose that $\alpha$ is an arc on $\Sigma$ connecting branch points $b_1$ and $b_2$; then the half--twist $H_\alpha$ on $\alpha$ interchanges $b_1$ and $b_2$ by a clockwise twist in a neighborhood of $\alpha$. If $c$ is a curve on $S$ whose quotient is $\alpha$, one may observe that $H_\alpha$ lifts to the Dehn twist on $c$. See Figure \ref{fig:twist_lift}.

\begin{figure}[ht]
\centering
\includegraphics[scale=.8]{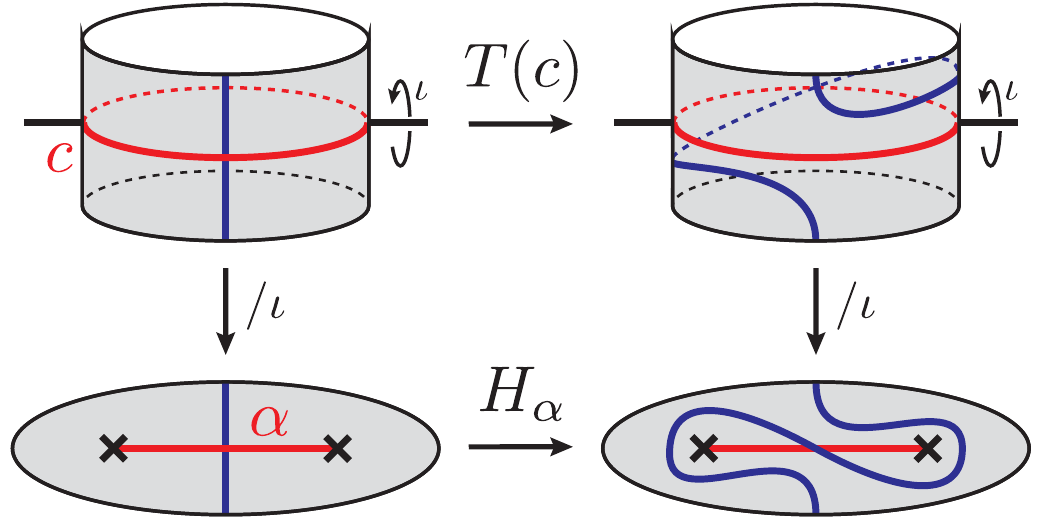}
\caption{Lifting a half--twist $H_\alpha$ to a Dehn twist $T(c)$.}
\label{fig:twist_lift}
\end{figure}

In this case, the Birman--Hilden theory states that

\begin{thm}[Birman--Hilden]\label{thm:BH1}
Let $\iota$ be a hyperelliptic involution of a closed surface $S$ and $\Sigma = S / \iota$. Then
\[\SMod(S) / \langle \iota \rangle \cong \Mod(\Sigma).\]
\end{thm}

One may perform a similar construction when the surface $S$ has punctures. Suppose first that $S$ has a unique puncture fixed by $\iota$, so that $\Sigma$ has $2g+1$ branch points and a unique puncture. Then the appropriate mapping class group $\Mod(\Sigma)$ is the subgroup of $\Mod(S_{0, 2g+2})$ which preserves the puncture but is allowed to interchange the branch points. When $S$ has two punctures which are interchanged by $\iota$, then $\Sigma$ has $2g+2$ branch points and a unique puncture and $\Mod(\Sigma)$ is defined similarly. In both of these cases, one has the same conclusion as in Theorem \ref{thm:BH1}, namely, that
\begin{equation}\label{eqn:BH_punc}
\SMod(S) / \langle \iota \rangle \cong \Mod(\Sigma).
\end{equation}

Finally, as it will play a large role in Appendix \ref{app:curves}, we also consider the case when $S$ has no punctures but two boundary components which are interchanged by $\iota$. In this case, the quotient $\Sigma$ again has $2g+2$ branch points but has a single boundary component, so $\Mod(\Sigma)$ is exactly the braid group $B_{2g+2}$ on $2g+2$ strands. Half--twists still lift to Dehn twists, but now the hyperelliptic involution $\iota$ is {\em not} a mapping class of the surface $S$ since it interchanges the boundary components. Therefore, one has that
\begin{equation}\label{eqn:BH2}
\SMod(S) \cong \Mod(\Sigma) \cong B_{2g+2}.
\end{equation}
One may of course perform similar constructions for surfaces with more punctures or boundary components, but the restrictions on which points may be interchanged become more involved.

Combining Theorems \ref{thm:hypstrata_pi1} and \ref{thm:BH1}, we arrive at a classification of the hyperelliptic components of $\HT(2g-2)$ and $\HT(g-1, g-1)$.

\begin{coro}\label{coro:comp_hyp}
For any $g \ge 3$, the strata $\HT(2g-2)$ and $\HT(g-1, g-1)$ each have infinitely many hyperelliptic connected components.
\end{coro}
\begin{proof}
Suppose that $\Omega$ is the hyperelliptic component of either $\HM(2g-2)$ or $\HM(g-1, g-1)$; then by Theorem \ref{thm:hypstrata_pi1} its fundamental group is a finite extension of a spherical braid group. Therefore, its punctured geometric monodromy group $\mathcal{G}^\circ(\Omega)$ must be
\begin{equation}\label{eqn:hypmon}
\mathcal{G}^\circ(\Omega) \cong \Mod(\Sigma) \times \Ztwo \cong \SMod(S_{g, n})
\end{equation}
where $n$ is the number of zeros of a differential in $\Omega$, and the corresponding hyperelliptic involution $\iota$ either preserves the single zero (in the case $\sing = (2g-2)$) or interchanges the two zeros (when $\sing = (g-1, g-1)$). Note that the last isomorphism of \eqref{eqn:hypmon} is just \eqref{eqn:BH_punc}, the Birman--Hilden correspondence for the surface punctured at the zeros of the differential.

The hyperelliptic involution $\iota$ remains a hyperelliptic involution after forgetting the puncture(s), and so we see that the image of $\SMod(S_{g,n})$ under the forgetful map lies inside of a different (unpunctured) symmetric mapping class group $\SMod(S_g)$.
\footnote{This map is by no means an isomorphism. When $n=1$, the map $\SMod(S_{g,1}) \rightarrow \SMod(S_g)$ is injective but not surjective \cite[Theorem 3.1]{BrenMarg_hypTorBirman}. When $n=2$, the map $\SMod(S_{g,2}) \rightarrow \SMod(S_g)$ is surjective but not injective \cite[Theorem 3.2]{BrenMarg_hypTorBirman}.}
We may therefore conclude that
\[\mathcal{G}(\Omega) \le \SMod(S_g).\]
Now for $g \ge 3$ any symmetric mapping class group has infinite index \cite[Proposition 7.15]{FarbMarg} and hence by the correspondence between monodromy groups and connected components (see \S \ref{sec:wn_monodromy}), there must be infinitely many connected components of $\HT(\sing)$ covering $\Omega$.
\end{proof}

\begin{rmk}
One can also use the above correspondence to prove that in genus 2 (where every surface and every differential is hyperelliptic), the stratum $\HT(1,1)$ is connected while the stratum $\HT(2)$ has 6 components, corresponding to the 6 Weierstrass points on a genus 2 surface.
\end{rmk}

\section{Higher spin structures}\label{sec:rspin}

In this section, we collect the necessary results on higher spin structures.
As these objects do not appear frequently in the flat surfaces or Teichm{\"u}ller theory literature,
we take a more expository approach and
summarize many of their important properties.

In \S\ref{sec:def_rspin}, we give two equivalent definitions of $r$--spin structure, and in \S\ref{sec:Arf} recall an important invariant of $r$--spin structures, called the {\em Arf invariant} (Definition \ref{Arf}). In order to compare $r$--spin structures on different surfaces, in \S\ref{sec:marked_rspin} we explain how $r$--spin structures interact with a marking and how a geometric homology basis can be used to determine equality of two $r$--spin structures (Lemma \ref{lem:hom_rspin_eq}).
Finally, we explain how this theory can be used to classify the action of $\Mod(S)$ on the set of $r$--spin structures (Theorem \ref{thm:Mod_rspin_action}).

Depending on the reader's mathematical taste, it may be helpful to read \S\ref{sec:rspin_winding}, in which we give a differential--geometric characterization of $r$--spin structures, in tandem with (or even before) this section.

\subsection{Two equivalent definitions}\label{sec:def_rspin}

The most natural way to define an $r$--spin structure is in analogy with the (classical) spin structures constructed in \S\ref{sec:background}. Recall that a {\em spin structure} on a Riemann surface $X$ is a square root of the canonical bundle, that is, a (complex) line bundle $\LL \rightarrow X$ such that $\LL^{\otimes 2} \cong K_X$.

\begin{defn}
An $r$--spin structure on a Riemann surface $X$ is an $r\ith$ root of the canonical bundle, that is, a line bundle $\LL \rightarrow X$ such that $\LL^{\otimes r} \cong K_X$.
\end{defn}

Observe that we do not require the root $\LL$ to admit a holomorphic section. In fact, if $\LL \rightarrow X$ does admit a section then $X$ must admit an abelian differential with certain constraints on its divisor (see \S\ref{sec:diffspin}).

From this definition, it is easy to see that there are exactly $r^{2g}$ $r$--spin structures up to isomorphism.
Indeed, the $r$--spin structures can be put into (non-canonical) bijection with torsion elements of the Jacobian $J(X)$:
\footnote{Recall that the Jacobian $J(X)$ of a genus $g$ Riemann surface $X$ is a $g$--dimensional complex torus. By the Abel--Jacobi theorem, $J(X)$ parametrizes degree--0 divisor classes on $X$, equivalently, degree--0 line bundles on $X$. Given this identification, it naturally has the structure of an abelian group whose addition is given by taking sums of divisor classes. In the line bundle formulation, addition takes the form of the tensor product and the inverse of a line bundle $\LL$ is its dual bundle $\LL^*$. See, e.g., \cite[pp. 224--39, 333--63]{Griff_Harr}).}
if $\LL$ is an $r$--spin structure on $X$ and $j$ is an $r$--torsion element of $J(X)$, then we have that
\[(\LL \otimes j)^{\otimes r} = \LL^{\otimes r} \otimes j^{\otimes r} \cong K_X \otimes \mathcal{O} \cong K_X\]
where $\mathcal{O}$ is a trivial bundle over $X$. Therefore $\LL \otimes j$ is an $r$--spin structure.

Likewise, if $\LL_1$ and $\LL_2$ are $r$--spin structures, then $\LL_1 \otimes \LL_2^*$ is $r$--torsion, for
\[(\LL_1 \otimes \LL_2^*)^{\otimes r} =
\LL_1^{\otimes r} \otimes (\LL_2^*)^{\otimes r} \cong K_X \otimes K_X^* \cong \mathcal{O}.\]

We will now reformulate the definition of an $r$--spin structure on a surface without reference to the underlying holomorphic structure. For more details on this equivalence, see \cite{Sipe_roots} or \cite[\S\S 2,3]{Salter_planecurves}.

Choose some $r$--spin structure $\LL$ on $X$. Puncturing (that is, removing the zero sections from $K_X$ and $\LL$) induces an (unramified) cover of the corresponding punctured bundles. The punctured canonical bundle is clearly homotopy equivalent to the unit cotangent bundle $T^*_0X$, and likewise we see that the punctured $\LL$ bundle 
is homotopy equivalent to some circle bundle $Q$. Moreover, since the process of tensoring $\LL \rightarrow \LL^{\otimes r}$ locally has the form
$z \mapsto z^r$, we see that the cover $Q \rightarrow T^*_0X$ induces the standard (connected) $r$--fold cover of $S^1 \rightarrow S^1$ on fibers \cite[Proposition 2.3]{Sipe_roots}.

Let $\alpha$ denote an $S^1$ fiber of $T_0^*X$. Now $\langle \alpha \rangle$ is central inside of $\pi_1(T^*_0 X)$,
hence the cover $Q$ of the preceding paragraph corresponds to a map
\[\phi: H_1( T_0^* X, \ZZ) \rightarrow G,\]
where $G$ is some group of size $|G| = r$. Since the induced map on the fibers is given by $z \mapsto z^r$, we see that $G \cong \Zr$ and $\phi(\alpha) = 1$.

A choice of Riemannian metric on $X$ induces an isomorphism between $T_0^* X$ and $T_0 X$, giving the following (co)homological characterization of $r$--spin structures:

\begin{thm}[Theorem 1 of \cite{Sipe_roots}, see also \S \S 2,3 in \cite{Salter_planecurves}]\label{thm:rspin_equiv_defn}
The $r\ith$ roots of the canonical bundle are in $\Mod(S)$--equivariant bijection with elements of 
\begin{equation}\label{eqn:space_of_spin_defn}
\Phi_r := \{\phi \in H^1(T_0 X, \Zr) : \phi(\alpha) = 1\}.
\end{equation}
\end{thm}

We will often use $\Phi_r$ in the sequel as shorthand for ``the set of all $r$--spin structures on $X$,'' freely passing between $r\ith$ roots of the canonical bundle on a Riemann surface and their induced cohomology classes.

A reader familiar with the literature will note that Sipe's original statement of the theorem requires that $\phi(\alpha)=-1$ instead of $1$. This sign arises because she uses a Hermitian metric on $X$
and by conjugate--linearity, the isomorphism between a Hermitian vector space and its dual reverses orientation.
\footnote{In all truth, Sipe actually induces the isomorphism via the Bergman Hermitian metric on the universal curve over Teichm{\"u}ller space \cite[\S5]{Sipe_roots}. This metric restricts to a Hermitian metric on each fiber, as does the induced isomorphism.}
If one instead uses a Riemannian metric, as appears here and in \cite{Salter_monodromy} and \cite{Salter_planecurves}, then the isomorphism preserves orientation and thus does not flip the fiber.

\subsection{The induced Arf invariant}\label{sec:Arf}

An $r$--spin structure $\phi$ on $X$ comes with more data than just an $r\ith$ root. Indeed, observe that any $r$--spin structure induces an entire family of intermediate roots of the canonical bundle simply by taking intermediate powers. More formally, if $s | r$, then for any $\phi \in \Phi_r$ we have that $\phi^{\otimes (r/s)} \in \Phi_s$. In particular, when $r$ is even, any $\phi \in \Phi_r$ induces a $2$--spin structure $\phi^{\otimes(r/2)}$.

For any $2$--spin structure $\psi$, Atiyah showed in \cite{Atiyah_spin} that $h^0(X, \psi) \mod 2$, the dimension of the space of holomorphic sections $X \rightarrow \psi$ mod $2$, is deformation invariant. Johnson later proved that this value is the same as the Arf invariant of a certain quadratic form on $H_1(X, \Ztwo)$ \cite{Johnson_spin}. We briefly sketch Johnson's construction below (see also \cite[\S3.1]{Salter_monodromy}).

To begin, we recall that a $\Ztwo$ {\em quadratic form} on a (nondegenerate) symplectic vector space $(V, \langle \cdot, \cdot \rangle )$ over $\Ztwo$ is a function $q: V \rightarrow \Ztwo$ such that for any $v, w \in V$,
\[q(v + w) = q(v) + q(w) + \langle v, w \rangle.\]

\begin{defn}\label{Arf}
If $\{v_1, \ldots, v_g, w_1, \ldots, w_g\}$ is a symplectic basis for $V$ (i.e., a basis such that $\langle v_i, w_j \rangle = \delta_{ij}$) then the {\em Arf invariant} of $q$ is the value
\begin{equation}\label{eqn:Arf}
\Arf(q) := \sum_{i=1}^g q(v_i) q(w_i) \mod 2.
\end{equation}
\end{defn}
Arf proved in \cite{Arf} that this value depends only on the quadratic form and not on the choice of basis. Moreover, the symplectic group $\Sp(V)$ acts on the set of quadratic forms with two orbits, distinguished by the Arf invariant. There is also a count of how many quadratic forms have even and odd parity, respectively.

\begin{lemma}\label{lem:Arfcount}
Let $V$ be a symplectic $\Ztwo$ vector space of dimension $2g$. Then exactly $\left( 2^{g-1}(2^g +1) \right)$ of the (nonsingular) $\Ztwo$--valued quadratic forms on $V$ have even parity and $\left( 2^{g-1}(2^g - 1) \right)$ have odd.
\end{lemma}

A $2$--spin structure $\phi$ in the sense of \eqref{eqn:space_of_spin_defn} does not itself define a quadratic form on homology with $\Ztwo$ coefficients, but can be made into one by considering the {\em Johnson lift} of a homology basis.
\footnote{We note that the map presented here is the same as Johnson's original lifting \cite{Johnson_spin}, and hence does not match the convention appearing in Salter's work \cite{Salter_monodromy}.}
To that end, fix a symplectic basis for $H_1(X, \ZZ)$ consisting of smooth simple closed curves. Mimicking \cite{Salter_monodromy}, we call such a basis {\em geometric}. For each curve $a$ in the basis, the framed curve $\overrightharp{a}$ defines an element in $H_1(T_0X, \ZZ)$, and reducing coefficients mod $2$ removes dependence on the initial orientation.

The framing is not a homology invariant since the framing of a small nulhomotopic loop returns $\alpha$, the class of the $S^1$ fiber. However, the map $a \mapsto \tilde{a} :=\overrightharp{a} + \alpha$ is.

\begin{defn}
Let $a = \sum_{i=1}^N n_i a_i$ be an integral multicurve (so that $a_i$ are all pairwise disjoint simple closed curves) on a surface $X$. The {\em Johnson lift} of $a$ is
\[ \tilde{a} := \sum_{i=1}^N n_i \left( \overrightharp{a}_i +\alpha \right) \in H_1(T_0X, \Ztwo) .\]
\end{defn}

Johnson proved that this lift only depends on the homology class, and has a certain {\em twist--linearity} condition:

\begin{lemma}[Theorems 1A and 1B in \cite{Johnson_spin}]\label{lem:lift_twistlin}
The map $a \mapsto \tilde{a}$ is well--defined on homology classes in $H_1(X, \Ztwo)$, and obeys the following:
\[\widetilde{(a+b)} \equiv \tilde{a} + \tilde{b} + \langle a,b \rangle \alpha\]
where all coefficients are taken mod $2$.
\end{lemma}

Therefore for any $\psi \in \Phi_2$, the function $q_\psi(a) = \psi (\tilde{a})$ is a quadratic form on $H_1(X, \Ztwo)$, for
\begin{align*}
q_\psi(a + b) &= \psi\left(\widetilde{(a+b)}\right) \\
&= \psi \left(\tilde{a} + \tilde{b} + \langle a,b \rangle \alpha \right) \\
&= \psi(\tilde{a}) + \psi(\tilde{b}) + \langle a,b \rangle\\
&=  q_\psi(a) + q_\psi(b) + \langle a,b \rangle
\end{align*}
where the third equality follows because $\psi(\alpha) = 1$.

\begin{defn}
If $r$ is even, then an $r$--spin structure $\phi$ is called {\em even} (respectively {\em odd}) if the Arf invariant of the induced quadratic form $q_{\phi^{\otimes (r/2)}}$ is $0$ (respectively $1$).
\end{defn}

We note that since the map from $\Phi_r$ to $\Phi_{s}$ is just reduction mod $s$, \eqref{eqn:Arf} can be written as
\begin{equation}\label{eqn:rspin_Arf}
\text{Arf}\left(q_{\phi^{\otimes (r/2)}} \right) =  
\sum_{i=1}^g \big(\phi( \overrightharp{a}_i )  +1 \big) \big(\phi( \overrightharp{b}_i )  +1 \big)\mod 2
\end{equation}
for any $r$--spin structure $\phi$, whenever $\{a_1, \ldots, a_g, b_1, \ldots, b_g\}$ is a geometric basis for $H_1(X, \ZZ)$.

\subsection{Marked $r$--spin structures}\label{sec:marked_rspin}

In order to compare $r$--spin structures on different Riemann surfaces, we need to identify $X$ with a reference topological surface $S$. This will give us an easy way to tell if two $r$--spin structures are equivalent (Lemma \ref{lem:hom_rspin_eq}) and another way of counting them (Lemma \ref{lem:count_equiv}).

To that end, we define a {\em marked} $r$--spin structure to be a marked Riemann surface $(X,f)$ together with an $r$--spin structure $\phi$ on $X$. If the reference surface $S$ is endowed with a smooth structure and the marking map is smooth, then $f: S \rightarrow X$ induces a homeomorphism $Df: TS \rightarrow TX$ of tangent bundles (and of their unit sub-bundles, which we will also denote by $Df$). We can therefore use $Df$ to pull back an $r$--spin structure $\phi$ on $X$ to one on the reference surface $S$. 

\begin{defn}
We say that two marked $r$--spin structures $(X,f, \phi)$ and $(Y, g, \psi)$ are {\em topologically equivalent} if
\[(Df)^* \phi = (Dg)^* \psi\]
as elements in $H^1(T_0 S, \ZZ_r)$.
\end{defn}

In particular, this gives us an easy way to tell if two $r$--spin structures are topologically equivalent.

\begin{lemma}[c.f. Theorem 2.5 in \cite{HJ_windingnumber}]\label{lem:hom_rspin_eq}
Two marked $r$--spin structures are topologically equivalent if and only if they take the same values
\footnote{One can evaluate an $r$--spin structure $\phi$ on an oriented simple closed curve $c$ by lifting $c$ to a framed curve $\overrightharp{c}$ as in \ref{sec:Arf} and then computing $\phi(\overrightharp{c})$. Such a lift is {\em not} well--defined on homology classes in $S$, since a nulhomotopic loop evaluates to either $\pm 1$, depending on its orientation. See, e.g., \cite[\S3.1]{Salter_monodromy} or \cite[Proposition 1]{Sipe_quotients}.}
on a geometric basis for $H_1(S, \ZZ)$
\end{lemma}

The cohomological formulation of $r$--spin structures also provides another way to count $r$--spin structures without appealing to torsion in the Jacobian of a reference holomorphic structure.

\begin{lemma}\label{lem:count_equiv}
There are exactly $r^{2g}$ topological equivalence classes of marked $r$--spin structures on a surface of genus $g$. If $r$ is even, then exactly 
\[(r/2)^{2g} \left( 2^{g-1}(2^g +1) \right)\] have even parity and 
\[(r/2)^{2g} \left( 2^{g-1}(2^g - 1) \right)\]
have odd.
\end{lemma}
\begin{proof}
Elements of $H^1(T_0 S, \ZZ_r)$ are determined by their values on a basis of $H_1(T_0S, \ZZ)$, and one can choose a basis consisting of the framings of a geometric basis for $H_1(S, \ZZ)$ together with the class $\alpha$ of a fiber. An $r$--spin structure must evaluate to $1$ on $\alpha$, but can take any value in $\Zr$ on each framed basis curve. Since $H_1(S, \ZZ)$ has rank $2g$, there are therefore $r^{2g}$ possible topological equivalence classes of $r$--spin structures.

The second statement follows from equation \ref{eqn:rspin_Arf} together with the count of quadratic forms with given Arf invariant (Lemma \ref{lem:Arfcount}).
\end{proof}

Since we already have perfectly good notation for the set $\Phi_r$ of $r$--spin structures on a given Riemann surface $X$, we will
assume the generosity of the reader and subsequently conflate $\Phi_r$ with the set of topological equivalence classes of marked $r$--spin structures on the underlying (topological) surface $S$.

\subsection{Action of the mapping class group}\label{sec:Mod_rspin}

The mapping class group $\Mod(S)$ acts naturally on the space of marked $r$--spin structures by change of marking. In order to understand this action (and in particular to understand the stabilizer of a given $r$--spin structure, see Definition \ref{def:spinstab}), we will relate the action of $\Mod(S)$ on $\Phi_r$ to its action on $H_1(S, \ZZ)$.

Choose a geometric basis for $H_1(S, \ZZ)$. By taking the framings of these curves as in \S\S \ref{sec:Arf} and \ref{sec:marked_rspin}, these together with the circular fiber $\alpha$ determine a homology basis for $H_1(T_0S, \ZZ)$. Lemma \ref{lem:hom_rspin_eq} tells us that the values of $\phi \in \Phi_r$ on this basis completely determine $\phi$, so to understand the action of $\Mod(S)$ on $\Phi_r$ it suffices to consider the action on homology.

With this description, one can carry out explicit matrix multiplication to understand the action of the mapping class group on $\Phi_r$. The following theorem appears in multiple places in the literature, for example in \S4 of \cite{Salter_monodromy} and as Theorem 3.2 in \cite{Jarvis_geometry}. It can also be deduced from Sipe's work in \cite{Sipe_quotients}. Morally similar computations also appear in the proof of \cite[Theorem 4]{Walker_components}.

\begin{thm}\label{thm:Mod_rspin_action}
Let $S$ be a surface of genus at least $2$. If $r$ is odd, then $\Mod(S)$ acts transitively on the set $\Phi_r$ of $r$--spin structures. If $r$ is even, then $\Mod(S)$ acts with two orbits, distinguished by the parity of the induced $2$--spin structure.
\end{thm}

\begin{defn}\label{def:spinstab}
Let $\phi$ be an $r$--spin structure. The stabilizer of $\phi$ under the $\Mod(S)$ action is called an {\em $r$--spin mapping class group}, and is denoted by $\Mod(S)[\phi]$.
\end{defn}

By the orbit--stabilizer theorem and Lemma \ref{lem:count_equiv}, the following statements are immediate.

\begin{coro}\label{coro:stab_index}
Let $\phi \in \Phi_r$. Then the stabilizer $\Mod(S)[\phi]$ has the following index in $\Mod(S)$:
\begin{itemize}
\item $r^{2g}$ if $r$ is odd.
\item $(r/2)^{2g} \left( 2^{g-1}(2^g +1) \right)$ if $r$ is even and $\phi$ has even parity.
\item $(r/2)^{2g} \left( 2^{g-1}(2^g - 1) \right)$ if $r$ is even and $\phi$ has odd parity.
\end{itemize}
Moreover, if $\psi \in \Phi_r$ is any other $r$--spin structure (with the same parity if $r$ is even), then $\Mod(S)[\phi]$ and $\Mod(S)[\psi]$ are conjugate subgroups of $\Mod(S)$.
\end{coro}

Since $\Mod(S)$ is finitely generated and $\Mod(S)[\phi]$ is of finite index, it is also finitely generated. In \cite{Salter_monodromy}, Salter gave a criterion for a finite collection of Dehn twists to generate $\Mod(S)[\phi]$. We record his theorem below.

First, define a {\em network} of curves on a surface (possibly with nonempty boundary) to be a set of simple closed curves such that any two curves in the network intersect at most once. A network is {\em connected} if the union of all curves in the network is connected (as a topological space), and {\em arboreal} if the graph whose vertices are curves and whose edges represent intersections is a tree. A network is {\em filling} if the union of the curves cuts the surface into disks and boundary--parallel annuli.

Salter then defines the $D_{2r+3}$ configuration to be the the arrangement of simple closed curves
\[\{a_1,  a_1', c_1, \ldots, c_{2r+1}\}\]
appearing in Figure \ref{fig:Dtwiddle}. Observe that the boundary of a regular neighborhood of $a_1 \cup a_1' \cup c_1 \cup \ldots \cup c_{2r}$ is isotopic to the multicuve $\Delta_0 \cup a_{r+1} \cup a_{r+1}'$.

\begin{figure}[ht]
\centering
\includegraphics[scale=.95]{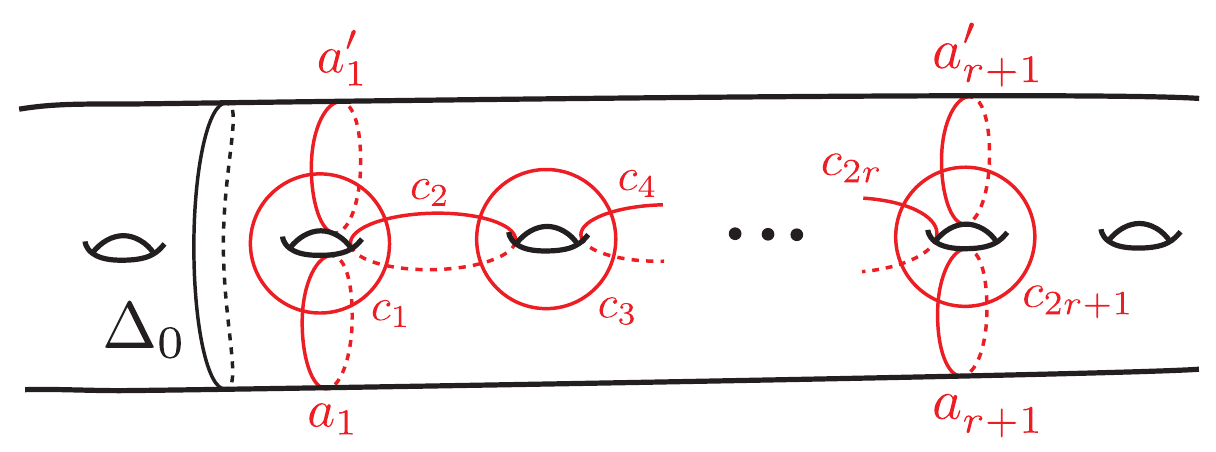}
\caption{The $D_{2r+3}$ configuration on a surface.}
\label{fig:Dtwiddle}
\end{figure}

\begin{thm}[Theorem 9.5 in \cite{Salter_monodromy}]\label{thm:Salter_generation}
Suppose that $\phi$ is an $r$--spin structure on a closed surface $S_g$ and $\mathsf{C}=\{c_i\}$ is a connected filling network on $S_g$ satisfying the following:
\begin{enumerate}
\item $\phi(\overrightharp{c}_i ) = 0$ for all $i$, where $\overrightharp{c}_i$ is the framing of the (oriented) curve $c_i$.
\item There is some subset $\{c_1, \ldots, c_{2r+4} \}$ of $\mathsf{C}$ such that $\{c_1, \ldots, c_{2r+3} \}$ is arranged in the $D_{2r+3}$ configuration and $c_{2r+4}$ corresponds to $a_{r+1}$, as shown in Figure \ref{fig:Dtwiddle}.
\item If $d$ is the curve corresponding to $\Delta_0$ in the $D_{2r+3}$ configuration, then there is some $c \in \mathsf{C}$ such that $i(c,d)=1$.
\item If $\mathsf{C}'$ is the subnetwork of $\mathsf{C}$ containing the curves which do not intersect $d$, then $\mathsf{C}'$ has a further subnetwork $\mathsf{C}''$ which is a connected arboreal filling network for $S \setminus d$.
\end{enumerate}
Then
\begin{itemize}
\item if $r$ is odd and $g \ge 5$, $\langle T({c_i}) : c_i \in \mathsf{C} \rangle = \Mod(S)[\phi]$.
\item if $r$ is even and $g \ge g(r)$ where
\[ g(r) = \left\{
\begin{array}{ll}
13 & r=4 \\
21 & r=8 \\
5 & \text{otherwise}
\end{array} \right. \]
then $\langle T({c_i}) : c_i \in \mathsf{C} \rangle$ is of finite index in $\Mod(S)[\phi]$.
\end{itemize}
\end{thm}

We remark that while Salter's theorem as stated in \cite{Salter_monodromy} requires $\mathsf{C}'$ to be an arboreal filling network for the cut surface, an analysis of his methods reveals that it is enough to require that $\mathsf{C}'$ {\em contains} such a subnetwork (c.f. \cite[Lemma 9.4]{Salter_monodromy}).

By a more careful analysis of the subgroup $\langle T({c_i}) : c_i \in \mathsf{C} \rangle$, Salter is also able to say something about its image under the symplectic representation $\psi: \Mod(S) \rightarrow \Sp(2g, \ZZ)$.

\begin{lemma}[c.f. Lemmas 5.4 and 6.4 in \cite{Salter_monodromy}]\label{lem:Salter_sympl}
Suppose that $\phi$ is an $r$--spin structure on a closed surface $S_g$ and $\mathsf{C}=\{c_i\}$ is as in Theorem \ref{thm:Salter_generation}.
\begin{itemize}
\item If $r$ is odd and $g \ge 5$,
then 
$\psi \left( \langle T({c_i}) : c_i \in \mathsf{C} \rangle \right) = \Sp(2g, \ZZ)$
\item If $r$ is even and $g \ge g(r)$,
then $\psi \left( \langle T({c_i}) : c_i \in \mathsf{C} \rangle \right)$ is the stabilizer in $\Sp(2g, \ZZ)$ of the $\ZZ/2$--quadratic form $q_{\phi^{\otimes r/2}}$.
\end{itemize}
\end{lemma}

\section{Abelian differentials and winding numbers}\label{sec:diffspin}

We have already seen in Definition \ref{def:spinstr} how any pair $(X, \omega) \in \HM(\sing)$ defines a square root of the canonical bundle $K_X$ whenever $r= \gcd(\sing)$ is even. In a similar way, it also defines an $r$--spin structure on $X$.

Below, we give an algebro-geometric interpretation of this correspondence before giving an equivalent formulation in terms of winding numbers (Proposition \ref{prop:wn_is_rspin}). Using this equivalence, in Proposition \ref{prop:wn_const_on_comp} we prove that the induced $r$--spin structure is an invariant of connected components of $\HT(\sing)$, and in \S\ref{sec:wn_monodromy} investigate the implications of this fact for the {\em geometric monodromy group} (Definition \ref{def:geomon}).

\begin{lemma}
If $X$ is a Riemann surface, then there exists an effective $r$--spin structure $\LL \rightarrow X$ if and only if $X$ admits an abelian differential $\omega$ such that $r | \gcd( \sing )$.
\end{lemma}
\begin{proof}
As $(X,\omega) \in \HM(\sing)$, the associated divisor $(\omega) = \sum_{i=1}^n k_i p_i$ is divisible by $r$, via 
\[(\omega) /r = \sum_{i=1}^n (k_i/r) p_i.\]
By the standard correspondence between divisor classes and line bundles (see, e.g., \cite[pp. 133--4]{Griff_Harr}), this divisor gives rise to a holomorphic line bundle $\LL=\LL_{(\omega)/r} \rightarrow X$ whose $r\ith$ tensor power is (isomorphic to) $K_X$. This $\LL$ is therefore an $r$--spin structure on $X$.

Moreover, $(\omega)/r$ is effective because its coefficients are all positive. Therefore the standard correspondence also yields a holomorphic section $\sigma: X \rightarrow \LL$ such that $\sigma^{\otimes r} : X \rightarrow \LL^{\otimes r} \cong K_X$ is a section of $K_X$ with
\[(\sigma^{\otimes r}) = r (\sigma) = r (\omega)/r = (\omega).\]

On the other hand, suppose that $X$ is a Riemann surface equipped with an $r$--spin structure $\LL$ and a holomorphic section $\sigma: X \rightarrow \LL$. Then by same manipulations as above, we see that $\sigma^{\otimes r}: X \rightarrow K_X$ is an abelian differential with divisor $r (\sigma)$.
\end{proof}

\subsection{Marked $r$--spin structures as winding numbers}\label{sec:rspin_winding}

While the above construction is natural (indeed, almost tautological) from an algebro-geometric perspective, it does not shed any light on the relation between $r$--spin structures and the flat geometry of $(X, \omega)$. In order to investigate this connection, we give one final interpretation of $r$--spin structures which will allow us to make the link with flat structures explicit. 
\footnote{To the best of the author's knowledge, this relationship first appears explicitly in print in work of Trapp \cite{Trapp_wn}, though a preliminary sketch appears in the proof of Proposition 3.2 of \cite{Sipe_roots}. More recently, it has resurfaced in \cite{Salter_planecurves} and \cite{Salter_monodromy} and in a partial form in \cite{Walker_components}.}

Observe that every abelian differential $\omega$ with divisor $\sum_{i=1}^n k_i p_i$ on a Riemann surface $X$ naturally defines a (nonvanishing) {\em horizontal unit vector field} $H_\omega$ on $X \setminus \{p_1, \ldots, p_n\}$. For every $x \in X \setminus \{p_1, \ldots, p_n\}$, the vector  $H_\omega(x)$ is the unique unit tangent vector such that $\omega( H_\omega(x)) \in \RR_{>0}$. Note that the horizontal foliation of $\omega$ exactly consists of the integral curves for this vector field, and at each point $p_i$ we have that
\begin{equation}\label{eqn:index_k_i}
\text{index}_{p_i}(H_\omega) = - k_i 
\end{equation}
where we recall that the index of a vector field at a singular point is the degree of the Gauss map on a small loop about that point.

\begin{figure}[ht]
\centering
\includegraphics[scale=.6]{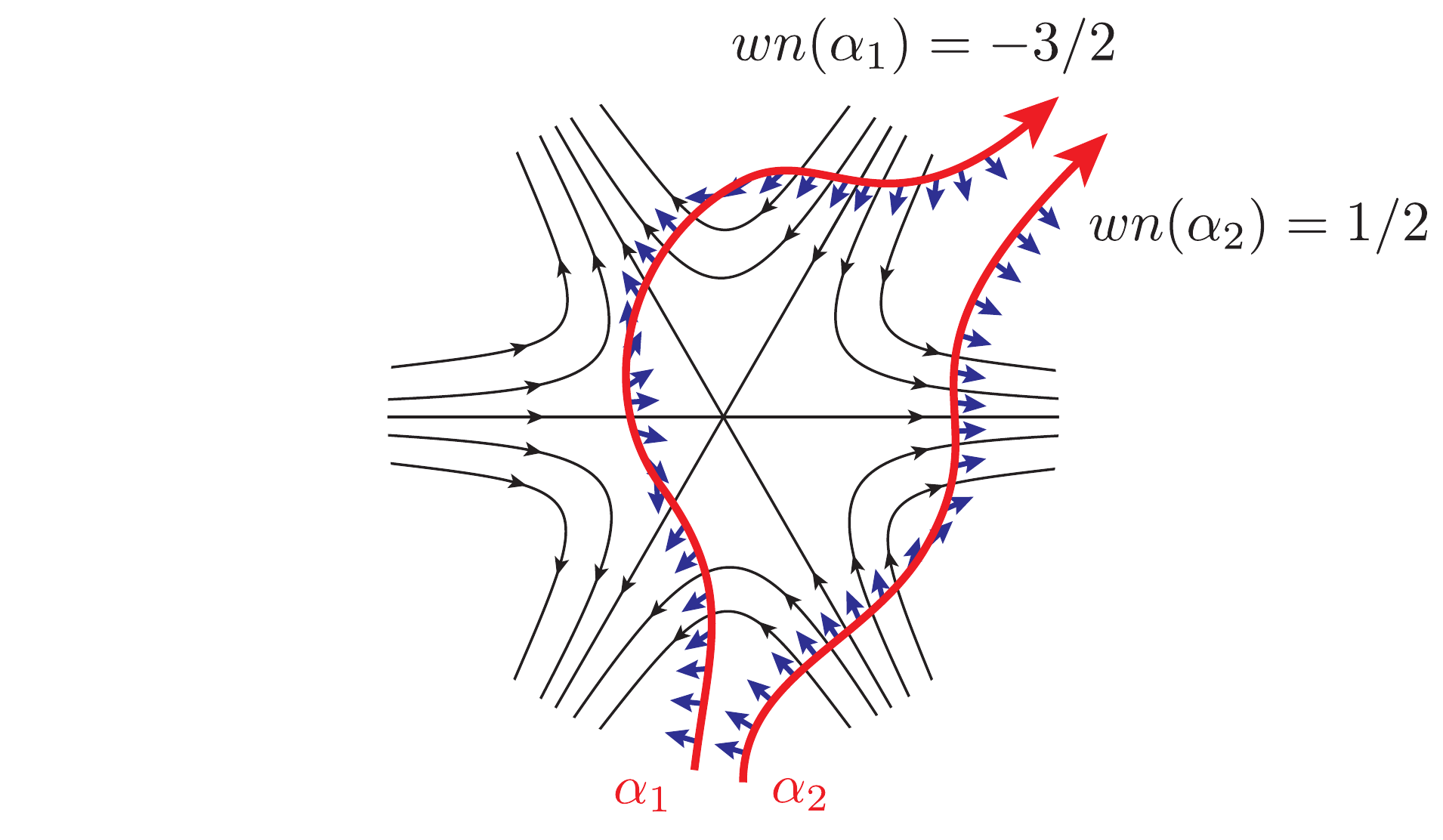}
\caption{The horizontal foliation around the zero of an abelian differential. The winding numbers of the (oriented) arcs $\alpha_1$ and $\alpha_2$, which are homotopic across the zero, differ by the degree of the zero.}
\label{fig:index_puncture}
\end{figure}

Define the {\em winding number} $wn_{(X,\omega)} (c)$ {\em with respect to} $H_\omega$ of any (smooth) oriented simple closed curve $c$ on $X \setminus \{p_1, \ldots, p_n\}$ by counting the number of times the tangent vector of $\overrightharp{c}$ turns about $H_\omega$. Observe that this assignment is not homotopy invariant, for a homotopically trivial counterclockwise loop has winding number $1$ with respect to the horizontal vector field.
\footnote{The correct notion is invariance under {\em regular homotopy}, which in particular includes isotopy. See \cite{Chill_wn1} and \cite{Chill_wn2} for a careful discussion of this construction.}
One may also compute that any small counterclockwise loop about $p_i$ has winding number exactly $k_i + 1$.

In order to make the above notion of winding number coherent for (smooth, oriented) simple closed curves on our original surface $X$, we need to understand what happens to winding numbers when we fill in a puncture. As a curve passes from one side of a zero to the other, its winding number must change by plus or minus the index of the vector field at that singularity, so by \eqref{eqn:index_k_i}, the winding number changes by the multiplicity of the zero (where the sign depends on which side of the curve the zero lies, see Figure \ref{fig:index_puncture}). Therefore, taking all winding numbers mod $r = \gcd(\sing)$ yields a well--defined function on isotopy classes of (smooth, oriented) simple closed curves on $X$. Note that a small nulhomotopic loop always has winding number $1$ mod $r$.

Moreover, this winding number function satisfies a {\em twist linearity condition}:
\footnote{Chillingworth actually only considers winding number functions corresponding to nonvanishing vector fields on punctured surfaces, and the mod $2g-2$ winding numbers obtained by taking a nonsingular vector field on $S_{g,1}$ and filling in the single puncture. For us, these correspond to the winding numbers functions obtained from a differential in the minimal stratum $\HT(2g-2)$. However, his work immediately generalizes to vector fields obtained by filling in multiple punctures.}

\begin{lemma}[Lemma 4.2 of \cite{Chill_wn2}]
If $r = \gcd(\sing)$ and $(X, \omega) \in \HM(\sing)$, one has
\begin{equation}\label{eqn:wn_twist_linear}
wn_{(X,\omega)}^r( T(b) \cdot c) = wn_{(X,\omega)}^r(c) + (b.c) wn_{(X,\omega)}^r (b) \mod r
\end{equation}
where $b$ and $c$ are oriented simple closed curves on $X$, $T(b)$ is the Dehn twist about $b$, and $(b.c)$ is the {\em algebraic} intersection number. 
\end{lemma}

In \cite{HJ_windingnumber}, Humphries and Johnson classified such twist--linear winding number functions, and in our case, their work implies that the winding number function factors through $H_1(T_0 X, \ZZ)$.

\begin{lemma}[c.f. Theorem 2.5 of \cite{HJ_windingnumber}]\label{lemma:HJ}
There is some $\phi \in H^1(T_0 X, \Zr)$ so that $wn_{(X,\omega)}^r = \phi \circ h$, where
\[h : \{ \text{oriented simple closed curves}\} \rightarrow H_1(T_0 X, \ZZ)\]
is the map which sends an oriented simple closed curve to the homology class of its framing.
\end{lemma}

Since the framing of a small nulhomotopic loop is homotopic to a fiber $\alpha$, we have that $\phi(\alpha)=1$, hence

\begin{prop}\label{prop:wn_is_rspin}
Let $(X, \omega) \in \HM(\sing)$ where $\gcd(\sing)=r$. Let $\phi \in H^1(T_0 X, \Zr)$ be the cohomology class resulting from Lemma \ref{lemma:HJ}; then $\phi$ is an $r$--spin structure.
\end{prop}

Tracing through the definitions, the reader should convince herself that this cohomology class is the same as the one corresponding to the $r$--fold cover of $T_0 X$ induced by the $r\ith$ root $\LL_{(\omega)/r}$ of $K_X$ discussed in the introduction to this section.

Moreover, given any $r$--spin structure $\phi$ on any Riemann surface $X$, a (meromorphic) section $\mu: X \rightarrow \LL$ defines a horizontal vector field $H_\mu$ on $X$ away from the zeros and poles of $\mu$ and hence a corresponding mod $r$ winding number function. Therefore we see that there is a natural one--to--one correspondence between $r$--spin structures and mod $r$ winding number functions.

By the work of Humphries and Johnson, we also have the following {\em homological coherence} property:

\begin{lemma}[Lemma 2.4 in \cite{HJ_windingnumber}, see also Proposition 3.8 of \cite{Salter_monodromy}]\label{lem:hom_coh}
Suppose that $\phi$ is any $r$--spin structure on $(X, \omega)$ and $Y$ is a subsurface of $X$ with boundary components $c_1, \ldots, c_m$. Then if the $c_i$ are oriented such that $Y$ always lies on the left--hand side of $\overrightharp{c}_i$,
\[ \sum_{i=1}^m \phi(\overrightharp{c}_i) \equiv \chi(Y) \mod r.\]
\end{lemma}

The above equivalence between mod $r$ winding number functions and $r$--spin structures then allows us to state the following geometrically obvious generalization of \cite[Lemma 1]{KZ_strata}.

\begin{lemma}\label{lem:cyl_wn0}
Let $(X,f,  \omega) \in \HT(\sing)$ be a marked abelian differential with $\gcd(\sing)=r$, and set $\phi$ to be the $r$--spin structure induced by $\omega$. Then if $c$ is a curve everywhere transverse to the horizontal foliation,  $\phi(\overrightharp{c})=0$. Similarly, if $c$ is the core curve of a horizontal cylinder on $X$, we have $\phi(\overrightharp{c})=0$.
\end{lemma}

\subsection{Invariance of winding number under deformation}\label{sec:invariance}

Now that we have interpreted $r$--spin structures in flat geometric language, we can use this to construct an invariant of components of $\HT(\sing)$. The arguments in this section are modeled on ideas contained in \cite[Proposition 1]{Walker_components}.

\begin{prop}\label{prop:wn_const_on_comp}
The mod $r$ winding number of any (smooth, oriented) simple closed curve is constant on each component of $\HT(\sing)$.
\end{prop}
\begin{proof}
Suppose that $(X, f, \omega)$ and $(Y, g, \eta)$ lie in the same component of $\HT(\sing)$ and $c$ is a (smooth, oriented) simple closed curve on our reference surface $S$. Then we need to show that
\[ wn_{(X,H_\omega)}^r(f(c)) = wn_{(Y,H_\eta)}^r(g(c)).\]
We prove below that the mod $r$ winding number of $c$ is continuous on $\HT(\sing)$. Therefore since it is a continuous map into the discrete space $\Zr$, it must be constant on the connected components of its domain.

To demonstrate continuity, we pull everything back to our reference surface $S$ and compare winding numbers there. To that end, observe that if $c$ is a simple closed curve on $S$ and $(X, f, \omega)$ is a marked abelian differential, then we can push forward the vector field $H_\omega$ on $X$ to a vector field $\left(Df^{-1}\right)_* H_\omega$ on $S$. One can analogously define a mod $r$ winding number of any (smooth, oriented) simple closed curve on $S$ with respect to $\left(Df^{-1}\right)_* H_\omega$, and it is immediate that
\begin{equation}\label{eqn:pullback_wn}
wn_{(S,\left(Df^{-1}\right)_* H_\omega )}^r(c) = wn_{(X,H_\omega)}^r(f(c)).
\end{equation}
Since the horizontal vector field $H_\omega$ depends continuously on $\omega$,
the left hand side of \eqref{eqn:pullback_wn} is continuous in $(X, f, \omega)$. Therefore the right hand side must be, and so the mod $r$ winding number of $c$ is constant on components of $\HT(\sing)$.
\end{proof}

Choosing a geometric basis of $H_1(S, \ZZ)$ and taking the corresponding framed curves, this implies that

\begin{coro}\label{coro:rspin_const_on_comp}
Any two marked abelian differentials in the same connected component of $\HT(\sing)$ define the same (topological equivalence class of) $r$--spin structure.
\end{coro}
\begin{proof}
Let $\Omega$ be a connected component of $\HT(\sing)$ and pick a geometric basis $\mathcal{B}$ of  $H_1(S, \ZZ)$. Suppose that $(X, f, \omega)$ and $(X', f', \omega')$ are both in $\Omega$ and define $r$--spin structures $\phi$ and $\phi'$.
By Proposition \ref{prop:wn_const_on_comp},
\[ \phi(b) = \phi'(b) \text{ for all } b \in \mathcal{B}\]
and therefore by Lemma \ref{lem:hom_rspin_eq}, it must be that $\phi = \phi'$.
\end{proof}

In particular, this allows us to put a lower bound on the number of connected components of strata over Teichm{\"u}ller space.

\begin{thm}\label{thm:lower_bd}
If $g \ge 3$ and $\sing$ is a partition of $2g-2$ with $\gcd(\sing)= r$, then there exist at least $r^{2g}$ non-hyperelliptic connected components of $\HT(\sing)$.
\end{thm}
\begin{proof}
First, assume that $r$ is odd; then by Theorem \ref{thm:KZ_class}, the stratum $\HM(\sing)$ is nonempty and connected (unless $r=g-1$, in which case 
\[\HM(g-1, g-1) \setminus \HM(g-1,g-1)^{\hyp}\]
is nonempty and connected). Choose some $(X, \omega) \in \HM(\sing)$, fix a marking $f: S \rightarrow X$, and let $\Omega$ denote the component of $\HT(\sing)$ containing $(X, f, \omega)$.

Now by the discussion above, $(X, f, \omega)$ defines a marked $r$--spin structure $\phi$ which by Corollary \ref{coro:rspin_const_on_comp} must be topologically equivalent to the marked $r$--spin structure coming from any marked abelian differential in $\Omega$. Since $\Mod(S)$ acts transitively on $\Phi_r$, there are elements $\{ e=g_1, \ldots, g_{r^{2g}}\} \subset \Mod(S)$ such that 
\[g_{i}^* \phi \neq g_j^* \phi \text{ for all } i \neq j.\]
Therefore by Corollary \ref{coro:rspin_const_on_comp}, $g_1 \Omega, \ldots, g_{r^{2g}} \Omega$ are all distinct, and the statement is proved.

The proof for even $r$ is analogous, but now there are non-hyperelliptic components of $\HM(\sing)$ corresponding to both parities of spin structures.
For this situation, one must choose a differential and a marking for each component, and note that $\Mod(S)$ acts transitively on the set of $r$--spin structures with fixed parity (Theorem \ref{thm:Mod_rspin_action}).
\end{proof}

\subsection{Winding numbers and monodromy}\label{sec:wn_monodromy}

In order to put an upper bound on the number of connected components of $\HT(\sing)$, we will use some elementary covering space theory to rephrase the problem in terms of subgroups of the mapping class group.

For $\spin \in \{\even, \odd\}$, the forgetful map $p: \HM_g \rightarrow \M_g$ induces a map of orbifold fundamental groups
\[p_* : \pi_1^\text{orb} \left( \HM(\sing)^{\spin}, (X, \omega) \right)  \rightarrow \pi_1^\text{orb} \left( \M_g, X \right) \cong \Mod(X)\]
(recall that when $r$ is odd the $\spin$ superscript is assumed to be empty). 

\begin{defn}\label{def:geomon}
The {\em geometric monodromy group} $\mathcal{G}(\sing, \spin)$ of the stratum $\HM(\sing)^{\spin}$ is the image of $p_*$ inside of $\Mod(X)$.
\end{defn}

\begin{rmk}
Note that our definition of $\mathcal{G}(\sing, \spin)$ depends on our choice of basepoint $(X, \omega)$, and while change of basepoint will result in isomorphic groups, it does not necessarily result in the same subgroup of $\Mod(X)$.
Because of this, we consider $\mathcal{G}(\sing, \spin)$ only ever up to conjugation within $\Mod(X)$.
\end{rmk}

A choice of marking $f:S \rightarrow X$ identifies $\Mod(S)$ and $\Mod(X)$, and moreover identifies $\mathcal{G}(\sing, \spin)$ with the stabilizer of the component $\widetilde{\Omega}$ of $\HT(\sing)^{\spin}$ containing $(X, f, \omega)$.

\begin{coro}\label{coro:monodromy_containment}
Let $g \ge 3$ and $\sing$ a partition of $2g-2$ with $\gcd(\sing) = r$. If $r$ is even, also choose $\spin \in \{\even, \odd\}$. Choose some marked abelian differential $(X, f, \omega)$ living inside a component $\widetilde{\Omega}$ of $\HT(\sing)^{\spin}$ and let $\phi \in \Phi_r$ be the (marked) $r$--spin structure induced by $\omega$. Then
\[ \mathcal{G}(\sing, \spin) = p_*\left( \pi_1 \left( \HM(\sing)^{\spin}, (X, \omega) \right) \right) 
\cong \textnormal{Stab}_{\Mod(S)} \left( \widetilde{\Omega} \right) \le \Mod(S)[\phi].\]
\end{coro}
\begin{proof}
Suppose that $g \in \mathcal{G}(\sing, \spin)$; then it can be represented as a loop $\gamma$ inside of $\HM(\sing)^{\spin}$ based at $(X, \omega)$. Lifting $\gamma$ to a path $\tilde{\gamma}$ in $\HT(\sing)^{\spin}$, we see that $\tilde{\gamma}$ connects $(X, f, \omega)$ and $g \cdot (X, f, \omega)$ and so $g$ must preserve the connected component $\widetilde{\Omega}$.

Similarly, if $g \in \Mod(S)$ stabilizes $\widetilde{\Omega}$, then since $\widetilde{\Omega}$ is also path--connected we may connect $(X, f, \omega)$ to $g \cdot (X, f, \omega)$ via some path whose projection to $\HM(\sing)^{\spin}$ under the covering map will be a loop based at $(X, \omega)$.
\end{proof}

The rest of the proof of Theorem \ref{thm:components} consists of showing that this containment is in fact an equality when $r$ is odd, and that it is of finite index when $r$ is even (and $r \neq 2g-2, g-1$).

\section{Construction of prototypes}\label{sec:prototypes}

In this section, we show how to construct a special (marked) abelian differential with given singularity and $2$--spin data. First, we construct a flat metric on a Riemann surface with the correct cone angles (Construction \ref{constr:proto}) and then in Lemma \ref{lem:trivhol} prove that the metric actually comes from an abelian differential. Finally, we show that the differential so constructed induces a spin structure of the correct parity (Lemma \ref{lem:rightArf}).

Throughout, we suppress the marking $f: S \rightarrow X$. However, since it is important exactly which curves are realized as the core curves of cylinders in $X$, the marking will be implicit in much of our discussion.

Before all else, we must fix a set of simple curves whose complement has combinatorial type compatible with a stratum. In order to define these, we adopt different naming conventions for simple closed curves on $S$ as pictured in Figure \ref{fig:curvelabels}, depending on the parities of $\gcd(\sing)$ and $\spin$, together with the residue class of $g$ mod $4$. We will subsequently conflate these curves with their images on $X$ under the marking $f: S \rightarrow X$.

\begin{figure}[ht]
\centering
\begin{subfigure}{\textwidth}
\centering
\includegraphics[scale=.8]{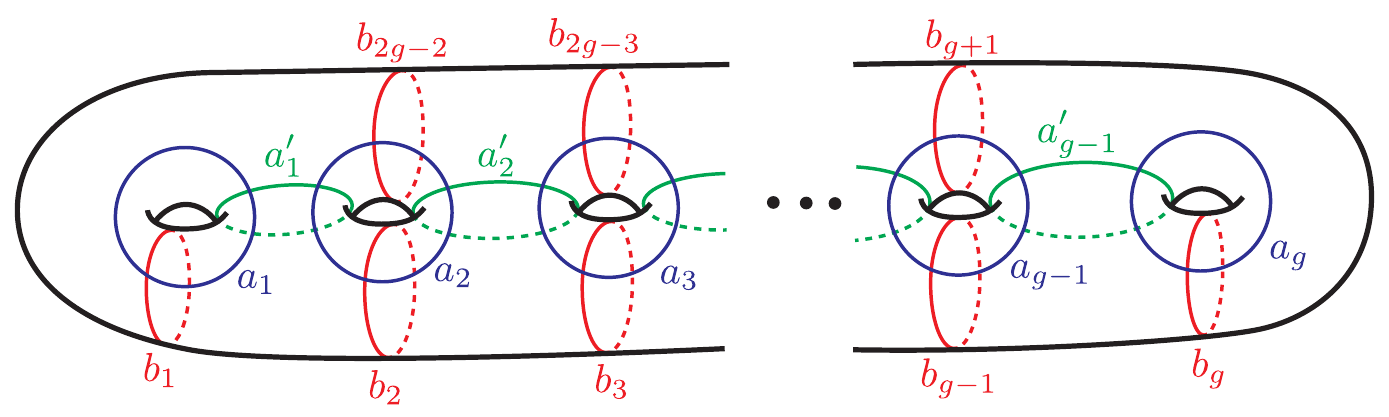}
\caption{Labels in cases (1) and (2) of Definition \ref{def:curvesystem}.}
\label{fig:curvelabels12}
\end{subfigure}\\
\begin{subfigure}{\textwidth}
\centering
\includegraphics[scale=.8]{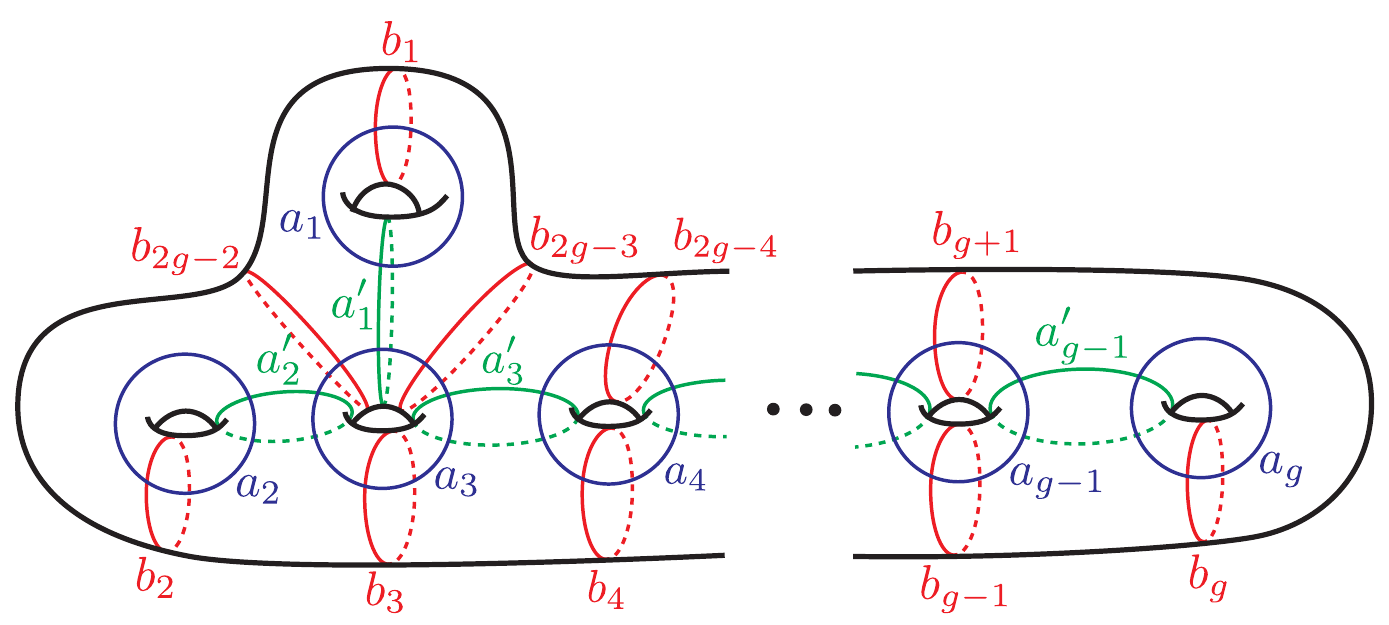}
\caption{Labels in case (3) of Definition \ref{def:curvesystem}.}
\label{fig:curvelabels3}
\end{subfigure}
\caption{Naming conventions for simple closed curves, depending on $\gcd(\sing)$, $\spin$, and $g$.}
\label{fig:curvelabels}
\end{figure}

\begin{defn}\label{def:curvesystem}
Let $g \ge 4$ and $\sing=(k_1, \ldots, k_n)$ a partition of $2g-2$. If $\gcd(\sing)$ is even, let $\spin \in \{\even, \odd\}$. Label the simple closed curves of $S$ in the following way:
\begin{enumerate}
\item If $\gcd(\sing)$ is odd, then label the curves of $S$ as in Figure \ref{fig:curvelabels12}.

\item If $\gcd(\sing)$ is even and either
\begin{enumerate}
\item $g \equiv 1 \text{ or } 2 \mod 4$ and $\spin = \odd$
\item $g \equiv 3 \text{ or } 0 \mod 4$ and $\spin = \even$
\end{enumerate}
then label the curves of $S$ as in Figure \ref{fig:curvelabels12}.

\item If $\gcd(\sing)$ is even and either
\begin{enumerate}
\item $g \equiv 1 \text{ or } 2 \mod 4$ and $\spin = \even$
\item $g \equiv 3 \text{ or } 0 \mod 4$ and $\spin = \odd$
\end{enumerate}
then label the curves of $S$ as in Figure \ref{fig:curvelabels3}.
\end{enumerate}
For either of the labeling schemes, set 
\[\mathsf{A} = \{ a_i\} \cup\{ a_i'\}\]
and define the {\em curve system of type $(\sing, \spin)$} to be
\[\mathsf{C}(\sing ,\spin ) = \mathsf{A} \cup \left\{b_i: i = 3 + \sum_{j=1}^\ell k_j \text{ for }j=1, \ldots, n\right\}\]
where indices are understood mod $2g-2$.
\footnote{The reason for starting at $b_3$ instead of $b_1$ or $b_2$ is to facilitate our proofs in Section \ref{sec:gen_monodromy} and to keep notation consistent between cases. The construction outlined below works just as well if one instead starts at any $b_i$, but then some extra work must be done to always recover a system of curves satisfying the conditions of Theorem \ref{thm:Salter_generation}.}
\end{defn}

Observe that the components of $S \setminus \mathsf{C}(\sing, \spin)$ are all disks. Moreover, if $\sing = (k_1, \ldots, k_n)$, then there are exactly $n$ disks $D_1, \ldots, D_i$ and the closure of each $D_i$ is an (immersed) $4(k_i+1)$-gon whose edges lie on $\mathsf{C}(\sing, \spin)$. 

\begin{constr}[Prototypes]\label{constr:proto}
To upgrade our curve system into an actual flat structure, we will employ a standard construction often attributed to Thurston and Veech. Consider $\mathsf{C}(\sing, \spin)$ as an embedded $1$--complex in $S$, with edges the simple arcs of $\mathsf{C}(\sing, \spin)$ and vertices their points of incidence. Since $\mathsf{C}(\sing, \spin)$ fills $S$, the dual complex $\mathcal{D}$ defines a square-ulation of $S$. Simply by declaring each square of $S \setminus \mathcal{D}$ to be a flat unit square, we get a flat cone metric $\sigma$ on $S$ with cone angles
\[\frac{\pi}{2} \cdot 4(k_i+1) = 2(k_i+1 ) \pi,\]
one contained in each $D_i$.
In addition, one can check by inspection that the curves of $\mathsf{C}_h$ and $\mathsf{C}_v$ are the core curves of cylinders on the surface.

Let $(X, f)$ denote the underlying (marked) Riemann surface so defined, and call $(X, f, \sigma)$ a {\em prototype} for the pair $(\sing, \spin)$.
\end{constr}

In general, the metric constructed above only comes a quadratic differential on $X$. To show that $\sigma$ comes from an abelian differential, we must analyze its holonomy.

\begin{lemma}\label{lem:trivhol}
The flat metric $\sigma$ on the prototype $(X, f)$ defined in Construction \ref{constr:proto} comes from an abelian differential; that is, there is some $\omega$ so that $\sigma$ is (isometric to) the metric induced by $\omega$.
\end{lemma}
\begin{proof}
To show that the flat metric comes from an abelian differential, we construct a horizontal (unit) vector field $V$ with singularities only at the cone points. This then implies that $\sigma$ has trivial holonomy and hence $(X, \sigma)$ is isometric to the flat metric on $(X, \omega)$ for some abelian differential $\omega$ (see, e.g., \cite[\S1.2]{Zorich_Survey}).

In order to build $V$, we will show that the squares tiling $(X, \sigma)$ can be coherently oriented so that the right hand side of any square is glued to the left side of another, and similarly the top of a square is glued to the bottom of another.
\footnote{Observe that this construction also directly exhibits $(X, \sigma)$ as a translation surface, glued together from squares.}
Each square can then be equipped with the rightwards--pointing horizontal vector field, and the coherence condition then guarantees that the resulting vector field extends over the edges of the squares.

Partition the curves of $\mathsf{C}(\sing, \spin)$ into two maximal multicurves $\mathsf{C}_h$ and $\mathsf{C}_v$ (for concreteness, say $\mathsf{C}_h$ consists of those curves labeled by some $a_i$ and $\mathsf{C}_v$ consists of those labeled by either $b_i$ or $a_i'$). 
To orient the squares, we note that it suffices to orient the curves of $\mathsf{C}(\sing, \spin)$ so that each curve of $\mathsf{C}_h$ intersects $\mathsf{C}_v$ positively at each point of intersection; then the horizontal direction is given by the orientation of $\mathsf{C}_h$ and the vertical by that of $\mathsf{C}_v$.  See Figure \ref{fig:curvesys_vect}.

For $a, b, \in \mathsf{C}(\sing, \spin)$, define
\[\{a.b\} := \left\{ \begin{array}{ll}
(a.b) & \text{ if } a \in \mathsf{C}_h \text{ and } b \in \mathsf{C}_v \\
(b.a) & \text{ if } b \in \mathsf{C}_h \text{ and } a \in \mathsf{C}_v \\
0 & \text{ else}
\end{array}\right.\]
where $(a.b)$ is the algebraic intersection number of $a$ and $b$. This function returns the algebraic intersection number of $a$ and $b$, ordered to take the intersection of $\mathsf{C}_h$ with $\mathsf{C}_v$. Our goal is thus to orient the curves of $\mathsf{C}(\sing, \spin)$ so that if $i(a,b)=1$ then $\{a.b\}=1$.

In order to construct the desired orientation, choose an arbitrary orientation for $a_1$. We claim that we can inductively extend this choice to a globally coherent orientation on $\mathsf{C}(\sing, \spin)$. Indeed, observe that the curves of $\mathsf{C}(\sing, \spin)$ form a connected, arboreal network. Let $N_r$ denote the $r$--neighborhood of $a_1$ in the intersection graph $\Lambda$ (recall that $\Lambda$ has one vertex for each curve of $\mathsf{C}(\sing, \spin)$ and an edge whenever two curves intersect). 

Suppose that we have induced a coherent orientation on all of the curves of $N_r$. Since $\mathsf{C}(\sing, \spin)$ is arboreal, each curve $a$ in $N_r \setminus N_{r-1}$ intersects exactly one curve $b$ of $N_{r-1}$, and hence there is a unique choice of orientation on $c$ which makes $\{a.b\}$ positive. See Figure \ref{fig:curvesys_orientation}.

Therefore, by induction (and the fact that $\Lambda$ is connected) we see that we can induce an orientation on the curves of $\Lambda$ so that whenever $i(a,b) = 1$ we have $\{a.b\} =1$.

\begin{figure}[ht]
\centering
\begin{subfigure}{.45\textwidth}
\centering
\includegraphics[scale=.5]{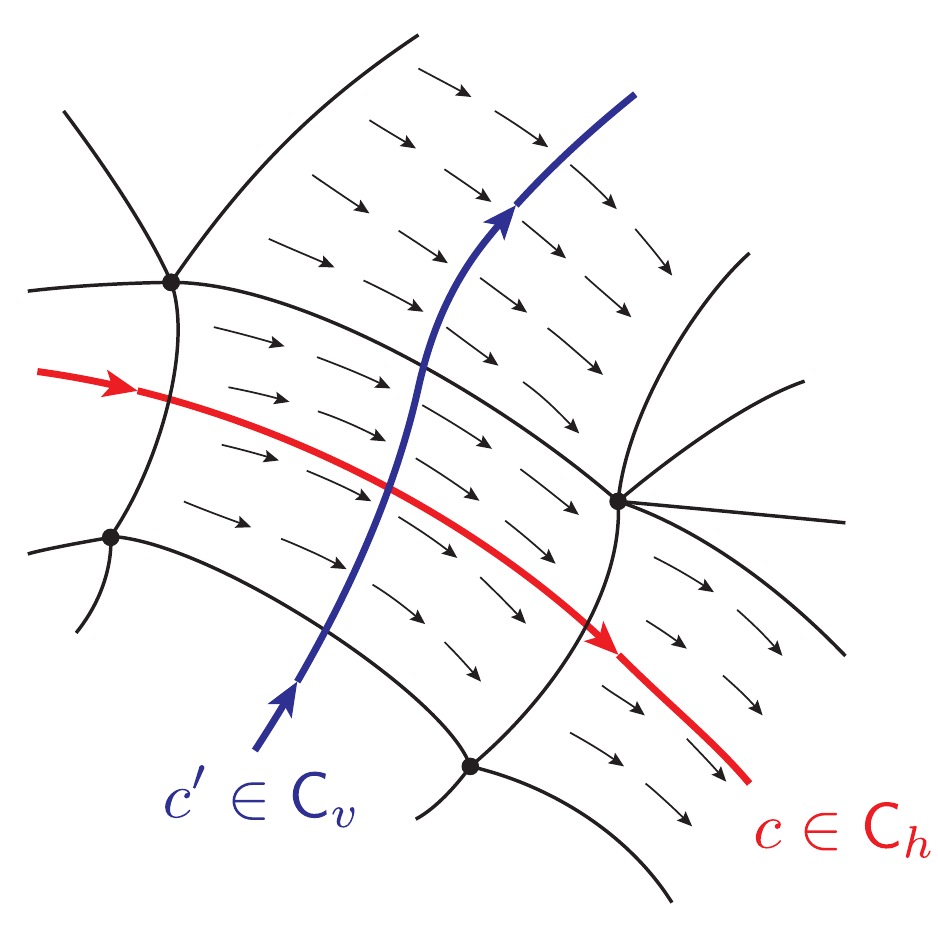}
\caption{Extending the orientation of $\mathsf{C}(\sing, \spin)$ to a\\horizontal vector field.}
\label{fig:curvesys_vect}
\end{subfigure}
\hspace{.25cm}
\begin{subfigure}{.45\textwidth}
\centering
\includegraphics[scale=.5]{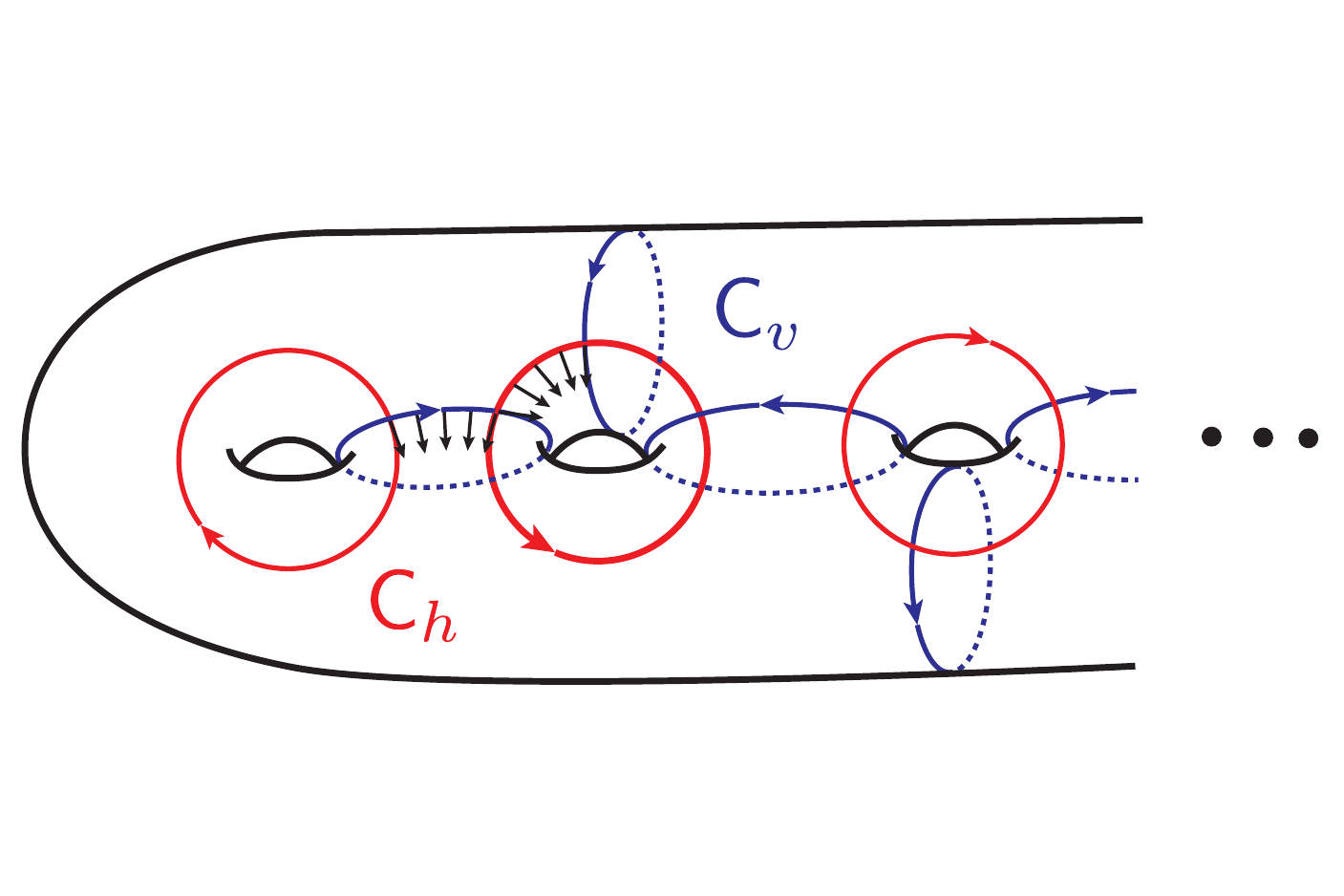}
\caption{Extending a local choice of orientation to a global \\orientation of $\mathsf{C}(\sing, \spin)$.}
\label{fig:curvesys_orientation}
\end{subfigure}
\caption{Proving that the flat square-ulation dual to $\mathsf{C}(\sing, \spin)$ has trivial holonomy.}
\label{fig:curvesys_STS}
\end{figure}

The horizontal vector fields on each square therefore glue together coherently, and so $X$ admits a horizontal unit vector field with singularities only at the cone points. It follows that the metric is induced by some abelian differential $\omega$.
\end{proof}

Finally, we need to show that the choice of $\spin$ used in the construction of the prototype actually matches the parity of the prototype abelian differential $(X, f, \omega)$.

\begin{lemma}\label{lem:rightArf}
When $r = \gcd(\sing)$ is even, the prototype $(X, f, \omega)$ for the pair $(\sing, \spin)$ has parity equal to $\spin$.
\end{lemma}

\begin{proof}
We use the homological coherence property of winding number functions (Lemma \ref{lem:hom_coh}). Let $\phi$ be the $r$--spin structure determined by the marked abelian differential $(X, f, \omega)$.

First, suppose that $\gcd(\sing)$ is even and either
\begin{itemize}
\item $g \equiv 1 \text{ or } 2 \mod 4$ and $\spin = \odd$ or
\item $g \equiv 3 \text{ or } 0 \mod 4$ and $\spin = \even$.
\end{itemize}
Then the curves are labeled as in Figure \ref{fig:curvelabels12}, and so for any $i \ge 4$ the set
\[\{b_3, a_3', \ldots, a_{i-1}', b_i\}\]
bounds an $(i-1)$ times--punctured sphere. Similarly, $\{b_2, a_2', b_3\}$ and $\{b_1, a_1', a_2', b_3\}$ bound a thrice--punctured sphere and four times--punctured sphere, respectively. Therefore by Lemmas \ref{lem:hom_coh} and \ref{lem:cyl_wn0}, we have that $\phi(\overrightharp{b}_i)$ is even if and only if $i$ is odd. Applying \eqref{eqn:Arf}, we get
\[\text{Arf}\left(q_{\phi^{\otimes (r/2)}} \right) \equiv 
\sum_{i=1}^g \big(\phi( \overrightharp{a}_i )  +1 \big) \big(\phi( \overrightharp{b}_i )  +1 \big)
 \equiv \# \{ 1 \le i \le g : i \text{ is odd}\} \mod 2\]
which is $0$ when $g \equiv 0 \text{ or } 3 \mod 4$ and $1$ if $g \equiv 1 \text{ or } 2 \mod 4$.

Now suppose $\gcd(\sing)$ is even and either
\begin{itemize}
\item $g \equiv 1 \text{ or } 2 \mod 4$ and $\spin = \even$ or
\item $g \equiv 3 \text{ or } 0 \mod 4$ and $\spin = \odd$.
\end{itemize}
Then likewise, we have that $\{b_2, a_2', b_3\}$ bounds a thrice--punctured sphere and for each $4 \le i \le g$,
\[\{b_3, a_3', \ldots, a_{i-1}', b_i\}\]
bounds an $(i-1)$--times punctured sphere. However, now $b_{2g-2}$ is symplectically dual to the basis element $a_1$ while $b_1$ is not. Therefore since $\{b_{2g-2}, a_g', b_3\}$ bounds a thrice--punctured torus, we have that $\phi( \overrightharp{b}_{2g-2})$ is odd by Lemmas \ref{lem:cyl_wn0} and \ref{lem:hom_coh}. It follows that \eqref{eqn:Arf} tells us that
\[\text{Arf}\left(q_{\phi^{\otimes (r/2)}} \right) \equiv 
\big(\phi( \overrightharp{a}_1 )  +1 \big) \big(\phi( \overrightharp{b}_{2g-2})+1 \big)+
\sum_{i=2}^g \big(\phi( \overrightharp{a}_i )  +1 \big) \big(\phi( \overrightharp{b}_i )  +1 \big) 
 \equiv \# \{ 2 \le i \le g : i \text{ is odd}\} \mod 2\]
which is $0$ when $g \equiv 1 \text{ or } 2 \mod 4$ and $1$ when $g \equiv 0 \text{ or } 3 \mod 4$.

Therefore in both cases the parity of the $2$--spin structure induced by the abelian differential matches the label used to construct the curve system.
\end{proof}

For ease of reference, we package the results of Construction \ref{constr:proto} and Lemmas \ref{lem:trivhol} and \ref{lem:rightArf} together into the following:

\begin{prop}\label{prop:prototypes}
For any $(\sing, \spin)$, there is an abelian differential
\[(X, f, \omega) \in \HT(\sing)^{\spin}\]
such that $\mathsf{C}(\sing, \spin)$ is set of all horizontal and vertical cylinders on $Y$.
\end{prop}

\section{Generating the geometric monodromy}\label{sec:alltogether}

In this section, we prove our main theorems. Throughout, we will let $(X, f, \omega)$ denote the prototype for the pair $(\sing, \spin)$ where $r=\gcd(\sing) \notin \{2g-2, g-2\}$ and $\phi$ the marked $r$--spin structure induced by $\omega$. As in Section \ref{sec:wn_monodromy}, the marking induces an identification 
\[\mathcal{G}(\sing, \spin)=\text{Stab}_{\Mod(S)}(\widetilde{\Omega}),\]
where $\widetilde{\Omega}$ is the component of $\HT(\sing)$ containing $(X, f, \omega)$.

The main result of this section is Theorem \ref{thm:monodromy}, which virtually identifies the groups $\mathcal{G}(\sing, \spin)$, $\Mod(S)[\phi]$, and the following group generated by Dehn twists:
\begin{defn}
Let $\mathsf{C}(\sing, \spin)$ be defined as in Definition \ref{def:curvesystem}. Then set 
\[\Gamma(\sing, \spin) = \langle T(c) : c \in \mathsf{C}(\sing, \spin) \rangle.\]
\end{defn}
In the process of proving Theorem \ref{thm:monodromy}, we also arrive at an understanding of the action of $\Mod(S)$ on the set of connected components of $\HT(\sing)$ (Theorem \ref{thm:components}).

Our strategy is to realize the elements of $\Gamma(\sing, \spin)$ as flat deformations (Lemma \ref{lem:twist_in_monodromy}) and then to show that these twists are enough to generate the entire geometric monodromy group (or if $r$ is even, a finite--index subgroup thereof). While in some special cases the latter statement follows easily from Theorem \ref{thm:Salter_generation}, in general we must implement some sort of iterative procedure to reduce down to a special case. As we describe in \S\ref{sec:Euclidean}, this procedure in turn is the consequence of a loose analogy between our curve systems and modular arithmetic, which allows us to use the Euclidean algorithm to complete $\mathsf{C}(\sing, \spin)$ to $\mathsf{C}\left((r^{(2g-2)/r}), \spin\right)$ (Theorem \ref{thm:Euclidean_algorithm}).

\subsection{Cylinder shears and Dehn twists}\label{sec:gen_monodromy}

The first thing we must do is show that Dehn twists in the curves of $\mathsf{C}(\sing, \spin)$ can be realized as flat deformations of our prototype surface.

Since each curve $c \in \mathsf{C}(\sing, \spin)$ is realized as the corve curve of a cylinder on $(X,f,\omega)$, we may twist along the cylinder without exiting the stratum $\HT(\sing)$. We briefly recall the construction of \cite{Wright_CylDeform} below, and direct the interested reader there for a much richer picture of these deformations.

Let $(X, f, \omega)$ be any marked abelian differential and let $\xi$ be a maximal flat cylinder of $(X, f, \omega)$ with core curve $f(c)$. Set $m$ to be the inverse modulus of $\xi$ (the ratio of its width to its height). Without loss of generality, we may assume that the cylinder is horizontal and apply the horocyclic flow
\[u_t = \left( \begin{array}{cc}
1 & t \\ 0 & 1
\end{array} \right)\]
to the cylinder $\xi$ while fixing the rest of the surface. This operation yields a family of {\em cylinder shears} $u_t(\xi) \cdot (X, \omega)$ of our original surface, as shown in Figure \ref{fig:cyl_shears}. 
Moreover, a full shear by the inverse modulus $m$ preserves the flat structure and acts by Dehn twisting in $f(c)$, that is,
\begin{equation}\label{eqn:twist}
u_m(\xi) \cdot (X, f, \omega) = (X, T({f(c)})^{-1} \circ f, \omega) = (X,  f \circ T(c)^{-1}, \omega).
\end{equation}

\begin{figure}[ht]
\centering
\includegraphics[scale=.6]{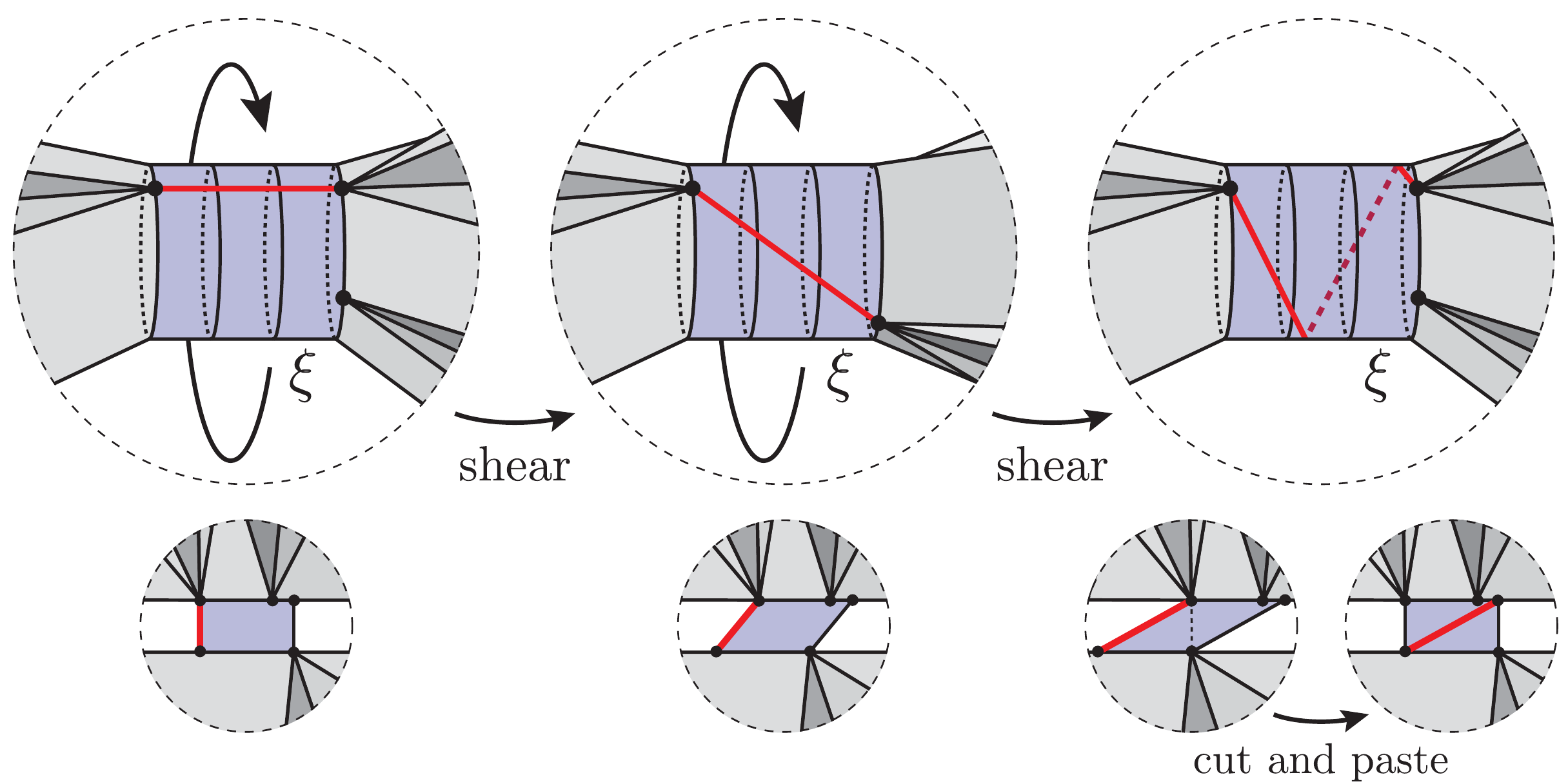}
\caption{A full shear in the cylinder $\xi$, both on the surface and on a polygonal presentation.}
\label{fig:cyl_shears}
\end{figure}

Using these deformations, we can realize twists on $\mathsf{C}(\sing, \spin)$ inside the geometric monodromy group.

\begin{lemma}\label{lem:twist_in_monodromy}
There is an inclusion $\Gamma(\sing, \spin) \le \mathcal{G}(\sing, \spin).$
\end{lemma}
\begin{proof}
Let $c$ be a curve of $\mathsf{C}(\sing, \spin)$; by Proposition \ref{prop:prototypes}, it is realized as the core curve of a cylinder $\xi$ on the prototype $(X, f, \omega)$.

By twisting on $\xi$, we see that $u_t(\xi) \cdot (X, f, \omega)$ for $t \in [0, m]$ gives a path $\gamma$ from $(X, f, \omega)$ to $(X,  f \circ T(c)^{-1}, \omega)$ \eqref{eqn:twist}. Moreover, since no zero of $\omega$ is contained in the interior of $\xi$ and the bordered surface $X \setminus \text{int}(\xi)$ is fixed throughout the shearing process, we see that the surface 
\[u_t(\xi) \cdot (X, f, \omega) \in \HT(\sing) \text{ for all } t.\]
Thus the projection of $\gamma$ to $\HM(\sing)$ is a loop from $(X, \omega)$ to itself. Since the mapping class group acts by precomposition (by inverses) with the marking, this demonstrates that $T(c) \in \mathcal{G}(\sing, \spin).$

Repeating this for each curve of $\mathsf{C}(\sing, \spin)$ gives the desired inclusion.
\end{proof}

Our final goal is to understand the relation of both groups with $\Mod(S)[\phi]$. We begin by considering a special case, which is an easy consequence of our definitions together with Theorem \ref{thm:Salter_generation}.

\begin{prop}\label{prop:gen_easycase}
If $r$ is odd, then 
\[\mathcal{G}( (r^{(2g-2)/r}),\spin )=\Mod(S)[\phi].\]
If $r$ is even, then $\mathcal{G}(\sing, \spin)$ is a finite index--subgroup of $\Mod(S)[\phi]$.
\end{prop}

Before we can prove the Proposition, we record a quick inequality which will be used to ensure that there is enough space on the surface to perform the required manipulations.

\begin{lemma}\label{lem:ineq}
Suppose that $g$ and $r$ are positive integers so that $g \ge 5$, $r < g-1$, and $r$ divides $2g-2$. Then
\[r < g-2.\]
\end{lemma}
\begin{proof}
Suppose towards contradiction that $r=g-2$; but now both $r$ and $g-1$ divide $2g-2$, and since $g-1$ and $g-2$ are coprime, it must be that
\[(g-1)(g-2) = \text{lcm}(g-1, g-2) \le 2g-2\]
which is equivalent to the inequality
\[g^2 - 5g+4 \le 0.\]
But this happens only for $g$ between $1$ and $4$, and we have assumed that $g \ge 5$, a contradiction.
\end{proof}

\begin{proof}[Proof of Proposition \ref{prop:gen_easycase}]
Observe that by Lemma \ref{lem:twist_in_monodromy} and Corollary \ref{coro:monodromy_containment}, we have that
\[\Gamma(\sing, \spin) \le \mathcal{G}(\sing, \spin) \le \Mod(S)[\phi].\]
Therefore in order to prove the statement, we need only prove that $\Gamma(\sing, \spin)=\Mod(S)[\phi]$ (or when $r$ is even, is of finite index). This reduces to checking the hypotheses of Theorem \ref{thm:Salter_generation}.

\begin{enumerate}
\item[(0)] Observe that by construction, $\mathsf{C}((r^{(2g-2)/r}),\spin)$ is a connected, filling network. Moreover, in this special case the definition of the curve system reduces to
\[\mathsf{C}((r^{(2g-2)/r}),\spin) = \mathsf{A} \cup \{b_i : i \equiv 3 \mod r\}.\]
\item Since each curve is realized as a cylinder on the prototype $(X, f, \omega)$, we see by Lemma \ref{lem:cyl_wn0} that $\phi(\overrightharp{c} ) = 0$ for each $c \in \mathsf{C}((r^{(2g-2)/r}),\spin ).$
\item The reader can verify that in both of the labeling schemes of Definition \ref{def:curvesystem}, the collection
\[ \{ b_3, a_2', a_3, a_3', \ldots, a_{r+2}, a_{r+2}', a_{r+3}\} \subset \mathsf{C}((r^{(2g-2)/r}),\spin) \]
is arranged in the $D_{2r+3}$ configuration and the labeled curve $b_{r+3}$ corresponds to the $a_{r+1}$ curve of the $D_{2r+3}$ configuration. Observe that by Lemma \ref{lem:ineq}, we have $r+3 \le g$ and so this configuration fits on the surface. See Figure \ref{fig:curvesys_Dtwiddle}.
\item The curve $b_2$ corresponds to $\Delta_0$ in the $D_{2r+3}$ configuration, and $i(b_2, a_2)=1$.
\item If the curves are labeled as in Figure \ref{fig:curvelabels3}, is clear by inspection that the subnetwork $\mathsf{C}(\sing, \spin) \setminus \{a_2\}$ is a connected arboreal network which fills $S \setminus b_2$.
\end{enumerate}

\begin{figure}[ht]
\centering
\includegraphics[scale=.5]{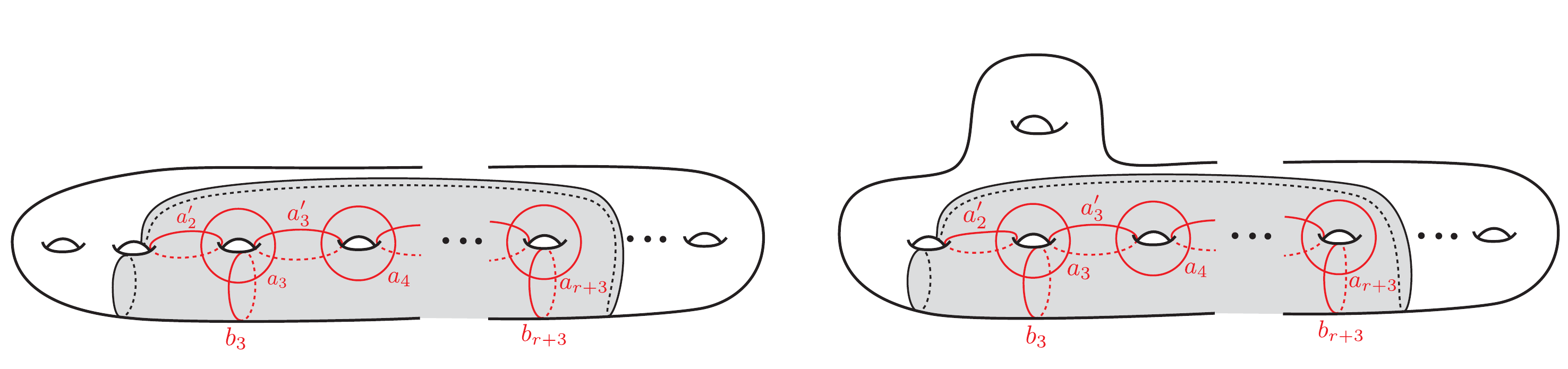}
\caption{The sets of curves in the $D_{2r+3}$ configuration and $a_{r+1}$, together with the subsurfaces which they fill.}
\label{fig:curvesys_Dtwiddle}
\end{figure}

When the curves are labeled as in Figure \ref{fig:curvelabels12}, the resulting subnetwork $\mathsf{C}(\sing, \spin) \setminus \{a_2\}$ is not a filling network for $S \setminus b_2$ (and indeed, is not even connected). To rectify this issue, we enhance our generating set by constructing a curve $c$ such that $\phi(\overrightharp{c})=0$, $T(c) \in \Gamma(\sing, \spin)$, and so that
\[\mathsf{C}':= \mathsf{C}((r^{(2g-2)/r}), \spin) \cup \{c \}\]
 is a network satisfying all of the hypotheses of Theorem \ref{thm:Salter_generation}. Once we have constructed such a $c$, then we will have that
\[G(\phi)
= \left \langle T(c) : c \in \mathsf{C}' \right \rangle
= \Gamma((r^{(2g-2)/r}),\spin) 
\le \Mod(S)[\phi]\]
where $G(\phi) = \Mod(S)[\phi]$ when $r$ is odd and is of finite index when $r$ is even. In either case, this will allow us to conclude our proof.

To find this curve, we will use a new, auxiliary curve $c_{(3,3+r)}$ which is the ``top'' boundary component of the chain
\[(a_3, a_3', a_4, \ldots, a_{r-1}', a_r).\]
See Figure \ref{fig:cij}. We claim (and prove below, see Proposition \ref{prop:heuristic}) that $c_{(3, 3+r)}$ is in the $\Gamma(\sing, \spin)$ orbit of $b_3$. 

Allowing this, let $S_{\mathsf{A}}$ denote the subsurface filled by $\A$. By the Birman--Hilden theory (\S\ref{sec:hyperelliptic}, see also \S\ref{sec:braids}), the image of $c_{(3,3+r)}$ encircles the $5\ith$ through $(6+2r)\ith$ branch points of $S_{\A}$ mod its obvious hyperelliptic involution. This curve can then be braided so that it encircles the $1^{\text{st}}$, $2^{\text{nd}}$, and through $(2g-2r+1)^{\text{st}}$ through $2g\ith$ branch points.

Lifting the braid action up to the action of the hyperelliptic mapping class group yields a curve $c$ which is in the $\GA$--orbit of $c_{(3,3+r)}$, and hence the $\Gamma(\sing, \spin)$ orbit of $b_3$. In particular, by Lemma \ref{lem:ineq} we have that
\[2g-2r+1 > 5\]
and so $c$ does not intersect $b_2$. See Figure \ref{fig:c_curve}.

Now since $c$ is in the $\Gamma(\sing, \spin)$ orbit of $b_3$, we have $T(c) \in \Gamma(\sing, \spin)$. Note that since 
\[\Gamma(\sing , \spin) \le \Mod (S)[\phi ]\]
and $\phi ( \overrightharp{b}_3)=0$, it must be that $\phi(\overrightharp{c})=0$.

\begin{figure}[ht]
\centering
\includegraphics[scale=.8]{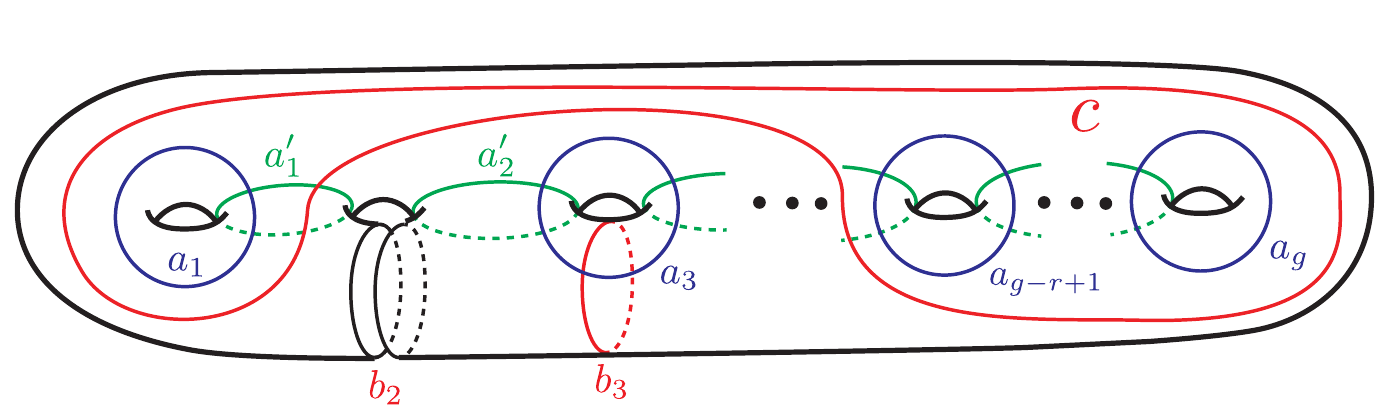}
\caption{Completing $\mathsf{C}(\sing, \spin) \setminus \{a_2\}$ to an arboreal, filling network on $S \setminus \{b_2\}$.}
\label{fig:c_curve}
\end{figure}

The new collection of curves $\mathsf{C}' $ is still a connected, filling network which contains the appropriate $D_{2r+3}$ configuration, and the subnetwork 
$\mathsf{A} \cup \{ b_3, c\}$
is a connected arboreal subnetwork which fills $S \setminus \{b_2\}$.

Therefore in either case, we can apply Theorem \ref{thm:Salter_generation} to deduce that $\Gamma(\sing, \spin)$ is either $\Mod(S)[\phi]$ (if $r$ is odd) or a finite--index subgroup thereof (if $r$ is even).
\end{proof}

\subsection{The Euclidean algorithm on simple closed curves}\label{sec:Euclidean}

In order to complete the proof of our main theorem, we need to extend Proposition \ref{prop:gen_easycase} to general partitions $\sing$ of $2g-2$. In particular, we need to show that we can recover the Dehn twists in the curves of $\mathsf{C}((r^{(2g-2)/r}), \spin)$ by twisting in $\mathsf{C}(\sing, \spin)$.

Let $(X, f, \omega)$ be the prototype constructed above for the curve system $\mathsf{C}(\sing, \spin)$. While the (framed lifts of the) curves of $\mathsf{C}((r^{(2g-2)/r}), \spin)$ all evaluate to $0$ under the $r$--spin structure $\phi$ induced by $(X, f, \omega)$ (by homological coherence, Lemma \ref{lem:hom_coh}), there is {\em a priori} no reason that we should expect to be able to twist in them.

It is tempting to speculate that every curve $c$ such that $\phi(\overrightharp{c})=0$ is realized as a cylinder on some $(X', f', \omega')$ living in the same component of $\HT(\sing)^{\spin}$ as our prototype $(X, f, \omega)$, but this is not the case.

For example, consider the stratum $\HM(1, 2g-3)$ for $g \ge 4$. Theorem \ref{thm:monodromy} implies that its monodromy group is the entire mapping class group, and in particular contains a Dehn twist about a separating curve $c$ whose complementary subsurfaces $S \setminus c$ both have genus at least $2$. However, if $c$ were realized as a cylinder on some abelian differential $(X, \omega) \in \HM(1, 2g-3)$ then the induced flat cone metrics on the pieces of $X \setminus N_\varepsilon(c)$ (where $N_\varepsilon(c)$ denotes a flat $\varepsilon$--neighborhood of $c$) would have cone angles $4\pi$ and $(4g-4)\pi$ with flat geodesic boundary of zero curvature. But this contradicts the Gauss--Bonnet theorem, and so $c$ can never be realized as a cylinder on a surface in $\HM(1, 2g-3)$.

We will therefore put aside our geometric interpretation of the monodromy group for the moment and instead appeal to perhaps the most established method of reducing to a greatest common divisor. That is to say, we are going to apply the Euclidean algorithm to the curve system $\mathsf{C}(\sing, \spin)$.

In order to use the Euclidean algorithm, one must first be able to ``add'' and ``subtract'' the quantities in question. In Proposition \ref{prop:curve_addsub} below, we demonstrate how to model the operations of arithmetic with simple closed curves by employing manipulations which are reminiscent of those arising in the derivation of the Lickorish generators from the Humphries generators \cite{Humphries_gen}.

Recall that if the curves of $S$ are labeled as in Figure \ref{fig:curvelabels12}, then we denote by $\mathsf{A}$ the set of all curves labeled by some $a_i$ or $a_i'$. Then define
\[\GA = \langle T(a) : a \in \mathsf{A} \rangle.\]
Observe that no matter the pair $(\sing, \spin)$, we have that $\mathsf{A} \subset \mathsf{C}(\sing, \spin)$ and hence $\GA < \Gamma(\sing, \spin)$.

\begin{prop}[addition and subtraction]\label{prop:curve_addsub}
Let the curves of $S_g$ be labeled as in Figure \ref{fig:curvelabels} and suppose that $x \le g-2$. Then
\[ T({b_{i+2x}})  \in \langle T({b_i}), T({b_{i+x}}), \GA \rangle\]
where indices are taken mod $2g-2$. Analogously,
\[ T({b_{i}}) \in \langle T({b_{i+x}}), T({b_{i+2x}}), \GA \rangle.\]
\end{prop}

In order to prove the first claim of Proposition, we will find some $f \in \langle T({b_i}), T({b_i+x}), \GA \rangle$ which takes one of $\{b_i, b_{i+x}\}$ to $b_{i+2x}$ and then apply the following standard fact:

\begin{fact}\label{lem:change_of_coords}
If $c$ is any simple closed curve on $S$ and $f \in \Mod(S)$, then $f T(c) f^{-1} = T({f(c)}).$
\end{fact}

The construction of the required element uses a detailed analysis of the group $\GA$ and its action on certain auxiliary curves.
In the interest of the reader, we will only give a schematic of its construction in a specific (but representative) case and defer the full proof to Appendix \ref{app:curves} (see in particular Proposition \ref{prop:heuristic} and the proof of Proposition \ref{prop:curve_addsub} at the very end of the Appendix).

\begin{proof}[Sketch of Proposition \ref{prop:curve_addsub}]
Suppose that $(\sing, \spin)$ and $g$ determine the labeling scheme pictured in Figure \ref{fig:curvelabels12}, and that 
\[1 \le i< i+x < i+2x \le g.\]
In this case, we define an auxiliary type of curve, $c_{(i,j)}$, which is one of the boundary curves of an $\varepsilon$--neighborhood of $a_i \cup a_i' \cup a_{i+1} \cup \ldots \cup a_{j-1}' \cup a_j$. See Figure \ref{fig:cij}.

\begin{figure}[ht]
\centering
\includegraphics[scale=.6]{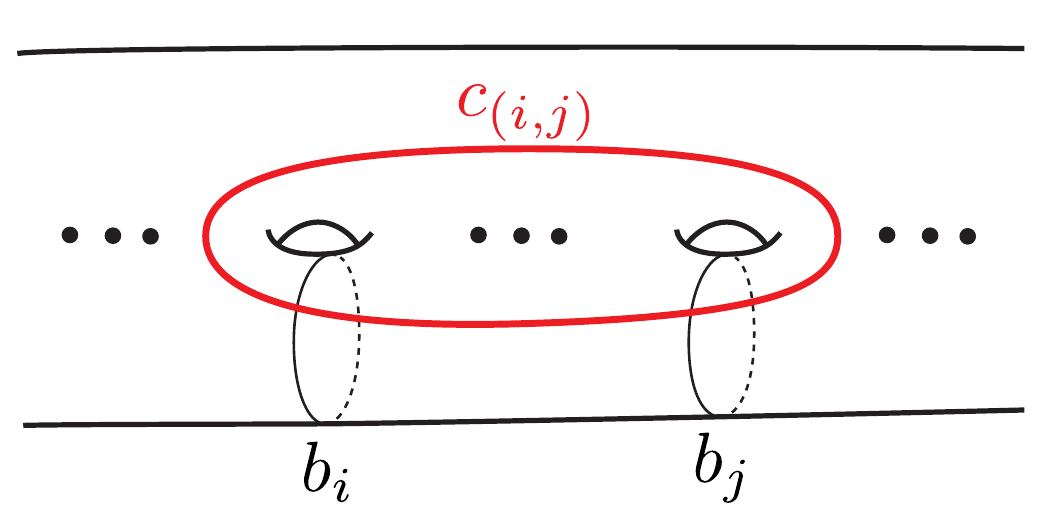}
\caption{The curve $c_{(i,j)}$.}
\label{fig:cij}
\end{figure}

The main idea of the proof is to understand the structure of the $\GA$ and $\langle \GA, b_i \rangle$--orbits of both the $c_{(i,j)}$ and the $b_i$ curves. These orbits are investigated in detail in Appendix \ref{app:curves}, but in our case we can distill the relevant results into the following

\begin{heur}
Any group containing both $\GA$ and  two of $\{T(b_i), T(b_{j}), T(c_{(i,j)}) \}$ contains the third.
\end{heur}

That is, if $\{u, v, w\} = \{b_i, b_{j}, c_{(i,j)} \}$, we have 
\[T(u) \in \left\langle \GA, T(v), T(w) \right\rangle \]
With this rule, we can now sketch the construction of an $f$ taking $b_i$ to $b_{i+2x}$.

Applying the heuristic, we observe that we have
\begin{equation}\label{eqn:sketch1}
T(c_{(i, i+x)}) \in \left\langle \GA, T(b_i), T(b_{i+x}) \right\rangle.
\end{equation}
Now $\GA$ acts transitively on the set of $c_{(i,j)}$ with fixed difference $j-i$ (Lemma \ref{lem:GA_action_cij}), so there is an element of $\GA$ which takes $c_{(i, i+x)}$ to $c_{(i+x, i+2x)}$ and hence
\begin{equation}\label{eqn:sketch2}
T(c_{(i+x, i+2x)}) \in \left\langle \GA, T(c_{(i, i+x)}) \right\rangle.
\end{equation}
By applying the heuristic again, we see that
\begin{equation}\label{eqn:sketch3}
T(b_{i+2x}) \in \left\langle \GA, T(b_{i+x}), T(c_{(i+x, i+2x)}) \right\rangle.
\end{equation}
Combining \eqref{eqn:sketch1}, \eqref{eqn:sketch2}, and \eqref{eqn:sketch3} then yields the desired containment. See Figure \ref{fig:addition_sketch} for an overview of this construction.
\end{proof}

\begin{figure}[ht]
\centering
\includegraphics[scale=.65]{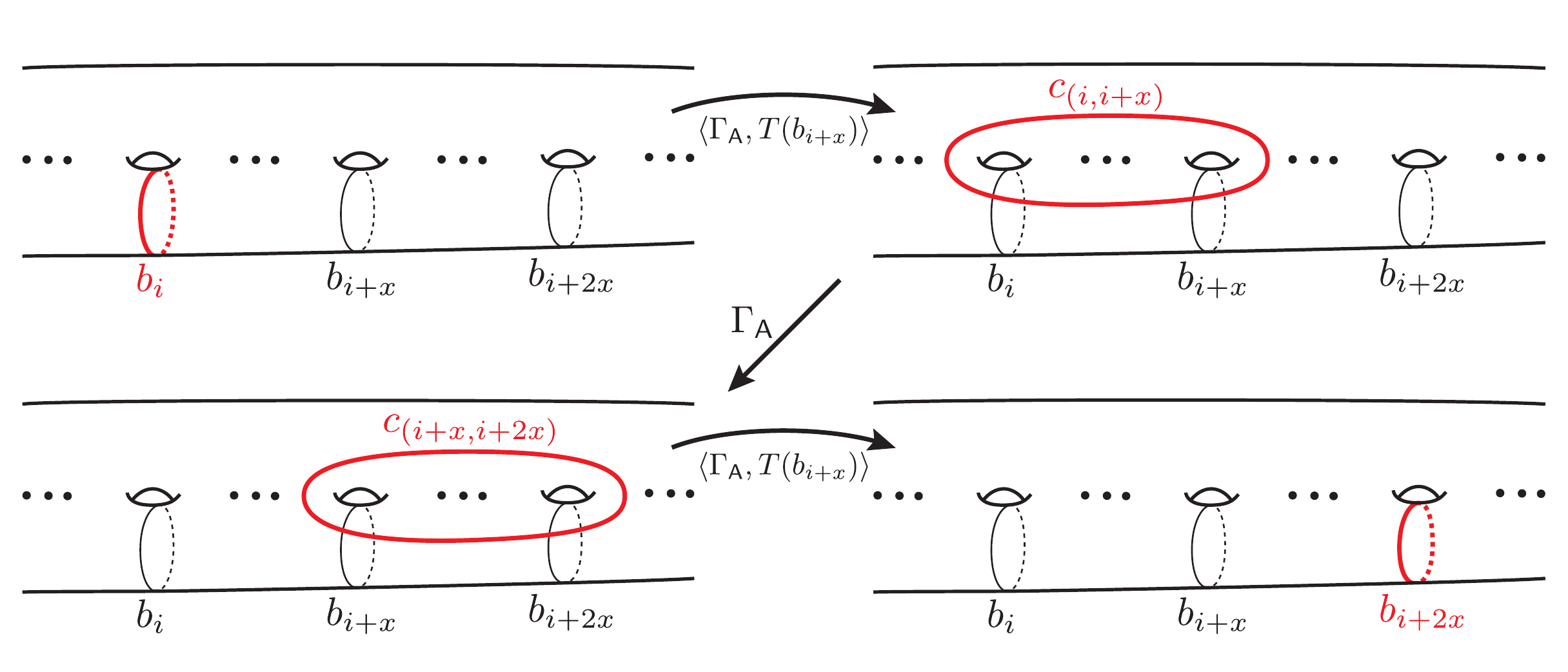}
\caption{Obtaining the twist on $b_{i+2x}$ from the twists on $b_{i}$ and $b_{i+x}$.}
\label{fig:addition_sketch}
\end{figure}

Of course, in the general case, one must take into account the different curve labeling schemes appearing in Figure \ref{fig:curvelabels}. Moreover, there is no guarantee that all of the curves $\{ b_i, b_{i+x}, b_{i+2x}\}$ will lie on the lower half of the surface (i.e., that $i+2x \le g$). In order to deal with the latter issue, we will need to understand how to ``go around the ends of the surface,'' the nuances of which account for a significant portion of the technical difficulty of the proof.

Assuming these simple closed curve analogues of addition and subtraction, we can iteratively apply the Euclidean algorithm to the curve system $\mathsf{C}(\sing, \spin)$
and reduce it to the case
considered in Proposition \ref{prop:gen_easycase}.

\begin{thm}\label{thm:Euclidean_algorithm}
Let $g \ge 4$ and $\sing$ a partition of $2g-2$. If $\gcd(\sing)=r$ is even, choose $\spin \in \{ \even, \odd \}.$ Then
\[\Gamma(\sing, \spin) = \Gamma((r^{(2g-2)/r}), \spin).\]
\end{thm}
\begin{proof}
In order to complete $\mathsf{C}(\sing, \spin)$ to $\mathsf{C}((r^{(2g-2)/r}), \spin)$, we pass through a filtration by intermediate partitions of $2g-2$, each related to the subsequent by an application of the Euclidean algorithm.

To that end, set $r_j = \gcd(k_1, \ldots, k_j)$,
\[d_j = \left( \sum_{i=1}^j k_i \right) / r_j,\]
and define
\[\sing_j = \left( r_j^{d_j}, k_{j+1}, \ldots, k_n \right)\]
for each $j=1, \ldots, n$. Note that $r_1 = k_1$ and $\sing_1 = \sing$, while
\[r_n = r = \gcd(\sing) \text{ and }\sing_n = \left( r^{(2g-2)/r} \right).\]
Observe also that $d_j \ge 1$ for all $j$.

Therefore, to prove the Theorem it suffices to show for each $j = 1, \ldots, n-1$ that
\[\Gamma(\sing_j, \spin) = \Gamma(\sing_{j+1}, \spin).\]

Observe that since the orders of zeros $k_i$ were assumed to be given in increasing order, we know that $r_j \le k_{j+1}$. To begin, we first run the Euclidean algorithm on $r_j$ and $k_{j+1}$; that is, we find a sequence of non-negative integers $Q_\ell$ and $R_\ell$ such that $R_\ell < Q_\ell$ for all $\ell$ and
\begin{equation}\label{eqn:Euclidean}
\begin{array}{rcl}
k_{j+1} &= &Q_1 r_j + R_1\vspace{2pt} \\
r_j &= &Q_2 R_1 + R_2\vspace{2pt}  \\
R_1 &= &Q_3 R_2 + R_3 \\
&\vdots \\
R_{N-1} &=& Q_{N+1} R_{N} + 0
\end{array}
\end{equation}
Then the Euclidean algorithm certifies that $R_{N} = \gcd(r_j, k_{j+1}) = r_{j+1}$.

To ease our notational burden, we define the following indices:
\begin{equation}\label{eqn:yi}
\begin{array}{rcl}
y_0 &= & 3 + \sum_{i=1}^{j} k_i  \vspace{2pt} \\
y_1 &= & y_0 + Q_1 r_j \vspace{2pt} \\
y_2 &= & y_1 - Q_2 R_1 \vspace{2pt} \\
y_3 &= & y_2 + Q_3 R_2 \\
&\vdots \\
y_{N+1} & = & y_N +(-1)^{N} Q_{N+1} R_{N}
\end{array}
\hspace{10pt}
\text{ and }
\hspace{10pt}
\begin{array}{rcl}
y_0' &= & 3 + \sum_{i=1}^{j+1} k_i  \vspace{2pt} \\
y_1' &= & y_0 + (Q_1-1) r_j \vspace{2pt} \\
y_2' &= & y_1 - (Q_2-1) R_1 \vspace{2pt} \\
y_3' &= & y_2 + (Q_3-1) R_2 \\
&\vdots \\
y_{N+1}' & = & y_N +(-1)^{N} (Q_{N+1}-1) R_{N}
\end{array}
\end{equation}
Now by construction of this recursive labeling scheme, we have
\[y_\ell ' = y_\ell + (-1)^{\ell} R_{\ell-1}\]
so for all $\ell \ge 1$,
\begin{equation}\label{eqn:Ri}
| y_\ell - y_{\ell-1}' |=
\left| ( y_{\ell-1} + (-1)^{\ell-1} Q_{\ell} R_{\ell-1}) - (y_{\ell-1} + (-1)^{\ell-1} R_{\ell-2} )\right|
= R_{\ell-2} - Q_{\ell} R_{\ell-1} = R_{\ell}.
\end{equation}

We can now realize the series of equations appearing in \eqref{eqn:Euclidean} as a sequence of curve diagrams by repeated application of Proposition \ref{prop:curve_addsub}, an example of which appears in Figure \ref{fig:Euclidean_example}. In order to keep our notation readable, we will denote the Dehn twist in $b_i$ by $T(i)$ for the rest of the proof.

Since the $k_j$ are assumed to be ordered from least to greatest and $r \notin \{2g-2, g-1\}$, we have that 
\[r_1 \le k_1 \le g-2.\]
As the Euclidean algorithm mandates that successive remainders always decrease (i.e., $R_\ell < R_{\ell - 1}$), we see that the $x$ value added to and subtracted from indices never exceeds $g-2$, thereby justifying our use of Proposition \ref{prop:curve_addsub}.

First, we note that by construction, both $T(y_0)$ and $T(y_0')$ are elements of $\Gamma(\sing_j, \spin).$ Moreover, since $d_j \ge 1$ (as it is the quotient of the partial sum $\sum_{i=1}^j k_i$ by $r_j$), we have that 
\[T(y_0 - r_j) = T(3 + r_j(d_j-1))\]
is also an element of $\Gamma(\sing_j, \spin).$
Therefore after applying the first half of Proposition \ref{prop:curve_addsub} (addition) with $x = r_j$ for $Q_1-1$ and $Q_1$ times, respectively, we see that
\[T(y_1'), \, T(y_1) \in \Gamma(\sing_j, \spin).\]
But now since $T(y_0')$ and $T(y_1)$ are both in the group, and we have from \eqref{eqn:Ri} that $y_0' - y_1 = R_1$, we may apply the second half of Proposition \ref{prop:curve_addsub} (subtraction) with $x = R_1$ to deduce that both
\[T(y_2'), \, T(y_2) \in \Gamma(\sing_j, \spin).\]
Likewise, the difference between $y_2$ and $y_1'$ is $R_2$, so again applying Propostion \ref{prop:curve_addsub} (addition) with $x=R_2$ for $Q_3-1$ and $Q_3$ steps yields
\[T(y_3'), \, T(y_3) \in \Gamma(\sing_j, \spin).\]
Continuing in this way, alternating between addition and subtraction of indices, we can work our way through the series of equations in \eqref{eqn:Euclidean} until terminating at $T(y_{N+1})$.
\footnote{Since $R_{N+1}=0$, we must have that $y_{N+1} = y_{N}'$.}
See Figure \ref{fig:Euclidean_example}.

\begin{figure}[ht]
\centering
\includegraphics[scale=.6]{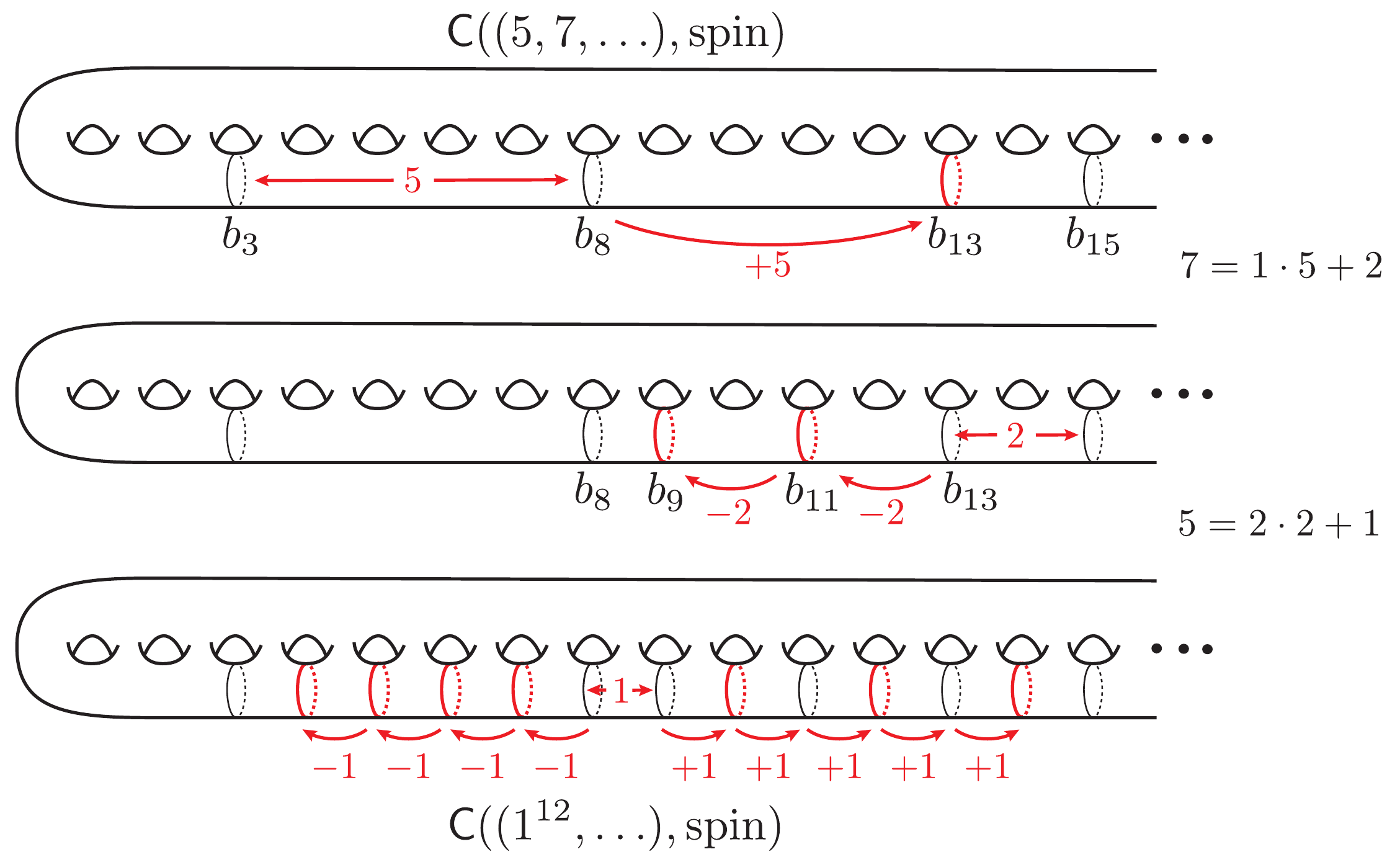}
\caption{Using the Euclidean algorithm to show $\Gamma((5, 7, \ldots), \spin) = \Gamma((1^{12}, \ldots), \spin).$}
\label{fig:Euclidean_example}
\end{figure}

In particular, both $T(y_N)$ and $T(y_{N-1}')$ are in $\Gamma(\sing_j, \spin)$, but by \eqref{eqn:Euclidean} and \eqref{eqn:Ri} we have that
\[| y_N - y_{N-1}'| = R_N = \gcd(r_j, k_{j+1}) = r_{j+1}.\]
Therefore, by applying Proposition \ref{prop:curve_addsub} to $y_N$ and $y_{N-1}=y_N - r_j$ with $x = r_j$ we see that
\[T(y_N - 2r_j) \in \Gamma(\sing_j, \spin).\]
Applying Proposition \ref{prop:curve_addsub} to $y_{N}-r_j$ and $y_{N}-2r_j$ we likewise have 
\[T(y_N - 3r_j) \in \Gamma(\sing_j, \spin).\]
Repeatedly applying Proposition \ref{prop:curve_addsub} in the same way, we see that
\[T(3+pr_{j+1}) \in \Gamma(\sing_j, \spin)\]
for any $0 \le p \le d_{j+1}$, and hence
\[\Gamma(\sing_j, \spin) = \Gamma(\sing_{j+1}, \spin).\]

By iterating the above procedure on $j$, it follows that
\[\Gamma(\sing, \spin) = \Gamma(\sing_2, \spin) = \ldots = \Gamma( \sing_n, \spin) = \Gamma (( r^{(2g-2)/r} ), \spin )\]
and so the Theorem is proved.
\end{proof}

Combining the above statements, we can now give a short proof of our main theorems.

\begin{proof}[Proof of Theorems \ref{thm:components} and \ref{thm:monodromy}]
Suppose $\sing = (k_1, \ldots, k_n)$ is given, and if $r = \gcd(\sing)$ is even, that $\spin \in \{ \even, \odd\}$. Suppose also that $r \notin \{2g-2, g-1\}.$ Let $(X, f, \omega)$ be the (marked) prototype for the pair $(\sing, \spin)$, and let $\phi \in \Phi_r$ be its induced $r$--spin structure.

By Lemma \ref{lem:twist_in_monodromy} and Corollary \ref{coro:monodromy_containment}, we have that
\[\Gamma(\sing, \spin) \le \mathcal{G}(\sing, \spin) \le \Mod(S)[\phi].\]
Combining Theorem \ref{thm:Euclidean_algorithm} and Proposition \ref{prop:gen_easycase}, it follows that if $r$ is odd, then 
\[\Gamma(\sing, \spin) = \Gamma ( (r^{(2g-2)/2}),\spin) = \Mod(S)[\phi].\]
Similarly, if $r$ is even, then $\Gamma(\sing, \spin)=  \Gamma ( (r^{(2g-2)/2}), \spin)$ is a finite index subgroup of $\Mod(S)[\phi]$. Therefore the same conclusions must hold for $\mathcal{G}(\sing, \spin)$. This concludes the proof of Theorem \ref{thm:monodromy}.

Consider now the action of $\Mod(S)$ on the set of connected components of $\HT(\sing)^{\spin}$. By Theorem \ref{thm:KZ_class}, every connected component must contain some $(X, g, \omega)$, where $g: S \rightarrow X$ is a marking, and hence by Theorem \ref{thm:Mod_rspin_action}, the action of $\Mod(S)$ on the set of components of $\HT(\sing)^{\spin}$ is seen to be transitive. Therefore by the orbit--stabilizer theorem the number of connected components is the same as the index of $\mathcal{G}(\sing, \spin)$ inside of $\Mod(S)$. Applying Corollary \ref{coro:stab_index} (which counts the number of $r$--spin structures of given parity) finishes the proof of Theorem \ref{thm:components}.
\end{proof}

We can also deduce the image of $\mathcal{G}(\sing, \spin)$ under the symplectic representation.

\begin{proof}[Proof of Corollary \ref{coro:symp_monodromy}]
Let $\psi: \Mod(S) \rightarrow \Sp(2g, \ZZ)$ denote the standard symplectic action of a mapping class on homology, and suppose $\sing$ is such that $r = \gcd(\sing) \notin\{2g-2, g-1\}$. If $r$ is even, also choose $\spin \in \{\even, \odd\}$.

By Theorem \ref{thm:monodromy}, the geometric monodromy group $\mathcal{G}(\sing, \spin)$ is either the stabilizer of an $r$--spin structure $\phi$ (for $r$ odd) or is a finite--index subgroup thereof (for $r$ even).

If $r$ is odd, then by Lemma \ref{lem:Salter_sympl}, $\Mod(S)[\phi]$ surjects onto the entire symplectic group. When $r$ is even, the lemma together with Theorem \ref{thm:Euclidean_algorithm} states that $\psi ( \Gamma(\sing, \spin))$ is the stabilizer $\Sp(q)$ of the quadratic form $q=q_{\phi^{\otimes r/2}}$. Moreover, since $\Mod(S)[\phi]$ preserves $\phi$, it preserves $\phi^{\otimes r/2}$ and therefore $q$, so its image under $\phi$ is also $\Sp(q)$. But now
\[\Gamma(\sing, \spin) \le \mathcal{G}(\sing, \spin) \le \Mod(S)[\phi]\]
and hence it must be that $\psi ( \mathcal{G}(\sing, \spin)) = \Sp(q)$.
\end{proof}

\section{Remarks and further directions}\label{sec:onwards}

It would be interesting to understand the robustness of the relationship between cylinder shears and monodromy groups. By our choice of prototype surface in \S\ref{sec:prototypes}, we could deduce that Dehn twists in the prototype's cylinders generated the entire monodromy group (or a finite--index subgroup for $r$ even). Our combinatorial arguments hinge on the specific structure of the curve system $\mathsf{C}(\sing, \spin)$, but the result may be more general.

\begin{ques}
Let $\sing$ be any partition of $2g-2$ and if $\gcd(\sing)$ is even, choose $\spin \in \{\even, \odd\}$. If $(X, \omega)$ is any square--tiled surface in $\HM(\sing)^{\spin}$, do the Dehn twists in the cylinders of $(X, \omega)$ generate $\mathcal{G}(\sing, \spin)$? What if $(X, \omega)$ is an arbitrary differential in $\HM(\sing)^{\spin}$?
\end{ques}

Parallel to our main theorems, one could also investigate the components of strata of quadratic differentials. 
Walker began an investigation into these questions in \cite{Walker_thesis}, \cite{Walker_components}, and \cite{Walker_groups}, but her results are incomplete and techniques generally insufficient (see \S\S\ref{sec:comp_rev} and \ref{sec:monodromyreview}).

Recall that if $\sing = (k_1, \ldots, k_n)$ is a partition of $4g-4$ and $0 \neq k_i \ge -1$ for each $i$, then the stratum $\QM(\sing)$ is space of all quadratic differentials with zeros (or simple poles) of degrees $k_1, \ldots, k_n$ which are not squares of abelian differentials, and $\QT(\sing)$ is the corresponding space of marked quadratic differentials.

\begin{ques}
How many connected components does $\QT(\sing)$ have? What is the geometric monodromy group of a component of $\QM(\sing)$?
\end{ques}

It is noteworthy that quadratic differentials generally do not define $r$--spin structures since their horizontal foliations generally are not orientable; one must instead define an $\RR P^1$--valued Gauss map and consider the winding number of a curve with respect to the horizontal line field.

Importantly, the action of $\Mod(S)$ on the set of these winding number functions (equivalently, roots of $K_X^{\otimes 2}$ which are not roots of $K_X$) is not fully understood, though Chen and M{\"o}ller have proven in low genus that it is not transitive \cite[Theorems 1.1, 1.2]{ChenMoller_QDinlowg}.

\bibliographystyle{utphyssorted}

\bibliography{library}

\pagebreak

\appendix

\section{Modular arithmetic and simple closed curves}\label{app:curves}

In this section, we prove Proposition \ref{prop:curve_addsub} and demonstrate more generally how one can model the operations of arithmetic with simple closed curves. The high--level idea is the same as presented in the proof sketch in \S\ref{sec:Euclidean}, and the
bulk of our proof consists of justifying and refining the heuristic used therein. For the convenience of the reader, we restate this principle below.

Recall that $c_{(i,j)}$ denotes one of the boundary curves of an $\varepsilon$--neighborhood of $a_i \cup a_i' \cup \ldots \cup a_j$, as shown in Figure \ref{fig:cij}.

\begin{heur}
Any group containing both $\GA$ and  two of $\{T(b_i), T(b_{j}), T(c_{(i,j)}) \}$ contains the third.
\end{heur}

The final form of this heuristic is Proposition \ref{prop:heuristic}, which allows us to replace $c_{(i,j)}$ with another auxiliary curve $c_{(k,\ell)}$ where 
\[\ell - k = j - i \pmod {2g-2}.\]

In Section \ref{sec:braids}, we relate the group $\GA$ to hyperelliptic mapping class groups of certain subsurfaces of $S$; this connection allows us to investigate the $\GA$ orbits of simple closed curves with relative ease. Once we have developed this machinery, we will put it to use in Section \ref{sec:heur}, where we carry out explicit computations on curves (Lemmas \ref{lem:bi_and_cij} through \ref{lem:LHS_case3}), culminating in the proofs of Propositions \ref{prop:heuristic} and \ref{prop:curve_addsub}.

Since the curve labeling schemes given in Figure \ref{fig:curvelabels} are the same away from the left--hand side of $S$, we will generally assume that we are in the case when the curves are labeled as in Figure \ref{fig:curvelabels12} and note where changes must be made on the indices if curves are labeled as in Figure \ref{fig:curvelabels3}. We will denote these scenarios by (1+2) and (3), respectively (corresponding to the cases given in Definition \ref{def:curvesystem}).

\subsection{Braiding and hyperelliptic subsurfaces}\label{sec:braids}

In order to investigate the $\GA$ action on the set of $c_{(i,j)}$ and $b_i$ curves, we must first understand the group itself. Once we have developed this geometric insight, we will use it to show that $\GA$ acts transitively on the set of $c_{(i,j)}$ curves (Lemma \ref{lem:GA_action_cij}).

Suppose for the moment that we are in case (1+2); then the set $\A$ is a chain of simple closed curves which fills a subsurface $S_\A$ of $S$ (that is, $\A$ may be ordered so that each curve $a_i$ intersects only $a_{i-1}$ and $a_{i+1}$). This subsurface has genus $g-1$ and two boundary components, and has a natural hyperelliptic involution $\iota$ which interchanges the boundary components and reverses the orientation of each curve of $\A$. Let
\[q: S_\A \rightarrow \Sigma = S_\A / \iota\]
denote the corresponding branched covering map. We will depict these coverings as in Figure \ref{fig:hypinv_subsurf12}, where the half--twists in the arcs in the figure lift to the Dehn twists on the $a_i$ curves.
\footnote{These arcs also serve as branch cuts for the covering $q: S_\A \rightarrow \Sigma$, where the sheets are the ``top'' and ``bottom'' halves of $S_\A$.}

Now by the theory of Birman and Hilden (see \S\ref{sec:hyperelliptic}), we have that the centralizer $\SMod(S_\A)$ of $\iota$ is isomorphic to the ($2g$--stranded) braid group $B$ of the quotient $\Sigma$, see \eqref{eqn:BH2}. One may verify by inspection that the Dehn twists in the group $\GA$ are lifts of the standard half--twist generators for $B$, and therefore
\begin{equation}\label{eqn:GA_BH12}
\GA = \SMod(S_\A) \cong B.
\end{equation}

\begin{figure}[h!]
\centering
\begin{subfigure}{.45\textwidth}
\includegraphics[scale=.5]{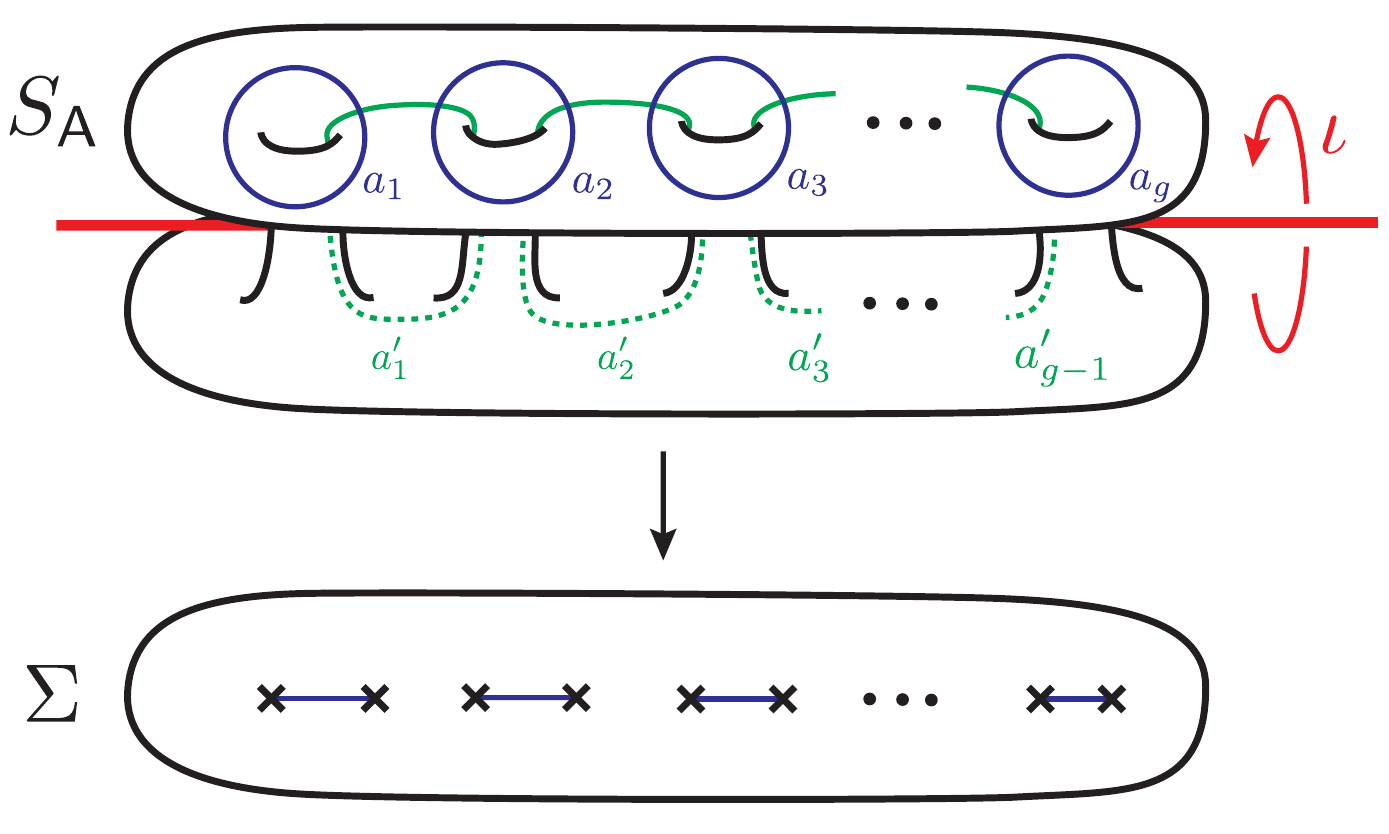}
\caption{The involution demonstrating $\GA = \SMod(S_\A)$.}
\label{fig:hypinv_subsurf12}
\end{subfigure}
\hspace{.5cm}
\begin{subfigure}{.45\textwidth}
\includegraphics[scale=.5]{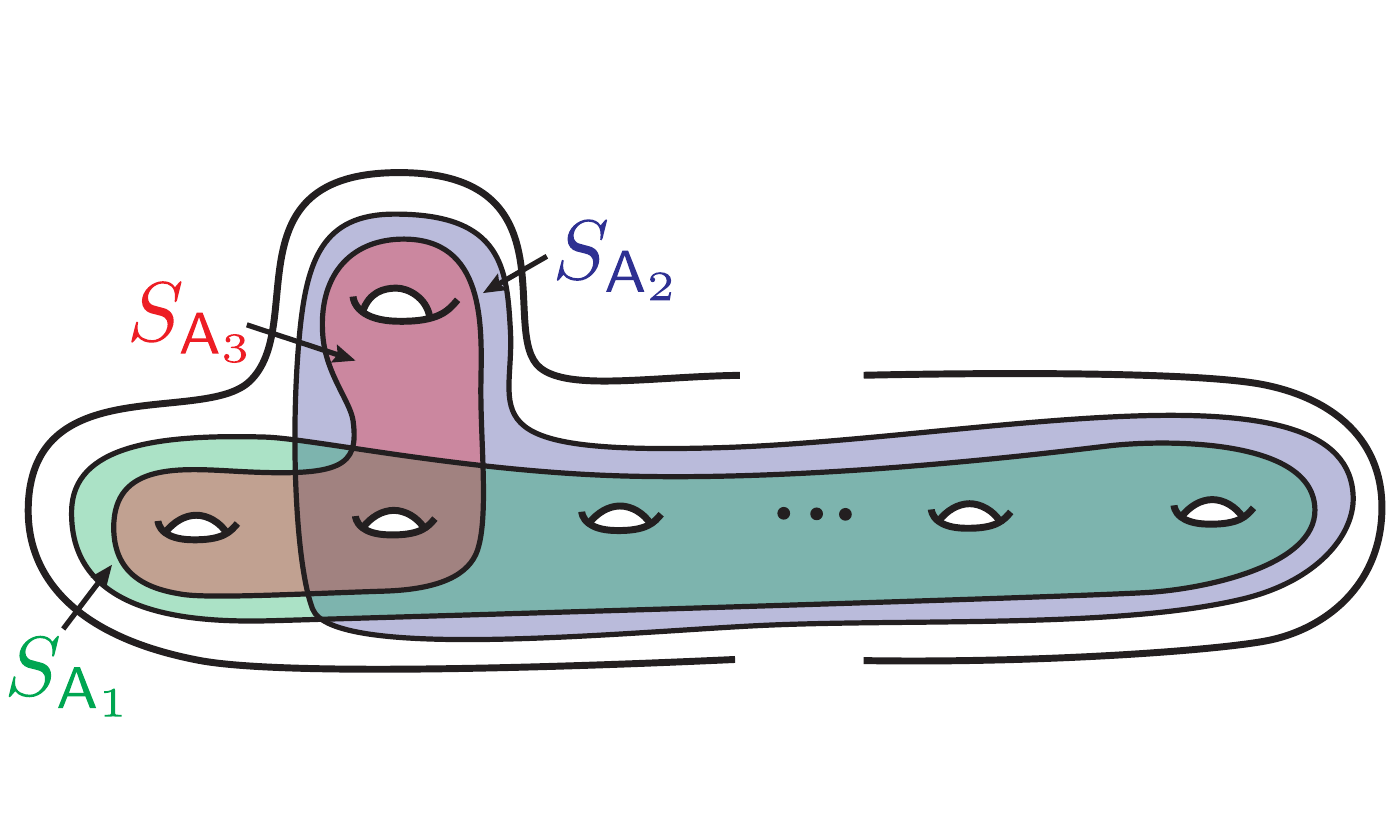}
\caption{The subsurfaces spanned by $\A_1$, $\A_2$, and $\A_3$.}
\label{fig:hypinv_subsurf3}
\end{subfigure}
\caption{The subsurfaces associated with $\GA$ and their hyperelliptic involutions.}
\label{fig:hypinv_subsurf}
\end{figure}

If instead we are in case (3), then the obvious involution $\iota$ of $S_\A$ (induced by the involution of $\A$, considered as a $1$--complex embedded in $S$) is not hyperelliptic, as it swaps $\{a_1, a_1'\}$ with $\{a_2, a_2'\}$ and hence $S_\A / \iota$ has genus $1$. Instead, we consider the following subchains of $\A$:
\[\begin{array}{lll}
\A_1 &:= &\{a_2, a_2', a_3, \ldots, a_{g-1}', a_g\}\\
\A_2 &:= &\{a_1, a_1', a_3, \ldots, a_{g-1}', a_g\}\\
\A_3 &:= &\{a_1, a_1', a_3, a_2', a_2\}
\end{array}\]
See Figure \ref{fig:hypinv_subsurf3}. The corresponding subsurfaces $S_{\A_m}$ then all admit hyperelliptic involutions which interchange their boundary components and fix the subchain $\A_m$, so as above we have that
\begin{equation}\label{eqn:GA_BH3}
\Gamma_{\A_m} = \SMod(S_{\A_m}) \cong B.
\end{equation}
where $\Gamma_{\A_m}$ denotes the subgroup of $\GA$ generated by the twists in the curves of $\A_m$ and $B$ is a braid group on $2g-2$ strands if $m=1, 2$ and on $6$ strands if $m=3$.

We will often use \eqref{eqn:GA_BH12} and \eqref{eqn:GA_BH3} to simplify our investigation of $\GA$ orbits. In particular, if $c$ is a simple closed curve on $S_\A$, then one can understand its $\SMod(S)$ orbit by projecting $c$ down to a (possibly non-simple) closed curve $q(c)$ on $\Sigma =  S_\A / \iota$. The action of the braid group $B$ on the curve $q(c)$ is now much easier to visualize, and by lifting a curve in $B\cdot q(c)$ back up to $S_\A$ we recover a curve in $\GA \cdot \, c$.

The same analysis works for curves which are not entirely contained in $S_\A$. In this case, the intersection of $c$ with $S_\A$ is a collection of pairwise disjoint simple arcs $\{\alpha_1, \ldots, \alpha_n\}$ and therefore they project to a collection of (possibly non-disjoint, non-simple) arcs on $\Sigma$. One may similarly lift the action of an element of the braid group to the action of some $g \in \GA$; then the image of the curve $c$ under the lifted element of $\GA$ may be obtained by replacing each $\alpha_i$ in $c \cap S_\A$ with $g(\alpha_i)$ (here we use the fact that the symmetric mapping class group must fix the boundary pointwise).

This trick allows us to (relatively) painlessly determine explicit elements of $\GA$ which take one specified curve on $S_\A$ to another. For example, with this framework we can easily prove that $\GA$ acts transitively on the set of the $c_{(i,j)}$ whose elements each encircle the same number of holes.

Recall that for $i < j \le g$ we define $c_{(i,j)}$ to be the ``top'' boundary curve of an $\varepsilon$--neighborhood of 
\[a_i \cup a_i' \cup a_{i+1} \cup a_{i+1}' \cup \ldots \cup a_{j-1}' \cup a_j\]
as in Figure \ref{fig:cij}. If we are in Case (3) and $i=1$, then we will alter our definition so that $c_{(1,2)}$ is the top boundary component of $S_{\A_3}$, while for each $j \ge 3$, the curve $c_{(1,j)}$ is the top boundary component of the subsurface filled by the chain
\[a_1 \cup a_1' \cup a_{3} \cup a_3' \cup \ldots \cup a_{j-1}' \cup a_j.\]
Note that in this case, $c_{(1,j)}$ does not meet $b_2$ but does meet $b_{2g-2}$.

In order to treat cases when $i < j \le g$ and $ g \le i < j$ uniformly, we will also define
\[c_{(2g-j, 2g-i)}=c_{(i,j)}\]
for all $3 \le i < j \le g$. In case (1+2), we will set
\[c_{(2g-j, 2g-2)} = c_{(2,j)},\]
while in case (3) we set\[c_{(2g-j, 2g-2)} = c_{(1, j)}.\]
Note that with these naming conventions, $c_{(i,j)}$ meets $b_i, b_{i+1}, \ldots, b_j$ in order when traversed in the counter-clockwise direction.

\begin{lemma}\label{lem:GA_action_cij}
Suppose that the curves of $S$ are labeled as in Figure \ref{fig:curvelabels} and that $i < j \le g$ and $k < \ell \le g$.
\begin{itemize}
\item In case (1+2), if $\ell - k = j - i \le g-1$, then
\[c_{(k, \ell)} \in \GA \cdot \, c_{(i,j)}.\]
\item In case (3), if $i \neq 1 \neq k$ and $\ell - k = j - i \le g-1$, then
\[c_{(k, \ell)} \in \GA \cdot\,  c_{(i,j)}.\]
If $i=1\neq k$ and $\ell-k = j-2$, then
\[c_{(k, \ell)} \in \GA \cdot \,c_{(1,j)}.\]
\end{itemize}
\end{lemma}
\begin{proof}
By the definition of the $c_{(i,j)}$, it suffices to restrict to the cases when $j, \ell \le g$. By our discussion above, we can reduce the proof of this lemma to proving that $q(c_{(k,\ell)}) \in B \cdot q(c_{(i,j)})$, where $B$ is some appropriate braid group.

In case (1+2), one can observe that the curve $c_{(i,j)}$ is always contained inside of $S_\A$ and $q(c_{(i,j)})$ is a simple closed curve which encircles the $(2i-1)^{\text{st}}$ through $(2j)\ith$ branch points. Therefore, since $\ell-k = j-i$, we see that $q(c_{(i,j)})$ and $q(c_{(k,\ell)})$ encircle the same number of branch points, and so it is easy to construct an element $b \in B$ as in Figure \ref{fig:BH_cij} which takes $q(c_{(i,j)})$ to $q(c_{(k,\ell)})$.
Lifting $b$ via the Birman--Hilden isomorphism \eqref{eqn:GA_BH12} yields an element $g \in \GA$ such that $g \cdot c_{(i,j)} = c_{(k, \ell)}$.

\begin{figure}[ht]
\centering
\includegraphics[scale=.5]{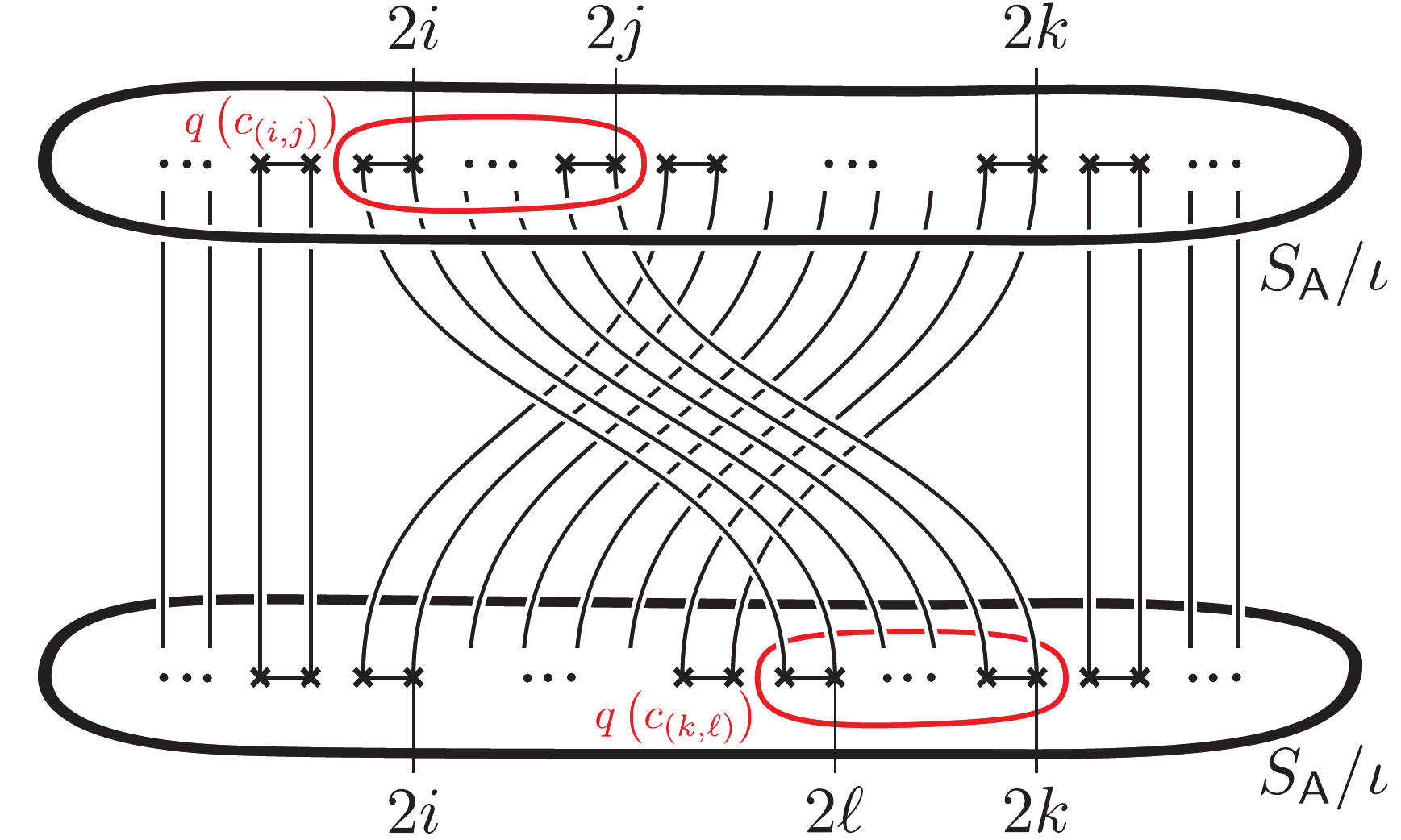}
\caption{Braiding the branch points of $S_\A / \iota$ to take $q(c_{(i,j)})$ to $q(c_{(k,\ell)})$. Such a braid lifts via the Birman--Hilden correspondence to an element of $\GA$ which takes $c_{(i,j)}$ to $c_{(k,\ell)}$.}
\label{fig:BH_cij}
\end{figure}

The proof in case (3) is similar, but now one must keep track of which subsurface(s) contain the curves in question.
\begin{enumerate}[label=(\alph*)]
\item If $i \neq 1 \neq k$, then $c_{(i,j)}$ and $c_{(k,\ell)}$ are both contained in $\A_1$, and so one may apply the same argument as in case (1+2).
\item If $i = 1 \neq k$ and $j \neq 2$, then a similar analysis with $\A_2$ in place of $\A_1$ shows that 
\[c_{(3, j+1)} \in \GA \cdot\,  c_{(1,j)}.\]
The result for general $(k ,\ell)$ with $\ell-k = j-2$ then follows from (a).
\item If $i=1$ and $j=2$, then we must be slightly more clever. To that end, let $\alpha$ denote the arc of intersection of $c_{(1,2)}$ with $S_{\A_2}$; then $q(\alpha)$ separates the first through fourth branch points of $S_{\A_2} / \iota$, and by braiding one can take $q(\alpha)$ to an arc separating off the third through sixth branch points. See Figure \ref{fig:BH_cij_arc}.

Lifting this arc up to $S_{\A_2}$ and replacing $\alpha$ with it results in a curve isotopic to $c_{(2, 4)}$, and lifting the braid via \eqref{eqn:GA_BH3} gives an element of $\GA$ taking $c_{(1,2)}$ to $c_{(2, 4)}$. Applying (a) and (b) then gives the result for general $(k ,\ell)$.
\end{enumerate}
This completes the proof of the Lemma.
\end{proof}

\begin{figure}[ht]
\centering
\includegraphics[scale=.6]{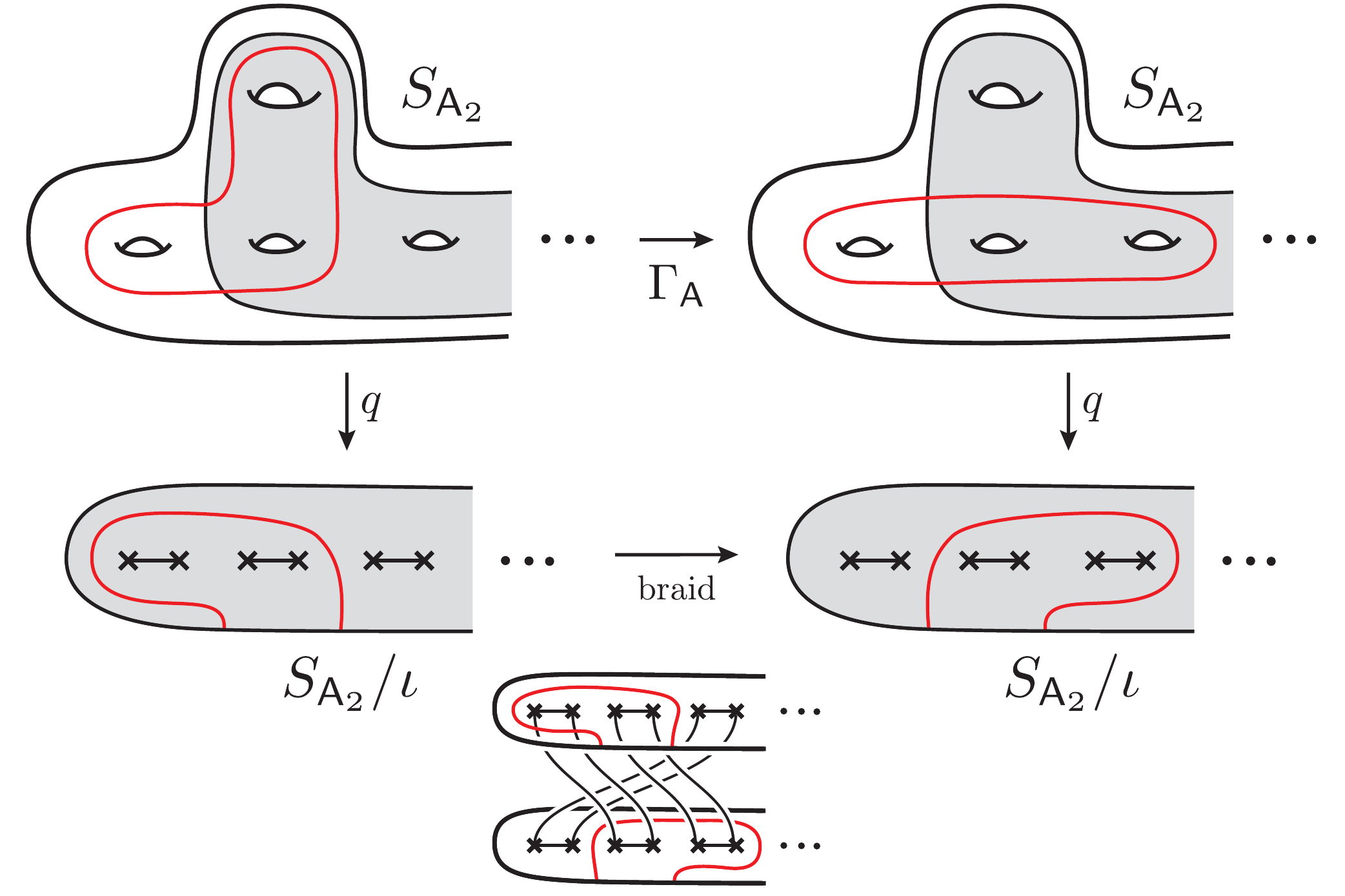}
\caption{Taking $c_{(1,2)}$ to $c_{(2,4)}$ using the Birman--Hilden theory.}
\label{fig:BH_cij_arc}
\end{figure}

\subsection{Justification of the heuristic}\label{sec:heur}

Now that we have established the conceptual basis for our analysis, in this section we state and prove a generalized version of our motivating heuristic.

\begin{prop}\label{prop:heuristic}
Let $1 \le i, j \le 2g-2$ be such that the residue class of $j-i$ mod $2g-2$ is at most $g-2$. Suppose that $2 \le k < \ell \le g$ is such that $\ell - k \equiv j - i \mod 2g-2$. Then
\[b_j \in \langle \GA, T(c_{(k, \ell)}) \rangle \cdot b_i
\,\text{ and } \,
c_{(k, \ell)} \in \langle \GA, T(b_j) \rangle\cdot  b_i.\]
\end{prop}

As noted in the main body of the text, there are multiple different regimes we must consider in our proof, depending on how the $b_i$ curves are positioned our surface. In order to define these in a uniform way, we consider the counterclockwise order on $b_i$ as a cyclic order on $\ZZ_{2g-2}$, so that
\[\ldots \prec 1 \prec 2 \prec \ldots \prec 2g-2 \prec 1 \prec 2\prec \ldots.\]
With this ordering, we observe that for a given $z$ and a given ordered pair $(x,y)$, all distinct, either $z$ separates $x$ and $y$ with respect to the cyclic order (that is, $x \prec z \prec y$) or it does not (so that $z \prec x \prec y \prec z$).

We can now describe the different possible arrangements of $b_i$ and $b_j$ on $S$.

\begin{defn}
Suppose that $i,j \in \ZZ_{2g-2}$ are such that the residue class of $j-i$ mod $2g-2$ is at most $g-2$.

If the surface $S$ is labeled as in case (1+2), then we say that the ordered pair $(i,j)$ is in the
\begin{itemize}
\item {\em one--sided regime} if neither $1$ nor $g$ separates $(i,j)$
\item {\em two--sided regime} if either $1$ or $g$ separates $(i,j)$.
\end{itemize}

If the surface $S$ is labeled as in case (3), then we say that the ordered pair $(i,j)$ is in the
\begin{itemize}
\item {\em one--sided} regime if no element of $\{2, g, 2g-2\}$ separates $(i,j)$
\item {\em two--sided} regime if exactly one element of $\{2, g, 2g-2\}$ separates $(i,j)$
\item {\em three--sided} regime if both $2$ and $2g-2$ separate $(i,j)$.
\end{itemize}
\end{defn}

Observe that in case (1+2), by our restrictions on $j-i$ we know that $g$ and $1$ cannot both separate $(i,j)$. Likewise, in case (3) it follows that $g$ cannot separate $(i,j)$ if either $(2g-2)$ or $2$ does.

The proof of Proposition \ref{prop:heuristic} when $(i,j)$ lies in the one--sided regime is quite straightforward:

\begin{lemma}\label{lem:bi_and_cij}
Suppose that the pair $(i,j)$ is in the one--sided regime. Then
\[ T({b_{j}})^{-1} \left( c_{(i, j)} \right) = T({c_{(i, j)}}) \left(b_j\right) \in \GA \cdot b_i.\]
Similarly, one has
\[T(b_i) \left( c_{(i,j)} \right) = T({c_{(i, j)}})^{-1} \left(b_i\right) \in \GA \cdot b_j.\]
\end{lemma}
\begin{proof}
Note that the equality in the statement is clear by inspection, and in fact $T(a)^{-1} (b)= T(b) (a)$ for any two simple closed curves $a$ and $b$ intersecting once on $S$.

In order to find an element of $\GA$ taking $b_i$ to $T({b_{j}})^{-1} (c_{(i, j)})$, we use the procedure outlined in \S\ref{sec:braids}. The proof is best understood by scrutinizing Figure \ref{fig:BH_oneside}, but for the convenience of the reader we trace its construction below.

\begin{figure}[ht]
\centering
\includegraphics[scale=.5]{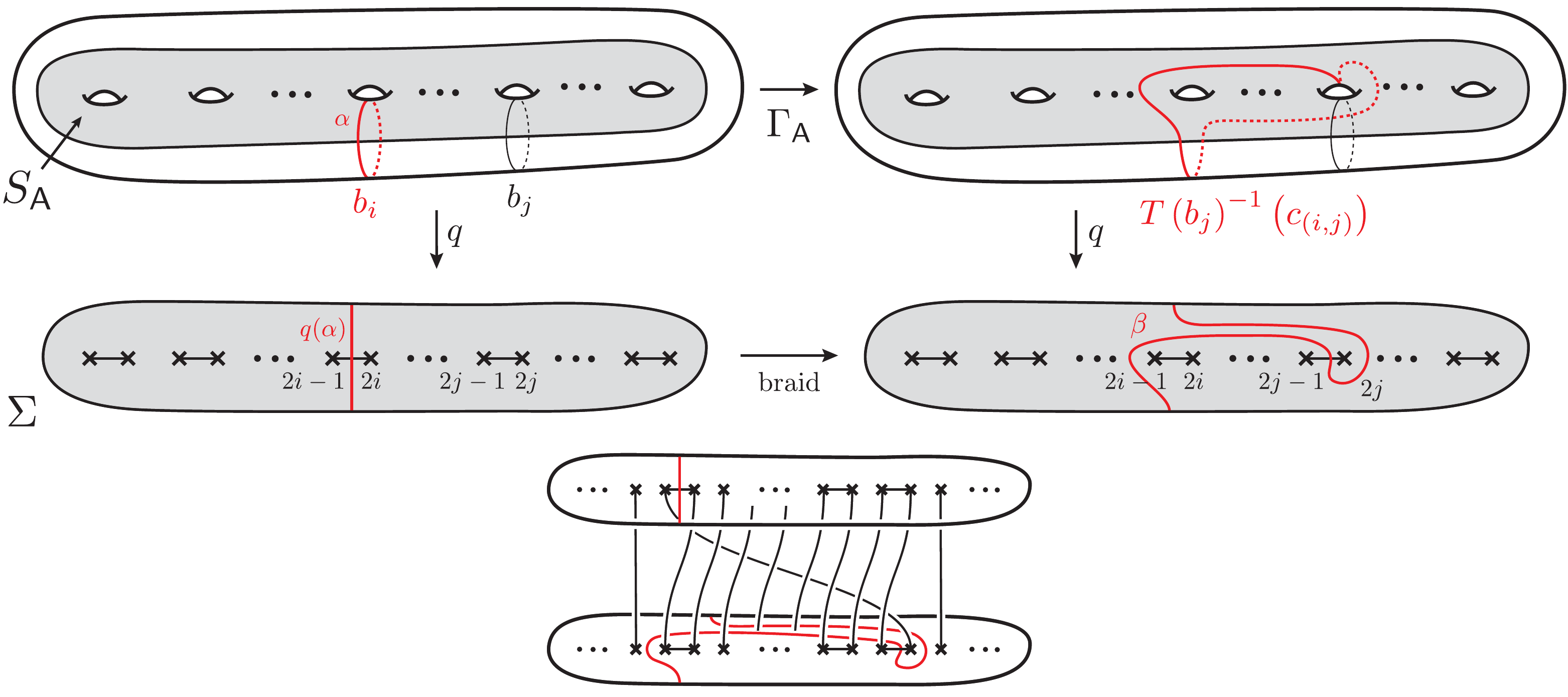}
\caption{A braid which takes $b_i \cap S_\A$ to $T({b_{j}})^{-1} \left( c_{(i, j)} \right) \cap S_\A$.}
\label{fig:BH_oneside}
\end{figure}

First, suppose that we are in case (1+2); then the intersection of $b_i$ with $S_\A$ is an arc $\alpha$ whose image $q(\alpha)$ under the hyperelliptic involution separates the first $(2i-1)$ branch points of $\Sigma$ from the others. By braiding the $(2i-1)^{\text{st}}$ branch point behind each of the $(2i)\ith$ to the $(2j)^{\text{th}}$ strands (and shifting each of $2i$ through $2j$ to the left by one), we can take $q(\alpha)$ to a new arc $\beta$. The lift of the corresponding braid under the Birman--Hilden isomorphism \eqref{eqn:GA_BH12} is an element of $\GA$, and can be seen to take $c$ to 
\[(c \setminus \alpha) \cup q^{-1}(\beta) = T({b_{j}})^{-1} \left( c_{(i, j)} \right).\]

In case (3), one must use one of the subsurfaces $S_{\A_m}$ and groups $\Gamma_{\A_m} = \SMod(S_{\A_m})$ (where the subsurface is determined by how $i$ and $j$ relate to $2$, $g$, and $2g-2$ in the cyclic order), but otherwise the procedure is completely analogous.

The second statement follows by braiding the $(2j)^{\text{th}}$ strand behind the $(2i-1)^{\text{st}}$ to the $(2j-1)^{\text{st}}$.
\end{proof}

In the one--sided regime, this lemma is enough to justify the heuristic, for we immediately note that
\[\begin{array}{cl}
 c_{(i, j)} & \in \langle \GA, T({b_j}) \rangle \cdot b_i \\
 b_j & \in \langle \GA, T({ c_{(i, j)}}) \rangle \cdot b_i \\
\end{array}\]
and hence by Fact \ref{lem:change_of_coords}, the twists on any two of $\{b_i, b_j, c_{(i,j)} \}$ together with $\GA$ are enough to recover the twist on the third.

When $(i,j)$ in the two--sided regime, the curve $c_{(i,j)}$ is no longer defined, and so the initial form of the heuristic makes no sense. However, we may still show that a similar statement still holds: from $\GA$, $b_i$, and an auxiliary curve $c$ one can obtain $b_j$ (and vice--versa). This case should be thought of as allowing us to ``pass around'' a single end of the surface $S$ when applying addition (or subtraction).

\begin{lemma}\label{lem:twosided}
Suppose that the curves of $S$ are labeled as in Figure \ref{fig:curvelabels} and $1 \prec i \prec g \prec j \prec 1$ (in the cyclic order) are so that $j - i \le g-2$. Then 
\[T\left({c_{(g-j+i,g)}}\right) T\left({b_g}\right) \left( c_{(g, j)} \right) \in \GA \cdot \, b_i.\]
In case (1+2), if $g \prec i \prec 1 \prec j \prec g$ and $j-i \le g-2$, then likewise
\[T\left(c_{(1, 2g-i+j-1)}\right) T\left(b_1\right) \left(c_{(1, j)} \right)
\in \GA \cdot \, b_i.\]
In case (3),
if $g \prec i \prec 2g-2$, then likewise
\[T\left( c_{(1,2g-i+1)} \right) T \left(b_{2g-2} \right) \left( c_{(1,3)} \right) \in \GA \cdot\,  b_i\]
and for $2 \prec j \prec g$, one has
\[T\left( c_{(2, j+1)} \right) T\left( b_2 \right) \left( c_{(2,j)} \right) \in \GA \cdot \, b_1.\]
\end{lemma}

Before proving the lemma, we note how it implies Proposition \ref{prop:heuristic} (the generalized heuristic) in any of the above scenarios. In the case in which one has the twists on $b_i$ and $c_{(g-j+i, g)}$ it follows from Lemmas \ref{lem:bi_and_cij} and \ref{lem:twosided}, respectively, that
\[\GA \cdot b_j \ni
T({b_g}) \left( c_{(g, j)} \right)
\in \langle \GA, T(c_{(g-j+i,g)}) \rangle \cdot b_i.\]
and therefore $b_j \in \langle \GA, T(c_{(g-j+i,g)}) \rangle \cdot b_i$.

To see that one can get the twist on $c_{(g-j+i, g)}$ from those on $b_i$ and $b_j$, observe that the $T(a)^{-1} (b)= T(b) (a)$ relation together with the commuting property of nonintersecting Dehn twists implies that
\begin{align}\label{eqn:twists1}
\begin{split}
T({c_{(g-j+i,g)}}) T({b_g}) \left( c_{(g, j)} \right)
& = T({c_{(g-j+i,g)}}) T(c_{(g, j)})^{-1} \left( b_g \right)\\
& =T(c_{(g, j)})^{-1} T({c_{(g-j+i,g)}}) \left( b_g \right)\\
&=T(c_{(g, j)})^{-1} T(b_g)^{-1} \left( c_{(g-j+i,g)} \right)\\
&= T(c_{(g, j)})^{-1} T(b_g)^{-1}T(c_{(g, j)}) \left( c_{(g-j+i,g)} \right)
\end{split}
\end{align}
and by Fact \ref{lem:change_of_coords}, one has
\begin{equation}\label{eqn:twists2}
T(c_{(g, j)})^{-1} T(b_g)^{-1}T(c_{(g, j)}) \left( c_{(g-j+i,g)} \right)
= T\Big( T(c_{(g, j)})^{-1} (b_g) \Big)^{-1} \left( c_{(g-j+i,g)} \right).
\end{equation}
Now by Lemma \ref{lem:bi_and_cij}, we have that
$T({c_{(g, j)}})^{-1} \left(b_g\right) \in \GA \cdot b_j$
and therefore
\begin{equation}\label{eqn:twists3}
T\Big( T(c_{(g, j)})^{-1} (b_g) \Big)^{-1} \in \langle \GA, T(b_j) \rangle.
\end{equation}
Putting together \eqref{eqn:twists1}, \eqref{eqn:twists2}, and \eqref{eqn:twists3} with Lemma \ref{lem:twosided}, we have that
\[c_{(g-j+i,g)} \in  \langle \GA, T(b_j) \rangle \cdot b_i .\]
A similar analysis may be performed for each of the other statements.

\begin{figure}[ht]
\centering
\includegraphics[scale=.45]{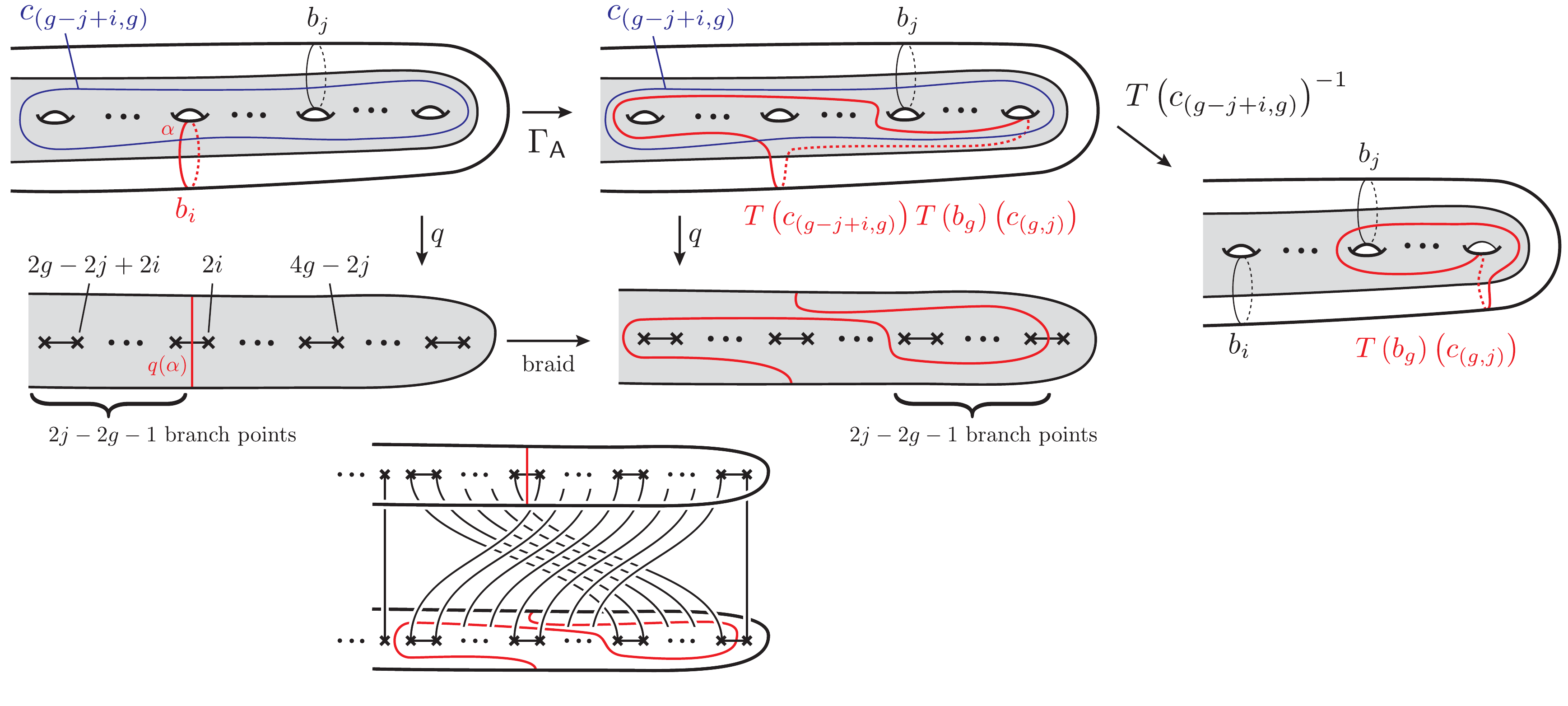}
\caption{The construction of a braid demonstrating Lemma \ref{lem:twosided}.}
\label{fig:BH_twoside}
\end{figure}

\begin{proof}
The proofs of all of the statements are exactly the same up to reindexing (and when in case (3), the use of the appropriate chain $\A_m$), 
so we will assume that we are in the case when $1 \prec  i \prec g \prec j \prec 1$ and $j - i \le g-2$ and leave the remaining cases to the scrupulous reader.

In order to find an element of $\GA$ sending $b_i$ to 
\[T\left({c_{(g-(j-i),g)}}\right) T\left({b_g}\right) \left( c_{(g, j)} \right),\]
we will employ the same strategy as in Lemma \ref{lem:bi_and_cij}. Intersect $b_i$ either with the surface $S_\A$ or $S_{\A_2}$ (when in cases (1+2) and (3), respectively) to get an arc $\alpha$. Upon passing to the quotient $\Sigma = S_\A / \iota$, $\alpha$ becomes an arc which separates the last $2g-2i+1$ branch points from the rest.

The lift of the braid which takes the $(2i)^\text{th}$ through $(2g-1)^{\text{st}}$ points in front of the $(2g-2j+2i-1)^\text{st}$ through $(2i-1)^\text{st}$ points is then our desired element of $\GA$. A schematic of this construction is presented in Figure \ref{fig:BH_twoside}.
\end{proof}

In order to deal with the left--hand side of a surface labeled as in case (3), we require a slightly generalized version of Lemma \ref{lem:twosided}.

\begin{figure}[hb]
\centering
\includegraphics[scale=.5]{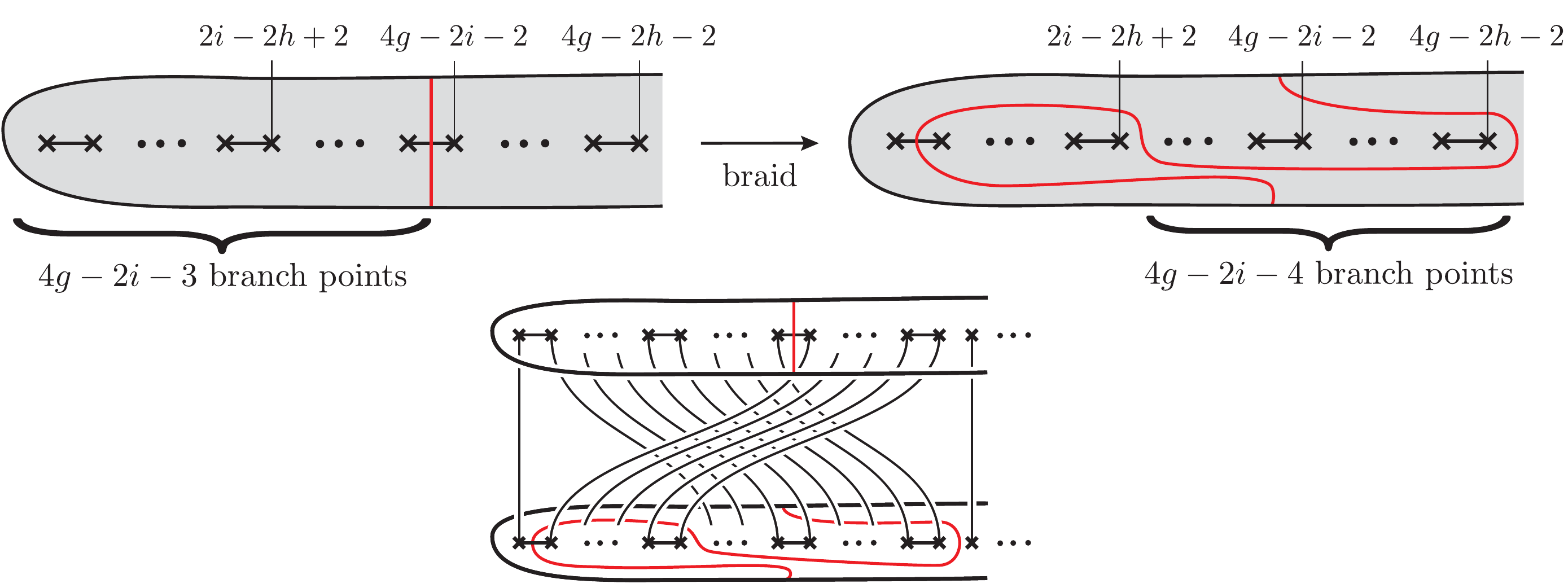}
\caption{A generalization of the braid appearing in Lemma \ref{lem:twosided}.}
\label{fig:BH_twoside_gen}
\end{figure}

\begin{lemma}\label{lem:twosided_generalized}
Suppose that the curves of $S$ are labeled as in case (3) and $g \ge h < i < 2g-2$. Then
\[T\left( c_{(1,2g-h)} \right) T \left(b_{2g-2} \right) \left( c_{(1,i-h+2)} \right) \in \GA \cdot b_i \]
\end{lemma}

The construction of the desired element follows as above, braiding the $2^\text{nd}$ through $(4g-2i-3)^\text{rd}$ strands of $S_{\A_2} / \iota$ behind the $(4g-2i-2)^\text{nd}$ through $(4g-2h-2)^\text{nd}$. We depict the corresponding braid in Figure \ref{fig:BH_twoside_gen} by way of proof.

Finally, we record below the last tool we need to prove Proposition \ref{prop:heuristic} when the surface is labeled as in case (3) and the pair $(i,j)$ is in the three--sided regime. In this scenario, one needs to be able to ``pass around'' both the $a_1$ and $a_2$ handles on the left--hand side of the surface.

\begin{lemma}\label{lem:LHS_case3}
If the curves of $S$ are labeled as in case (3) and $2 \le j \le g$, then 
\[T(b_2) \left(c_{(2, j)}\right) \in \GA \cdot \, T(b_{2g-2} )\left(c_{(1,j+2)}\right).\]
\end{lemma}
\begin{proof}
To construct the desired element of $\GA$, we make use of all three hyperelliptic subsurfaces $S_{\A_m}$ and their respective symmetric mapping class groups. An overview of the construction is presented in Figure \ref{fig:BH_threeside}.

\begin{figure}[ht]
\centering
\includegraphics[scale=.5]{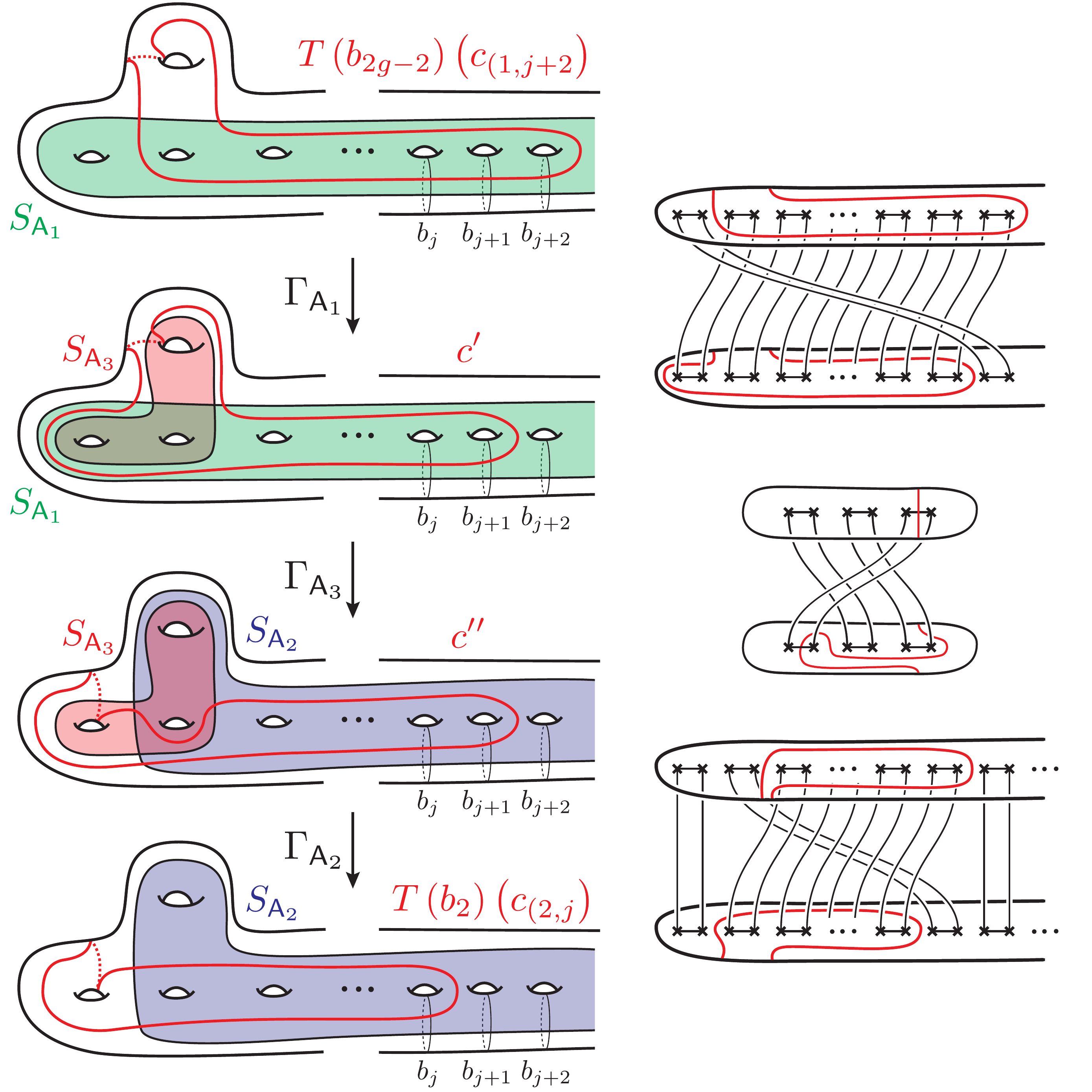}
\caption{A sequence of braids which allows us to ``pass around'' the left hand side of a surface labeled as in case (3).}
\label{fig:BH_threeside}
\end{figure}

For ease of notation, throughout this proof we will write $c$ for $T(b_{2g-2} )\left(c_{(1,j+2)} \right)$. 

We begin by intersecting $c$ with $S_{\A_1}$; call this arc $\alpha$. Its quotient in $S_{\A_1} / \iota$ is an arc which separates off the third through $(4j+4)^\text{th}$ branch points. By braiding the first and second branch points in front of these, we arrive at an arc separating off the first through $(4j+2)^\text{nd}$ branch points, whose lift (with the same endpoints as $\alpha$) we will denote by $\alpha'$. The Birman--Hilden theory then implies that there exists some element of $\Gamma_{\A_1} = \SMod(S_{\A_1})$ which takes $c$ to
\[c':= c \setminus \alpha \cup \alpha'.\]
See Figure \ref{fig:BH_threeside}.

Now intersect $c'$ with $S_{\A_3}$. Upon quotienting by the appropriate hyperelliptic involution, this yields an arc which separates the last branch point from the other five. Braid the fifth and sixth strands of $S_{\A_3} / \iota$ in front of the other four and lift back up to $S_{\A_3}$; as before, Birman--Hilden implies that the resulting curve $c''$ is in the $\Gamma_{\A_3}$ orbit of $c'$.

Finally, consider the intersection of $c''$ with $S_{\A_2}$. The resulting arc on the quotient surface $S_{\A_2} / \iota$ separates the fifth through $(2j)\ith$ branch points from the others, so by braiding the third and fourth branch points behind these and lifting back up, one arrives at a curve $c''' \in \Gamma_{\A_2} \cdot c''$ whose intersection with $S_{\A_2}$ is an arc which encircles the handles corresponding to $a_3$ through $a_j$.

Tracing through this construction, one observes that the resulting curve $c'''$ is isotopic to $T(b_2) \left( c_{(2,j)} \right)$, and
\[c''' \in \Gamma_{\A_2} \cdot c'' \subseteq \Gamma_{\A_2} \Gamma_{\A_3} \cdot c'
\subseteq \Gamma_{\A_2} \Gamma_{\A_3} \Gamma_{\A_1} \cdot c\]
thus concluding the proof of the lemma.
\end{proof}

We may now finish the proof of (the refined form of) our motivating heuristic.

\begin{proof}[Proof of Proposition \ref{prop:heuristic}]
This proof naturally breaks into multiple cases, depending on the labeling scheme of the surface $S$ and the sided-ness of the pair $(i,j)$. The one-- and two--sided cases have already been justified above (see the discussions after Lemmas \ref{lem:bi_and_cij} and \ref{lem:twosided}, respectively, together with Lemma \ref{lem:GA_action_cij}) and so we will not reproduce those arguments here. 

That leaves the three--sided case to consider. To that end, we may suppose that the surface $S$ is labeled as in case (3) and that $2\prec j \prec g \prec i \prec 2g-2$, so 
\[j-i + (2g-2) \le g-2.\]
In particular, this implies that $i - j \ge g$. Therefore, setting $h=i-j$ in Lemma \ref{lem:twosided_generalized} yields
\begin{equation}\label{eqn:heur1}
T\left( c_{(1,2g-i+j)} \right) T \left(b_{2g-2} \right) \left( c_{(1,j+2)} \right) \in \GA \cdot\, b_i .
\end{equation}
Now we note that by Lemma \ref{lem:LHS_case3} we have that
\begin{equation}\label{eqn:heur2}
T \left(b_{2g-2} \right) \left( c_{(1,j+2)} \right) \in \GA \cdot \, T\left( b_2 \right) \left( c_{(2,j)} \right)
\end{equation}
and by Lemma \ref{lem:bi_and_cij},
\begin{equation}\label{eqn:heur3}
T\left( b_2 \right) \left( c_{(2,j)} \right) \in \GA \cdot\, b_j.
\end{equation}
Therefore, by combining \eqref{eqn:heur1}, \eqref{eqn:heur2}, and \eqref{eqn:heur3} it follows that
\[b_j \in \langle \GA, T\left( c_{(1,2g-i+j)} \right) \rangle \cdot b_i.\]
Now by definition we have that $c_{(1,2g-i+j)} = c_{(i-j, 2g-2)}$, and so by Lemma \ref{lem:GA_action_cij} we know that
\[c_{(1,2g-i+j)} \in \GA \cdot c_{(k,\ell)}.\]
Therefore we may conclude that
\[b_j \in \langle \GA, T(c_{(k, \ell)}) \rangle \cdot b_i\]
as desired.

In order to prove the second statement, we apply the same manipulations appearing in \eqref{eqn:twists1} and \eqref{eqn:twists2} to \eqref{eqn:heur1} to deduce that 
\begin{equation}\label{eqn:heur4}
T\left( c_{(1,2g-i+j)} \right) T \left(b_{2g-2} \right) \left( c_{(1,j+2)} \right) 
=
T\Big( T(c_{(1, j+2)})^{-1} (b_{2g-2}) \Big)^{-1} \left( c_{(1,2g-i+j)} \right).
\end{equation}
Applying the $T(a)^{-1} (b)= T(b) (a)$  relation once more, we have that
\[T(c_{(1, j+2)})^{-1} (b_{2g-2}) = T(b_{2g-2}) \left( c_{(1,j+2)} \right)
\in \GA \cdot \,T\left( b_2 \right) \left( c_{(2,j)} \right)
\subseteq \GA \cdot \left( \GA \cdot\, b_j \right)\]
where the second and third inclusions follow from Lemmas \ref{lem:LHS_case3} and \ref{lem:bi_and_cij}, respectively. Therefore
\[T\Big( T(c_{(1, j+2)})^{-1} (b_{2g-2}) \Big) \in \langle \GA, T(b_j) \rangle\]
and so by \eqref{eqn:heur4} and Fact \ref{lem:change_of_coords}, we have that
\[c_{(1, 2g-i+j)} = c_{(i-j, 2g-2)} \in \langle \GA, T(b_j) \rangle \cdot \, b_i.\]
A final application of Lemma \ref{lem:GA_action_cij} finishes the proof.
\end{proof}

With this general form of the heuristic, it is now very simple to prove Proposition \ref{prop:curve_addsub}; indeed, the entire argument appears in the sketch in \S\ref{sec:Euclidean}. For completeness, we reproduce it below.

\begin{proof}[Proof of Proposition \ref{prop:curve_addsub}]
By Fact \ref{lem:change_of_coords}, in order to show that
\[T(b_{i+2x}) \in G:= \langle T\left(b_{i}\right), T\left(b_{i+x}\right), \GA \rangle,\]
one need only find an element $g \in G$ which takes $b_{i+x}$ to $b_{i+2x}$.

First, note that since $x \le g-2$, we know that there is some pair $(k, \ell)$ with $2 \le k < \ell \le g$ (i.e., lying in the one--sided regime) such that $\ell - k  =x$. Therefore we may apply Proposition \ref{prop:heuristic} and deduce that there is some element of $\langle \GA, T(b_{i+x}) \rangle$ taking $b_i$ to $c_{(k, \ell)}$, and therefore $T\left(c_{(k, \ell)} \right) \in G$.

A second application of Proposition \ref{prop:heuristic} yields an element of $\langle \GA, T(c_{(k, \ell)}) \rangle \subseteq G$ which takes $b_{i+x}$ to $b_{i+2x}$, thereby proving the first statement (addition) of the Proposition. The proof of the second statement (subtraction) is completely analogous.
\end{proof}

\end{document}